\documentclass{amsart}
\usepackage{amsmath, amsthm, amssymb}
\usepackage{verbatim}
\usepackage{tikz}
\usetikzlibrary{matrix}

\usepackage[mathscr]{eucal} 
\usepackage{mathrsfs} 
\usepackage{cancel}

\usepackage{hyperref}

\def\XXint#1#2#3{{\setbox0=\hbox{$#1{#2#3}{\int}$}
     \vcenter{\hbox{$#2#3$}}\kern-.5\wd0}}

\renewcommand{\d}{\partial}
\newcommand{\ddbar}{\sqrt{-1}\d\overline{\d}}
\newcommand{\dbar}{\overline{\d}}
\newcommand{\ii}{\sqrt{-1}}
\newcommand{\wt}[1]{\widetilde{#1}}
\newcommand{\lc}{\left<}
\newcommand{\rc}{\right>}

\newcommand{\wh}[1]{\widehat{#1}}

\newcommand{\eps}{\epsilon}
\newcommand{\veps}{\varepsilon}
\newcommand{\vphi}{\varphi}

\newcommand{\al}{\alpha} 
\newcommand{\ze}{\zeta} 
\newcommand{\be}{\beta}
\newcommand{\ga}{\gamma}
\newcommand{\de}{\delta}
\newcommand{\la}{\lambda}
\newcommand{\om}{\omega}
\newcommand{\te}{\theta}

\newcommand{\sig}{\sigma}

\newcommand{\ka}{\kappa}
\newcommand{\io}{\iota}

\newcommand{\tc}{{\tt c}}

\newcommand{\Ga}{\Gamma}
\newcommand{\Om}{\Omega}
\newcommand{\De}{\Delta}

\newcommand{\na}{\nabla}

\newcommand{\cE}{\mathcal{E}}

\newcommand{\cF}{\mathcal{F}}

\newcommand{\cL}{\mathcal{L}}

\newcommand{\supp}{\mbox{supp }}

\newcommand{\bD}{\mathbb{D}}

\newcommand{\bR}{\mathbb{R}}

\newcommand{\bC}{\mathbb{C}}

\newcommand{\cali}[1]{\mathscr{#1}}

\newcommand{\Ec}{\cali{E}}

\newcommand{\Kc}{\cali{K}}
\newcommand{\Lc}{\cali{L}}

\newtheorem{thm}{Theorem}
\newtheorem{prop}[thm]{Proposition}
\newtheorem{lem}[thm]{Lemma}
\newtheorem{cor}[thm]{Corollary}

\theoremstyle{definition}
\newtheorem{defn}[thm]{Definition}
\newtheorem{remark}[thm]{Remark}
\newtheorem{claim}[thm]{Claim}

\newtheorem{expl}[thm]{Example}

\numberwithin{thm}{section}
\numberwithin{equation}{section}

\renewcommand{\[}{\begin{equation}}
\renewcommand{\]}{\end{equation}}

\newcommand{\rv}{\;\rvert\;}

\newcommand{\wed}{\wedge}

\usepackage{scalerel}[2014/03/10]
\usepackage[usestackEOL]{stackengine}
\def\intavg{\,\ThisStyle{\ensurestackMath{%
    \stackinset{c}{0\LMpt}{c}{0\LMpt}{\SavedStyle-}{\SavedStyle\phantom{\int}}}%
    \setbox0=\hbox{$\SavedStyle\int\,$}\kern-\wd0}\int}

\title{Fine properties of  functions in complex Sobolev spaces}

 \author{Ngoc Cuong Nguyen} 
\address{Department of Mathematical Sciences, KAIST, 291 Daehak-ro, Yuseong-gu, Daejeon 34141, South Korea}
\email{cuongnn@kaist.ac.kr}
\date{\today}

\begin{document}

\maketitle
\begin{abstract} We study comprehensively local properties of functions in complex Sobolev spaces on a bounded open subset of $\bC^n$. The main tool is the corresponding functional capacity for  the space which is inspired by the global one due to Vigny \cite{Vigny}. An inequality between this capacity and the Bedford-Taylor capacity for plurisubharmonic functions is proved, which is sharp as far as the exponents are concerned. Moreover, it is shown that the functional capacity is a Choquet capacity. The Alexander-Taylor type inequality for the capacity is also proved.  This allows us to strengthen the results in  the works of Dinh, Marinescu and Vu \cite{DMV}, Vigny and Vu \cite{VV24}. Lastly, the Moser-Trudinger type inequality in this space is characterized by the volume-capacity inequality.
\end{abstract}

\tableofcontents

\section{Introduction}

Let $(X,\om)$ be a compact K\"ahler manifold.  A natural subspace of the Sobolev space $W^{1,2}(X)$, called the complex Sobolev space $W^*$, was introduced by Dinh and Sibony \cite{DS} in their work on complex dynamics. 
This new space turned out to be a good complex version of the former one as it takes into account the complex structure of manifolds. Many important families of functions which carry useful properties of underlying manifolds belong to the space. Among them are Lipschitz functions, bounded quasi plurisubharmonic (psh) functions or more general bounded delta quasi-psh function studied in  \cite{CW05}. Basic functional aspects of this space were studied by Vigny \cite{Vigny} where he showed that it is a Banach space with a corresponding norm $\| \cdot\|_*$, this norm  is stronger than the Sobolev norm. However, it is not reflexive and smooth functions are not dense in $W^*$ with respect to the strong topology of that norm.  He also defined a functional capacity for this space and showed that it is a capacity in sense of Choquet. Moreover, the functional capacity is qualitative comparable to the relative capacity of Bedford and Taylor in global pluripotential theory defined in  \cite{ko03}. 

Thanks to flexible properties this space has found many of applications in complex dynamic in high dimensions, complex Monge-Amp\`ere equations and other areas, see e.g., \cite{BiD23}, \cite{DKN}, \cite{DKW}, \cite{DNV25}, \cite{Vig15}, \cite{Vu20, Vu24, Vu26} and \cite{WZ}. Also a higher complex Sobolev space is proposed in \cite{DoN25}. 

An element in a Sobolev space is defined almost everywhere up to a set of Lebesgue measure zero which is unlike psh or quasi-psh functions. Its fine properties with respect to the Sobolev capacity are classical ones \cite{EG92}, \cite{KLV-book}. However, this capacity is not good enough to characterize pluripolar sets, therefore it is not suitable for studying functions in $W^*$. One needs to work with the corresponding capacity of $W^*$ which dominates the Sobolev one. However, the new capacity in complex Sobolev spaces is harder to analyze. A crucial difference is that we can no longer rely on maximal function techniques, which are very powerful, in these spaces (see, e.g \cite{KLV-book}). 

Recently, it has been showed in \cite{DMV}, \cite{Vigny} and \cite{VV24} that pluripotential theory is a suitable and powerful tool for studying $W^*$. As a result they made major  progresses on study local properties of functions in such spaces.  

We continue to study local complex Sobolev spaces which has been considered partially in \cite{DMV} and  \cite{Vigny}. Our goal is to develop further this approach by focusing on new and optimal inequalities between the functional capacity and classical capacities of Bedford and Taylor \cite{BT82} and of Alexander and Taylor \cite{AT84} in pluripotential theory.

Let $\Om \subset \bC^n$ be a bounded domain  and $\om = dd^c |z|^2$  the standard K\"ahler form in $\bC^n$.  Denote by $W^{1,2}(\Om, \bR)$ the usual Sobolev space. For  $f\in W^{1,2}(\Om, \bR)$ we define $\Ga_f$  to be the set of all positive closed $(1,1)$-current $T$ satisfying
\[\label{eq:test-currents}\notag
	df \wed d^c f \leq T \quad\text{weakly in } \Om.
\]
Here we use the normalization 
$$d^c = \frac{\ii}{2\pi} (\dbar -\d), \quad dd^c = \frac{\ii}{\pi}\d\dbar.
$$
Consider the subspace
\[ \label{eq:sobolev-space} \notag
W^*(\Om)= \left\{f \in W^{1,2}(\Om,\bR):   \text{there exists $T\in \Ga_f$ with } \|T\|_\Om <+\infty
\right\},
\]
where 
$
	 \|T \|_\Om := \int_\Om T \wed \om^{n-1}.
$
For $f \in W^*(\Om)$ one defines the $W^*$-norm
\[\label{eq:norm} \notag
	\|f\|_*^2 = \| f\|_{L^2(\Om)}^2 + \inf_{T\in \Ga_f} \|T\|_{\Om}.
\]
We then reprove basic functional results obtained in \cite{Vigny}. Namely, $(W^*(\Om), \|\cdot\|_*)$ is a Banach space (Theorem~\ref{thm:star-norm}) and  continuous functions are not dense in $W^*(\Om)$ with respect to strong topology of the norm $\|\cdot\|_*$ (Proposition~\ref{prop:density}). If $n=1$, then  $W^*(\Om) = W^{1,2}(\Om)$, otherwise $W^*(\Om)$ is not reflexive (Corollary~\ref{cor:reflexive}). It should be pointed out that there are some similarity of these properties with the ones of the space of delta psh functions studied by Cegrell and Wiklund \cite{CW05}.
 
Inspired by the functional capacity in \cite{Vigny} we define for a Borel subset $E\subset \Om$ the local capacity $\tc(E):= \tc(E,\Om)$  as follows:
\[\label{eq:intro-cap-defn}\notag
	{\tt c}(E) = \inf 
	\left\{ \| v \|_*^2 \;\rvert\;  v\in \Kc(E)
	\right\},
\] 
where
\[\label{eq:intro-ke-cap}\notag
 	\Kc(E) = \left\{ v \in W^*(\Om) \rv \{v \leq  -1\}^o \supset E \text{ and } v\leq 0 \right\}.
\]
Here $ \{v \leq  -1\}^o \supset E$ means $v\leq -1$ a.e in a neighborhood of $E$. 

The functional capacity in the Sobolev spaces is very well-understood which is also an effective tool to study fine properties of its element, see e.g. \cite[Chapters~4.7-4.8]{EG92}. In contrast, the use of $\tc(\cdot)$ for the complex Sobolev space is not effective so far. 
Our goal is to investigate systematically local properties of functions in $W^*(\Om)$ by using this capacity. This is a different perspective compared to \cite{Vigny} and \cite{DMV} who used the relative capacity for psh functions due to Bedford and Taylor \cite{BT82}. Namely, for a Borel set $E\subset \Om$, 
\[\label{eq:intro-BT-cap}\notag
 cap(E,\Om) = \sup\left\{ \int_E (dd^c u)^n :  u\in PSH(\Om), \;-1 \leq u \leq 0\right\}.
\]
Our first main result explains the reason why $cap(\cdot, \Om)$ could be effectively used in previous works. Roughly speaking they are equivalent to each other. 

\begin{thm} \label{thm:intro-cap-comp} 
Let $\Om \subset \subset \bC^n$ be a strictly pseudoconvex domain. 
\begin{itemize}
\item[(a)] There exists a constant $A>0$ such that for every Borel set $E\subset \Om$,
$$
	\tc(E) \leq A [cap(E,\Om)]^\frac{1}{n}.
$$
\item[(b)]
Assume $D \subset \subset \Om$ be a subdomain. There exists a constant $A'$ such that for every Borel set $E\subset D$,
$$
	\frac{1}{A'} cap(E,\Om) \leq \tc (E).
$$
\end{itemize}
Moreover, the above inequalities are sharp as far as the exponents are concerned.
\end{thm}
The statements of the theorem consist of Lemmas~\ref{lem:c-cap}, ~\ref{lem:cap-c} and Remarks~\ref{rmk:inner-reg},~\ref{rmk:sharp-exp}. 
On compact K\"ahler manifolds both inequalities  was obtained in \cite[Proposition~5.1]{DKN} where $\tc(E)$ is defined by another equivalent norm with $W^*$-norm, but the optimality was unknown. 
 A qualitative version  of the first inequality on such manifolds was obtained early in \cite{Vigny} and the second inequality was also proved there without showing that the exponents are sharp.

In the local setting the proofs are different. The key idea in the proof of the first inequality is making use of Cegrell's inequality \cite{Ce04} and the one for the second inequality comes from the comparison principle. On the other hand, the optimality of the exponents is extracted from the Alexander-Taylor type inequality which is proved in Lemma~\ref{lem:AT-type}. The latter inequality is of independent interest. One of its applications is the existence of an entire psh subextension for a given $f \in W^*(B(0,1))$ in Theorem~\ref{thm:subextension}. This psh subextension result is an important ingredient to derive the property that  the set Lebesgue points of such functions are of capacity zero. This result is due to Vigny and Vu \cite{VV24} where the bounded case was proved earlier in \cite{Vigny}. We improve the statement of that result and its  proof is also simplified in Theorem~\ref{thm:VV}.

An immediate consequence of Theorem~\ref{thm:intro-cap-comp} and \cite{BT82} is that a Borel set $P\subset \Om$ is pluripolar if and only if $\tc(P)=0$. Provided the equivalence of capacities  the results in \cite{DMV} can be restated in terms of $\tc (\cdot)$ where they are obtained previously for $cap(\cdot, \Om)$. However, as emphasized above we would take another route by using advantages of the functional capacity $\tc(\cdot)$ (see e.g., Lemma~\ref{lem:outer-reg}, Proposition~\ref{prop:ae-qe}, Lemma~\ref{lem:cap-sublevel-set}) and hence providing alternative  proofs for the ones obtained in \cite{DMV}. In particular,  using the ideas in \cite{DMV} and \cite{Vigny},  we  simplify  the proof of quasi-continuous representative result proved earlier in \cite[Theorem~2.10]{DMV} (see also \cite[Theorem~22]{Vigny}).

\begin{cor} \label{cor:intro-quasi-mod} Let  $f\in W^*(\Om)$. There exists a quasi-continuous modification $\wt f$ such that $\wt f = f$ a.e. Moreover, if $g$ is another quasi-continuous modification of $f$, then $g =\wt f$ outside a pluripolar set.
\end{cor}

The most technical part of the proof  lies in delicate integral estimates. Along the way  we refine the ones obtained in \cite{DMV}. It turns out in  Theorem~\ref{thm:VV}  that the precise representative, in the sense of  Sobolev spaces \cite[Chapter~4]{EG92},
\[\label{eq:rep} f^*(x) \equiv
\begin{cases}
\lim_{r\to 0} \intavg_{B(x,r)} f \; d y &\quad \text{if this limit exits,} \\
=0 &\quad \text{otherwise},
\end{cases}	
\] 
is a quasi-continuous modification, where $d y$ is the Lebesgue measure on $\bC^n = \bR^{2n}$.

Thanks to the quasi-continuity result 
 we can define the integration of a function $g\in W^*(\Om)$ against  Monge-Amp\`ere measures associated with a bounded psh function (Definition~\ref{defn:L2-ma}). This result is pushed much further in Section~9 where we show that $W^*(\Om)$ is  naturally embedded  into a Dirichlet space associated with a closed positive  $(n-1,n-1)$ current $
\te = [dd^c (\phi + \de |z|^2)]^{n-1}$ where $\phi \in PSH \cap L^\infty(\Om)$ and $\de>0$ are given.
Therefore, the integral  $\int_K df \wed d^c g\wed \te$ for $f,g\in W^*(\Om)$ is well-defined on every compact set $K\subset \Om$. This aspect has been considered recently by Do, Nguyen and Vu \cite{DNV25}.

The second main result is as follows. 
\begin{thm}\label{thm:intro-choquet-c} Let $\Om$ be a bounded open set in $\bC^n$. The set function on Borel subsets $E\mapsto \tc (E)= \tc (E,\Om)$ is a Choquet capacity.
\end{thm}

This result is the local analogue of \cite{Vigny}. On compact K\"ahler manifolds the proof of  \cite[Theorem~30]{Vigny} employed the capacity notion in the Dirichlet spaces \cite[Lemma~23]{Vigny} and the characterization of pluripolar sets via a family of capacity associated with closed positive current \cite{FO84}.  Here we use a similar strategy. Notably, we are able to give a simpler proof of the characterization in Corollary~\ref{cor:FO-polar} by using only pluripotential theory. 

Another remarkable property of functions in the complex Sobolev space is the uniformly exponential integrability due to Dinh, Marinescu and Vu  \cite{DMV}, which is very close to the one of psh functions.  Namely, 
there exist positive constants $\al$ and $A$  such that for every $f \in W^*(B(0,1))$ with $\|f\|_*\leq 1$, 
\[\label{eq:DMV-intro}
	\int_{\bar B(0,\frac{1}{8})} e^{\al |f|^2} d x \leq A.
\]
This result should be compared with the well-known exponential integrability of psh functions due to H\"ormander \cite[Proposition~4.2.9]{Ho07} and Skoda \cite{Sk72}. In the proof of  \cite[Theorem~1.2]{DMV}  the authors used the induction argument in dimension and slicing theory for positive currents. We provide a new and simpler one in Proposition~\ref{prop:DMV} by reducing the inequality to one dimensional case. In this way, it is enough  to  work with smooth functions after taking  the standard convolution with smooth kernels. The constants can be explicitly computed though they  are suboptimal.  Additionally, we point out in Remark~\ref{rmk:sharp-DMV} that the exponent 2 of $|f|$ in \eqref{eq:DMV-intro} is the optimal for all $n\geq 1$, which was only known for $n=1$ in \cite{DMV}. 

The third main result fully characterizes the inequality \eqref{eq:DMV-intro} by the volume-capacity inequality of sublevel sets.
 
\begin{thm} \label{prop:equiv-est-intro} Let $K\subset\subset \Om$ be a compact subset.  The following statements hold and they are equivalent to each other.
\begin{itemize}
\item[(a)] 
There exist constants $A_1>0$ and $\al>0$ depending only on $K, \Om$ such that for every  Borel  set $E\subset K$ $$V_{2n}(E) \leq A_1 e^{-\al/\tc (E)}.$$

\item[(b)] There exist uniform constants $A_1>0$ and $\al>0$ depending only on $K, \Om$ such that for every  $f\in W^*(\Om)$ and $f\leq 0$,
 $$V_{2n}(\{f< -1\} \cap K) \leq A_1 e^{\frac{-\al}{\|f\|_*^{2}}}.$$
\item
[(c)]   There exist uniform constants $A_1>0$ and $\al>0$ depending only on $K, \Om$ such that for every $f\in W^*(\Om)$ and $f\leq 0$ whose norm $\|f\|_* \leq 
\ka$,
$$V_{2n} (\{f < -s\} \cap K) \leq A_1 e^{-\al s^2/\ka^2}, \quad \forall s>0.$$

\item 
[(d)]  There exist uniform constants $A_1>0$ and $\al>0$ depending only on $K, \Om$ such that for every $f\in W^*(\Om)$ and $f\leq 0$ with $\|f\|_* \leq 1$,
$$V_{2n} (\{f < -s\} \cap K) \leq A_1 e^{-\al s^2}, \quad \forall s>0.$$
\end{itemize}
\end{thm}

The characterization is partially inspired by the work of \AA hag, Cegrell, Ko\l odziej, Pham and Zeriahi \cite{ACKPZ09}. Lastly, we point out in Remark~\ref{rmk:DMV-holder} that the inequalities holds for a very large family  of Monge-Amp\`ere measures associated with  H\"older continuous psh functions.

\bigskip

{\em Acknowledgement.} I would like to thank G. Marinescu and D.-V. Vu for a kind invitation to visit University of Cologne in September 2025 and to whom I had many helpful discussions on the results of \cite{DMV}.  I am grateful to S. Ko\l odziej and D.-V. Vu for useful comments on the draft of the paper. It was written while the author  visited the Center for Complex Geometry (Daejeon). He would like to thank Jun-Muk Hwang and Yongnam Lee for their kind support and exceptional hospitality. He is also grateful to the institution for providing perfect working conditions.

\section{Definitions and elementary properties}

In this section we set out the notations and study basic functional properties.

\subsection{Complex Sobolev spaces}

Let $\Om$ be a bounded open subset in $\bC^n$ and $$\om = dd^c |z|^2$$ is the standard K\"ahler form in $\bC^n$. For $f\in W^{1,2}(\Om, \bR)$ let $\Ga_f$ denote the set of all positive closed $(1,1)$-currents $T$ satisfying
\[\label{eq:test-currents}
	df \wed d^c f \leq T \quad\text{weakly in } \Om.
\]
(This set $\Ga_f$ could be empty.) 
The total mass of a positive $(1,1)$-current $T$ on $\Om$ is given by
\[\notag
	 \|T \|_\Om := \int_\Om T \wed \om^{n-1}.
\]
Recall that we define the subspace 
\[ \label{eq:sobolev-space}
W^*(\Om)= \left\{f \in W^{1,2}(\Om,\bR):   \text{there exists $T\in \Ga_f$ with } \|T\|_\Om <+\infty
\right\}
\]
and  the $W^*$-norm for $f \in W^*(\Om)$ is given by
\[\label{eq:norm}
	\|f\|_*^2 = \| f\|_{L^2(\Om)}^2 + \inf_{T\in \Ga_f} \|T\|_{\Om}.
\]
At this point it is a quantity associated with  $f$.  After  studying its basic properties we shall  prove that it is indeed a norm in Theorem~\ref{thm:star-norm}.

By the weak compactness of closed positive $(1,1)$-currents with finite mass in $\Om$, there is a closed positive $(1,1)$-current $T_f$ such that 
\[ \label{eq:min-current}
	df \wed d^c f \leq T_f \quad\mbox{and}\quad \inf_{T \in \Ga_f} \|T\|_\Om = \|T_f\|_\Om.
\] 

\begin{defn} \label{defn:w-current} A closed positive $(1,1)$-current $T$ satisfying \eqref{eq:min-current} is called a minimum current of $f$ or the current realizing the $W^*$-norm of $f$.
\end{defn}

The Sobolev space with zero boundary value is denoted by  $W^{1,2}_0(\Om)$ and we also consider
\[\notag
	W_0^*(\Om) = W^*(\Om) \cap W^{1,2}_0(\Om).
\]
Note from \cite[Lemma~7.6]{GT} that if $f,g \in W^{1,2}(\Om)$, then for a.e. points 
\[\label{eq:grad-max}
	 \na \max\{f, g\} = \begin{cases}
	\na f \quad &\mbox{on } \{f> g\}, \\
	\na g \quad &\mbox{ on } \{ f \leq g \},
	\end{cases}
\]
where $\na f$ denotes the gradient of $f$ in $\bR^{2n}$. 
In the special case, $g=0$, we write $$f^+ = \max\{f,0\} \quad\text{and}\quad f^- = -\min\{f, 0\}.$$
A basic inequality is that for every $s, t\in \bR$ and $f, g\in W^{1,2}(\Om)$,  
$$d(sf+tg) \wed d^c (sf+tg) \geq 0$$ 
as a positive $(1,1)$-form whose coefficients belong to $L^1(\Om)$. Equivalently, we have
\[\label{eq:basic-ineq}
	s^2df \wed d^c f + t^2 dg\wed d^c g \geq st (df \wed d^c g + dg\wed d^c f).
\]
More generally, if $u, v \in L^\infty(\Om,\bR)$, then by the Cauchy-Schwarz inequality for every strongly positive test $(n-1, n-1)$-form of $\Phi$,
$$\begin{aligned}
	2 \int u v \;d f \wed d^c g \wed \Phi  &\leq 2  \left( \int u^2 df \wed d^c f \wed \Phi \right)^\frac{1}{2} \left(  \int v^2 dg \wed d^c g \wed \Phi \right)^\frac{1}{2} \\
&\leq \int u^2 df \wed d^c f \wed \Phi + \int v^2 dg \wed d^c g \wed \Phi.
\end{aligned}$$
Therefore,
\[\label{eq:basic-ineq-cs}
	2 uv \;df \wed d^c g \leq u^2 df \wed d^c f + v^2 dg \wed d^c g
\]
in the weak sense of currents.

\begin{lem}\label{lem:ineq-max-min} Let $f, g \in W^*(\Om)$. Let $T$ and $S$ be positive $(1,1)$-currents satisfying $df\wed d^cf \leq T$ and $dg\wed d^c g\leq S$. Then,
$$\begin{aligned}
	d \max\{f,g\} \wed d^c\max\{f,g\} \leq  	T +   S. \\
	d \min\{f,g\} \wed d^c\min\{f,g\} \leq  	T + S. 
\end{aligned}$$
Consequently, 
\[\notag	d |f| \wed d^c|f| \leq 2 T.
\]
\end{lem}

\begin{proof} These are direct consequences of \eqref{eq:grad-max} and positivity of currents $S$ and $T$.
\end{proof}

This lemma combined with the definition of  $\|\cdot\|_*$ yields

\begin{lem}\label{lem:norm-min-max} Let $f, g \in W^*(\Om)$. Then,  $\max\{f, g\}$ and $\min\{f,g\}$ belong to $W^*(\Om)$ and 
$$
	\|\max\{f,g\} \|_* \leq \|f\|_* + \|g\|_*, \quad
	\|\min\{f, g\}\|_* \leq \|f\|_* + \|g\|_*.
$$
Furthermore, for any constant $c\in \bR$,
\[ \label{eq:norm-max-min}
	\| \max\{f,c\}\|_* \leq \| f \|_*, \quad \|\min\{f,c\}\|_* \leq \|f\|_*
\]
\end{lem}

An immediate consequence which is very useful  is  as follows.

\begin{cor} If $f\in W^*(\Om)$, then $f^+ = \max\{f,0\}$ and $f^- = -\min\{f, 0\}$ belong to $W^*(\Om)$. Moreover,  $\|f^+\|_* \leq \|f\|_*$, $\|f^-\|_* \leq \|f\|_*$ and $\| |f| \|_* \leq 2 \|f\|_*$.
\end{cor}

\subsection{Basic examples of functions}  We provide several examples of functions in the  new space.

\begin{expl} \label{expl:obv} Smooth functions with compact supports functions or  global Lipschitz functions in $\Om$ belong to $W^*(\Om)$.
\end{expl}

\begin{expl} \label{expl:basic} Let $\Om$ be an open set in $\bC^n$.
\begin{itemize}
\item[(a)]
If $u, v \in PSH(\Om) \cap L^\infty(\Om)$, then $u \pm  v \in W^*(\Om)$. 
\item[(b)] Let $u \in PSH(\Om) \cap L^1_{\rm loc}(\Om)$ and $u \leq -1$. Then, $-\log (-u) \in  PSH(\Om)$. A direct computation shows that $-\log (-u) \in W^*(\Om')$ for every $\Om'\subset\subset \Om.$ In general they are unbounded.
\item[(c)]
Let $\Om = B(0, 1/2)$. Then, $u(z_1, ...,z_n) = -(-\log |z_1|)^a \in W^*(\Om)$  if and only if $a<1/2.$ This function is also plurisubharmonic.
\end{itemize}
\end{expl}

The next result is due to Vigny \cite{Vigny} and Dinh, Marinescu and Vu \cite[Lemma~2.8]{DMV}.

\begin{lem} \label{expl:construction} Let $f\in W^*(\Om) \cap L^\infty(\Om)$. Assume  $df \wed d^c f \leq dd^c u$ for some negative function $u \in PSH(\Om) \cap C^0(\Om)$.  Denote by $N:= \|f\|_{L^\infty(\Om)}>0$ and we set $u_N := \max\{u, -N\}+N$. Then, 
$$
	d (u_N f) \wed d^c (u_N f) \leq N^2 dd^c (u_N^2 + u_{N+1}).
$$
Consequently, $u_N f \in W^*(\Om')$ for every $\Om'\subset\subset \Om$.
\end{lem}

\begin{proof} It follows from \eqref{eq:basic-ineq-cs} that
$$
	d (u_N f) \wed d^c (u_N f) \leq 2 u_N^2 df\wed d^c f + 2 f^2 du_N \wed d^c u_N
$$
as positive $(1,1)$-currents whose coefficients are in $L^1(\Om)$. It is easy to see that $du_N \wed d^c u_N \leq dd^c u_N^2$ in the weak sense as  $0\leq u_N \leq N$ and it is psh.    The continuity of $u$ implies
$$
	dd^c \max\{u, -N-1\} = dd^c u 
$$
on the open set $\{u>-N-1\} \supset \{u \geq -N\}$. Therefore,  $u_N^2 df\wed d^c f \leq N^2 dd^c u_{N+1}$. The proof follows easily by combining these facts.
\end{proof}

The continuity assumption of $u$ will not be needed in Lemma~\ref{expl:construction} if we invoke an approximation argument below (Remark~\ref{rmk:general-construction}).

\subsection{Approximation}
\label{ss:approx}

Let us fix the standard smoothing kernel  as follows. Define 
$ \eta(z) \in C^\infty(\bC^n)$ by 
$$
	\eta(z) 
= \begin{cases}  c_0 e^{\frac{1}{|z|^2-1}} &\quad \text{if }  |z| <1,\\
		0 &\quad \text{if } |z| \geq 1,
	\end{cases}
$$
where the constant $c_0$ is chosen such that $\int_{\bC^n} \eta (x) dx =1$. For each $\veps>0$, set 
\[ \label{eq:molifier}\eta_\veps = \frac{1}{\veps^n} \eta\left(\frac{|z|}{\veps} \right).
\]
This family is used repeatedly throughout the paper.  We also write
\[\label{eq:epsilon-dom}
	\Om_\veps := \{z\in \Om : {\rm dist} (z, \d\Om) >\veps\}.
\]

The following simple observation is very useful.
\begin{lem} \label{lem:ineq-weak} Let $f\in W^{1,2}(\Om)$ and $T = \sum T_{j\bar k}  \frac{\ii}{\pi}  dz_j \wed d\bar z_k$ a $(1,1)$-current of order zero. Then, 
$$
df \wed d^c f \leq T
$$ 
weakly 
if and only if   for every $\ze \in \bC^n$,
\[\label{eq:test-distribution}
	\sum_{j,k} \d_j f (z) \d_{\bar k} f (z) \ze^j \bar\ze^k \leq \sum_{j,k} T_{j\bar k} (z) \ze^j \bar\ze^k
\]
in the sense of distributions of zero order, where $\d_j f := \d f/\d z_j$ and $\d_{\bar k} f = \d f/\d\bar z_k$.
\end{lem}

\begin{proof}  Clearly, $df\wed d^c f\leq T$ holds if and only if 
$$
	\sum (T_{j\bar k} - \d_j f \;\d_{\bar k} f)  \ii dz^j \wed d\bar z^k = \sum S_{j\bar k}  \ii dz^j \wed d\bar z^k =: S \geq 0
$$
as a positive $(1,1)$-current, where $S_{j\bar k} = T_{j\bar k} - \d_j f \d_{\bar k}f$. Since $S*\eta_\veps$ converges weakly to $S$ as $\veps \to 0$, the above inequality is equivalent to $S *\eta_\veps \geq 0$ as (1,1)-forms. By a characterization of positive $(1,1)$-forms, this holds if and only if 
$$
	\sum (S_{j\bar k} *\eta_\veps) \ze^j \bar\ze^k = \sum (S_{j\bar k} \ze^j \bar\ze^k)*\eta_\veps \geq 0 \quad \forall \ze\in \bC^n.
$$
Therefore, $S\geq 0$ if and only if 
$
	\sum S_{j\bar k} \ze^j \bar\ze^k \geq 0 
$
weakly as distributions of order zero for every $\ze \in \bC^n$. The proof is completed. 
\end{proof}

It leads to a characterization 
\begin{prop}\label{prop:ineq-weak}  Let $f \in W^{1,2}(\Om)$ and $T$ is a positive $(1,1)$-current. Then,
$$
	df \wed d^c f \leq T
$$
in the weak sense of currents if and only if for every $ \veps>0$,
$$
	d (f*\eta_\veps) \wed d^c (f *\eta_\veps) \leq T*\eta_\veps 
$$
in the weak sense of currents (or point-wise) in $\Om_\veps$.
\end{prop}
\begin{proof} The sufficient condition is easy because if the second inequality holds for every $\veps>0$, then letting $\veps\to 0$ one gets the first inequality. It remains to prove
the necessary condition. Assume $df \wed d^c f \leq T$ weakly where $T = \sum T_{j\bar k}  \frac{\ii}{\pi}  dz_j \wed d\bar z_k$. 
Since
$$
	f*\eta_\veps(z) = \frac{1}{\veps^{2n}}\int f (z -x) \eta \left(\frac{x}{\veps} \right) dx = \int f(z-x) \eta_\veps(x)dx,
$$
we compute
$$ \frac{\d }{\d z_j } (f *\eta_\veps (z))  =  \int (\d_j f) (z-x) \eta_\eps(x) dx.
$$
By the Cauchy-Schwarz inequality for $\ze \in \bC^n$, 
$$\begin{aligned}
	&\sum \d_j f (z-x) \; \d_{\bar k} f (z-y) \ze^j \bar\ze^k \\
	&\leq \frac{1}{2}\sum \d_j f (z-x) \d_{\bar k} f (z-x) \ze^j \bar\ze^k +\frac{1}{2} \sum \d_j f(z-y) \d_{\bar k} f(z-y) \ze^j \bar\ze^k.
\end{aligned}$$
Since $\eta_\veps(x) \geq 0$, it follows that 
$$\begin{aligned}
&	2\iint  \sum \d_j f (z-x)  \d_{\bar k} f (z-y) \ze^j \bar\ze^k\eta_\veps(x) \eta_\eps(y) \,dxdy \\
& 	\leq  \iint \sum \d_j f (z-x) \d_{\bar k} f (z-x) \ze^j \bar\ze^k \eta_\veps(x) \eta_\eps(y)  \,dxdy  \\
&\quad	+ \iint \sum \d_j f (z-y) \d_{\bar k} f (z-y) \ze^j \bar\ze^k \eta_\veps(x) \eta_\eps(y)\,dxdy \\
&= 2 \int \sum \d_j f (z-x) \d_{\bar k} f (z-x) \ze^j \bar\ze^k \eta_\veps(x)  \,dx \\
&:=2 I,
\end{aligned}$$
where  where we used the fact that $\int \eta_\veps (x) dx =1$ for every $\veps>0$. 
The inequality \eqref{eq:test-distribution} in Lemma~\ref{lem:ineq-weak}  gives 
$$\begin{aligned}
I \leq	\int \sum T_{j\bar k} (z-x) \ze^j \bar\ze^k \eta_\veps(x)\, dx  
=   \sum T_{j\bar k} *\eta_\veps (z) \ze^j \bar\ze^k,
\end{aligned}$$
Therefore,
$$
	\sum \d_j (f*\eta_\veps) (z) \; \d_{\bar k}(f*\eta_\veps) (z) \ze^j \bar\ze^k \leq \sum (T_{j\bar k}*\eta_\veps) (z) \ze^j \bar\ze^k
$$
in the distributional (or point-wise) sense.
The left hand side is rewritten as
$$
	\sum [(\d f/\d z_j  ) *\eta_\veps] \; [(\d f/\d \bar z_{k}) *\eta_\veps] \ze^j \bar\ze^k. 
$$
In other words, 
$
	d f_\veps \wed d^c f_\veps \leq T_\veps$ in  $\Om_\veps$
by the characterization of positive $(1,1)$-forms, see e.g., Demailly's book \cite[Chapter~3, Corollary~1.7]{D-book}. 
\end{proof}

\begin{remark} \label{rmk:general-construction} Lemma~\ref{expl:construction} is valid for a general bounded psh function $u$, i.e., it needs not be continuous. In fact, let $f$ and $u$ are functions in the lemma, define $f_\veps=f*\eta_\veps$ and $u_\veps = u*\eta_\veps$. They are smooth and satisfy $df_\veps \wed d^c f_\veps \leq dd^c u_\veps$ by Proposition~\ref{prop:ineq-weak}. So, 
\[\label{eq:rmk-bdd}
	d (u_{N,\veps} f_\veps) \wed d^c (u_{N,\veps} f_\veps) \leq N^2 (dd^c u_{N,\veps}^2 + u_{N+1,\veps}),
\]
where $u_{N,\veps} = \max\{u_\veps, -N\} +N$. Since  $u_{N,\veps}$  are bounded psh functions decreasing to $u_N$ as $\veps \to 0$, it follows that the right hand side of \eqref{eq:rmk-bdd} converges weakly to $N^2 dd^c(u_N^2 + u_{N+1})$. On the other hand, the above inequality implies that   $\{v_\veps f_\veps\}_{\veps>0}$ has uniformly bounded norm  in $ W^{1,2} (\Om')$ for a given open subset $\Om'\subset \subset \Om$. So, it has a subsequence converges weakly to $w \in W^{1,2}(\Om')$. However, $u_{N,\veps} f_\veps \to u_N f$ in $L^1(\Om')$. Hence, $w = u_Nf$ in $L^1(\Om')$ by the uniqueness of weak limit in $W^{1,2}(\Om')$. We conclude that the left hand side in \eqref{eq:rmk-bdd}  converges weakly in $W^{1,2}(\Om')$ to $d(u_Nf) \wed d^c(u_Nf)$ as $\veps\to 0$.
\end{remark}

The following result is motivated by the one for functions of bounded variations \cite[Chapter~5]{EG92}.

\begin{prop}\label{prop:smoothing} Let $f \in W^*(\Om)$. Then,
$ f_\veps = f*\eta_\veps \in W^* \cap C^\infty(\Om_\veps)$ and 
$$\lim_{\veps\to 0}\|f_\veps\|_{W^*(\Om_\veps)} = \|f\|_{W^*(\Om)}.$$
Similarly, for every $\Om'\subset\subset \Om$, $\lim_{\veps \to 0} \|f_\veps\|_{W^*(\Om')} = \|f\|_{W^*(\Om')}$. Moreover, for $0< \veps < {\rm dist}(\d\Om, \Om')$,  $$\|f_\veps\|_{W^*(\Om')} \leq \|f\|_{W^*(\Om)}.$$
\end{prop}

\begin{proof}
Let $T$ be a closed positive $(1,1)$-current realizing  $W^*$-norm of $f$ (Definition~\ref{defn:w-current}). Since 
$$
	\int_\Om T \wed\om^{n-1} = \int_\Om \sum_{j=1}^n T_{j\bar j} <+\infty,
$$
and each coefficient $T_{j \bar j}$  is a positive Radon measure, so we may extend them to $\bC^n$ as zero outside $\Om$ by setting $\wt T:= {\bf 1}_{\Om} \cdot T$ (it is still a positive current, but it is no longer closed). Hence, $(\wt T)_\veps\wed \om^{n-1}$ is defined in $\bC^n$ where $(\wt T)_\veps = (\wt T)*\eta_\veps$. Moreover, we have $(\wt T)_\veps = T*\eta_\veps =: T_\veps$ on $\Om_\veps$.

By Proposition~\ref{prop:ineq-weak} we have
\[\label{eq:smoothing-ineq}\begin{aligned}
	\|f_\veps\|_{W^*(\Om_\veps)}^2 
&\leq \int_{\Om_\veps} |f_\veps|^2 dx + \int_{\Om_\veps} T_\veps \wed \om^{n-1}\\
&\leq \int_\Om |f|^2 dx + \int_\Om T \wed \om^{n-1},
\end{aligned}\]
where in the second inequality we used the Young's and Minkowski's inequalities for the convolutions with $\eta_\veps$ of functions in $L^2(\Om)$ and of positive Radon measures in $\bC^n$ (see, e.g., \cite[Propositions~8.9, ~8.49]{Fd99}). It follows that $$\limsup_{\veps \to 0} \|f_\veps\|_{W^*(\Om_\veps)} \leq \|f\|_*.$$ 
On the other hand, for each $\veps>0$ let $S_\veps$ be a closed positive $(1,1)$-current which realize the norms of $f_\veps$ in $\Om_\veps$. They have uniformly bounded mass on compact subsets. Hence,  
there exists a closed positive $(1,1)$-current $S$ on $\Om$ such that
$$
	S_\veps \to S  
$$ 
weakly in subsequence (we still denote by the same sequence for simplicity).  Clearly, $df \wed d^c f \leq S$ weakly in $\Om$ as $f_\veps \to f$ in $W^{1,2}_{\rm loc}(\Om)$. 
Next, we observe that $\|S\wed \om^{n-1}\|_{\Om} <+\infty$. 
In fact, since $S_\veps \to S$ weakly, it follow that for an open set $D\subset\subset \Om$, 
$$
 \liminf_{\veps\to 0} \int_{\Om_\veps} S_\veps \wed \om^{n-1} \geq \liminf_{\veps \to 0}	\int_{D} S_\veps \wed \om^{n-1}\geq \int_{D} S \wed \om^{n-1}.
$$
Letting $D \uparrow \Om$, we obtain
$
	\liminf_{\veps \to 0} \int_{\Om_\veps} S_\veps \wed \om^{n-1}  \geq \int_\Om S\wed \om^{n-1}.
$
Therefore,
$$
	\|f\|_{W^*(\Om)}^2 \leq \int_\Om |f|^2 dx + \int_\Om S\wed \om^{n-1} \leq \liminf_{\veps\to 0} \|f_\veps\|^2_{W^*(\Om_\veps)}.
$$
In conclusion we have $\liminf_{\veps \to 0} \|f_\veps\|_{W^*(\Om_\veps)} \geq \|f\|_* \geq \limsup_{\veps \to 0} \|f_\veps\|_{W^*(\Om_\veps)}$. 

Next, the second identity is proved similarly to the first one but it is easier without change of the domain of integration. Finally, $\|f_\veps\|_{W^*(\Om')} \leq \|f_\veps\|_{W^*(\Om_\veps)}$ and  \eqref{eq:smoothing-ineq} implies the last inequality in the proposition.
\end{proof}

We record here a useful fact that Lemma~\ref{lem:ineq-weak}, Proposition~\ref{prop:ineq-weak} and Proposition~\ref{prop:smoothing} are valid for other positive families of approximation of unity.  In particular, the family
$$
	\chi_r = \frac{{\bf 1}_{B(0,r)}}{|B(0,r)|}, \quad r>0
$$
and 
\[\label{eq:mean}
\begin{aligned}
	(f)_{x,r}:= f * \chi_r (x) &=  \frac{1}{|B(x,r)|} \int_{B(x,r)} f(y) dy.
\end{aligned}\]
Here $B(z,r)$ is the ball of radius $r$ and center at the $x$ and $|B(x,r)|$ is its Lebesgue measure.
Then, $(f)_r(x) = (f)_{x,r}$ belongs to $W^* \cap C^0(\Om_r)$.
The same is true for  
$$
	(|f|)_{x,r} = |f|*\chi_r(x).
$$

For simplicity we denote the average of $f$ over a general Borel set $E \subset \bC^n$ by
\[\label{eq:avg-int}
	\intavg_E f dx = \frac{1}{|E|} \int f dx,
\]
where $|E|$ is the Lebesgue measure of $E$ in $\bC^n$.  A classical theorem in  Evans and Gariepy's book \cite[Sect. 1.7, Theorem~1]{EG92}) says that  for $f\in L^1_{\rm loc}(\Om)$, we have
$$
	\lim_{r\to 0}	\intavg_{B(x,r)} f dy = f(x) 
$$
for $x\in \Om \setminus F $ for some Borel set $F$ with $|F| =0$. Furthermore, its stronger version \cite[Section. 1.7, Corollary~1]{EG92} also holds 
\[\label{eq:LP}
	\lim_{r\to 0} \intavg_{B(x,r)} |f-f(x)| dy =0 \quad x\in \Om\setminus F.
\] 

Next, we recall a useful consequence of this fact whose  proof  is derived from \cite[Theorem~2.35]{KLV-book}. 
\begin{prop}  \label{prop:weak-conv-cp} Let $1<p<\infty$ and $E \subset \Om$ be a Borel set.  Let $\{f_j\} \subset L^p(\Om, \bR)$ and $\{g_j\} \subset L^p(\Om,\bR)$ be such that  $f_j \to f \in L^p(\Om,\bR)$ weakly and $g_j \to g\in L^p(\Om,\bR)$ weakly. 
If $|f_j(x)| \leq g_j(x)$ for a.e. $x\in E$, then $|f(x)| \leq g(x)$ a.e. $x\in E$. 
\end{prop}

\begin{proof}  
Define $\wh f_j = {\bf 1}_{E} f_j$ and $\wh g_j = {\bf 1}_{E} g_j$. Then,
$|\wh f_j| \leq  \wh g_j$  a.e. in  $\Om.$ We also have both
$\wh f_j \to \wh f$ and $\wh g_j \to \wh g$ weakly in $L^p(\Om)$.
Thus, by replacing the sequences we may assume that $E= \Om$. 
Let $x \in \Om$ be a Lebesgue point of both $g$ and $f$. Let $0<r <{\rm dist} (x, \d\Om)$. 

We assume first that $(f)_{x,r}=\intavg_{B(x,r)} f dy \neq 0$.
Define
$$
	\psi (y)= \frac{{\bf 1}_{B(x,r)}}{|B(x,r)|} (y)\times \frac{(f)_{x,r}}{|(f)_{x,r}|}.
$$
By the Cauchy-Schwarz inequality and the assumptions on weak convergence,
$$\begin{aligned}
	\left| \intavg_{B(x,r)} f dy\right| 
&= \int_\Om f(y) \psi(y) dy 
= \lim_{j\to \infty} \int_\Om f_j \psi dy \\
&\leq \liminf_{j\to \infty} \int_\Om |f_j| |\psi| dy \\
&\leq \liminf_{j\to \infty} \int_\Om g_j |\psi| dy \\
&= \intavg_{B(x,r)} g dy.
\end{aligned}$$
Therefore, 
$$ \left| \intavg_{B(x,r)} f dy\right| \leq \intavg_{B(x,r)} g dy.
$$
This inequality obviously holds when $(f)_{x,r}=0.$ Now letting $r\to 0$ we will get the conclusion as the limit exists almost everywhere in $\Om$ by \eqref{eq:LP}.
\end{proof}

\subsection{Compactness and completeness} We first show that the complex Sobolev space enjoys good compactness properties under a uniform bound of $W^*$-norm. Then, using this we prove the basic functional properties of the space.  

Below we use several times the weak convergence in $W^{1,2}(\Om,\bR)$. A useful criterion for such a convergence is as follows.  
By the Riesz representation theorem, a sequence $f_j \to f$ weakly in $W^{1,2}(\Om,\bR)$ if and only if $f_j \to f$ weakly in $L^2(\Om)$ and $\na f_j \to \na f$ weakly in $L^2(\Om, \bR^{2n})$.

\begin{lem} \label{lem:weak-compactness} Let $\{f_j\}_{j\geq 1} \subset W^*(\Om)$. Assume that $\| f_j \|_* \leq A$ for every $j\geq 1$. If $f_j \to f$ weakly in $W^{1,2}(\Om)$, then $f \in W^*(\Om)$ with $\|f\|_* \leq A$.
\end{lem}

\begin{proof} Let $T_j$ be a minimum current of $ f_j$, $j\geq 1$. Since $\|T_j \|_\Om \leq A$, there exists a subsequence $T_{j_k}$ converging weakly to a closed positive (1,1)-current $T$ and $\|T\|_\Om \leq \liminf \|T_{j_k}\|_\Om$. We still write  $ f_j, T_j$ instead of $f_{j_k}, T_{j_k}$ for simplicity. Let $\phi$ be a positive test form in $\Om$. Then,
$$
	\lc d f_{j} \wed d^c f_{j}, \phi\rc \leq T_{j} (\phi).
$$
Note also that
$$\begin{aligned}
0&\leq	\lc d (f- f_j) \wed d^c (f- f_j), \phi\rc  \\
&= \lc d f \wed d^c f, \phi \rc + \lc d f_j \wed d^c  f_j , \phi\rc - 2 \lc d f \wed d^c  f_j,\phi \rc \\
&\leq \lc d f \wed d^c f, \phi \rc + T_j(\phi) -  2 \lc d f\wed d^c  f_j,\phi \rc.
\end{aligned}$$
Letting $j \to \infty$, we get from weak convergence in $W^{1,2}(\Om)$ and weak convergence of currents that  $\lc d f \wed d^c f,  \phi \rc \leq T (\phi)$. So, 
$f \in W^*(\Om)$. Moreover, if we write $\lc f, g\rc_\Om = \int_\Om fg dx$ for the inner product in the Hilbert space $L^2(\Om)$, then
$$
	0 \leq \lc f -  f_j, f - f_j \rc_\Om =  \|f\|^2_{L^2(\Om)}+ \| f_j\|^2_{L^2(\Om)} - 2 \lc f,  f_j\rc_\Om.
$$
The semi-continuity of norm in weak convergence of $ f_j\to f$ in $W^{1,2}(\Om)$ (see e.g., \cite[Theorem~2.11]{LL-book}) implies $$  \|f\|^2_{L^2(\Om)} \leq  \liminf_{j\to \infty} \| f_j\|^2_{L^2(\Om)}.$$ Thus, $\|f\|_*\leq A$.
\end{proof}

\begin{cor} \label{cor:w-compactness}  Let $\{f_j \}_{j\geq 1} \subset W^*(\Om)$ be such that $\|f_j \|_* \leq A$ for every $j\geq 1$. There exists a subsequence $\{f_{j_k} \}$ converging weakly in $W^{1,2}(\Om)$ to $f \in W^*(\Om)$  and $\|f\|_* \leq \liminf_{j\to \infty} \|f_j\|_*$.
\end{cor}

\begin{proof} It follows immediately from $\|f\|_{W^{1,2}(\Om)} \leq \|f\|_*$ for every $f\in W^*(\Om)$ and Lemma~\ref{lem:weak-compactness}.
\end{proof}

The following basic functional result is due to Vigny\cite[Proposition~1]{Vigny}.
\begin{thm}\label{thm:star-norm} Let $\Om$ be a bounded domain in $\bC^n$.
\begin{itemize}
\item[(a)]  $\| \cdot\|_*$ is a norm in $W^*(\Om)$ and $\| \cdot \|_{W^{1,2}(\Om)} \leq \|\cdot \|_*$.
\item[(b)] $W^*(\Om)$ and $W^*_0(\Om)$ are Banach spaces with this norm.
\end{itemize}
\end{thm}

\begin{proof} Let $f, g \in W^*(\Om)$ and let $T,  S$ be the minimum currents of $f, g$, respectively. Choosing $s,t >0$ with $st =1$ in \eqref{eq:basic-ineq} we get 
$$\begin{aligned}
	d(f+g)\wed d^c (f+g) 
&	= df \wed d^c f + dg \wed d^c g + [df\wed d^c g + dg\wed d^c f]  \\
&	\leq (1+ s^2) df \wed d^c f + (1+ t^2) dg \wed d^c g \\
&	\leq (1+ s^2) T + (1+t^2) S.
\end{aligned}$$
It follows that 
\[ \label{eq:triangle-parameter}
	\|f+g\|_*^2 \leq  \min_{s>0} \left\{ (1+ s^2) \|f\|_*^2 + (1+ 1/s^{2}) \|g\|_*^2\right\}.
\]
We now verify the triangle inequality. Without loss of generality,  we may assume that $g \neq 0$. Moreover,  we may assume that $\|g\|_* =1$ by dividing by a constant. We need to show
$$
	\|f+g\|_* \leq \|f\|_* +1.
$$
If $f =0$, then the inequality holds. Otherwise, 
applying \eqref{eq:triangle-parameter} for $s^2 = 1/ \|f\|_* >0$,
$$\begin{aligned}
	\|f+g\|_*^2
&	\leq  (1+s^2) \|f\|_*^2 + 1 + 1/s^2  \\
&	= (1+ \|f\|_*^2) + s^2 \|f\|_*^2 + 1/s^2 \\
&	= (1+ \|f\|_*)^2.
\end{aligned}$$
This proves the triangle inequality.

Next, we  show that it is a complete space with this norm. In fact, let $\{f_j\}$ be a Cauchy sequence with respect to $\|\cdot\|_*.$ Since the $W^{1,2}$-norm is dominated by the $W^*$-norm, the sequence is Cauchy in $W^{1,2}(\Om)$. By Lemma~\ref{lem:weak-compactness} it converges to $f \in W^*(\Om)$ with respect to $W^{1,2}$-norm. Let $S_{jk}$ be the minimum current for $f_j - f_k$. We have
$$
	d(f_j -f_k) \wed d^c (f_j - f_k) \leq S_{jk}.
$$
Let $\veps>0$. There exists $k_0$ such that for  $j, k \geq k_0$,
$$
	\|S_{jk} \|_{\Om} \leq \|f_j - f_k\|_* \leq \veps. 
$$
By weak compactness for (1,1)-currents, for each fixed $j$ there exists a subsequence $\{S_{j k_s}\}_{k_s\geq 1}$ converging weakly to  a closed positive (1,1)-current $\tau_j$ and $\|\tau_j\|_\Om \leq \veps$. 
We have for a positive test form $\phi$,
$$
	 d (f_j - f_{k_s}) \wed d^c (f_j - f_{k_s}) (\phi)   \leq S_{j k_s} (\phi).
$$
Letting $k_s\to \infty$, the left hand side  converges to $d(f_j - f) \wed d^c (f_j -f) (\phi)$ by the Lebesgue domination convergence theorem and the right hand side converges to $\tau_j(\phi)$. Therefore,
$ \| f_j - f\|_* \leq 2\veps$ for $j \geq k_0$. In other words, $f_j$ converges to $f$ with respect to the $W^*$-norm.
\end{proof}

\begin{prop} Let $f, g\in W^*(\Om) \cap L^\infty(\Om)$. 
\begin{itemize}
\item
[(a)] $fg\in W^*(\Om)\cap L^\infty(\Om)$.
\item
[(b)] If in addition $f\in W^*_0(\Om)$, then $fg\in W^*_0(\Om)$.
\end{itemize}
\end{prop}

\begin{proof} The statements (a) and (b) for the Sobolev spaces $W^{1,2}(\Om) \cap L^\infty(\Om)$ and $W^{1,2}_0(\Om)$ are classical, see e.g., \cite[Proposition~9.4, ~9.18]{Bz11}. On the other hand, the basic inequality \eqref{eq:basic-ineq-cs} implies
$$\begin{aligned}
d(fg)\wed d^c (fg) &\leq 2 f^2 dg\wed d^c g + 2g^2 df \wed d^c f \\
&\leq 2 f^2 T_f + 2 g^2 T_g,
\end{aligned}$$
where $T_f, T_g$ are the minimum currents of $f,g$, respectively. 
Since $f,g$ are bounded, the last current is dominated by a closed positive current $A (T_f + T_g) \in \Ga_{fg}$.
\end{proof}

Recall that a function $f\in L^1_{\rm loc}(\Om)$  belongs to ${\rm BMO}(\Om)$ (bounded mean oscillation) if there exists a constant $A>0$ such that
\[\label{eq:bmo}
	 \frac{1}{|B|} \int_{B} |f - f_B | dx \leq A
\]
holds for all balls $B\subset \Om$, where as in \eqref{eq:avg-int}, 
$$f_B =  \frac{1}{|B|}\int_B f dx.$$
Moreover, a function $f\in {\rm BMO}(\Om)$ belongs to ${\rm VMO}(\Om)$ (vanishing mean oscillation) if 
\[ \label{eq:vmo}
	\lim_{|B| \to 0} \frac{1}{|B|} \int_B |f - f_B | dx  =0
\]
holds for every balls $B\subset \Om$.

A well-known fact is that if a psh function $u$ belongs to ${\rm VMO}(\Om)$, then its Lelong number vanishes everywhere. The converse also holds true thanks to a recent work by Biard and Wu \cite{BW24}.

\begin{expl}\label{expl:VMO} Let $n\geq 2$ and assume $0\in \Om\subset \bC^n$. For $1\leq \ell \leq n$, the function 
$$u(z_1,...,z_n) = \frac{|z_\ell|^2}{\sum_{i =1}^n |z_i|^2}$$
belongs to $W^*(\Om) \cap {\rm BMO}(\Om)$ but not to ${\rm VMO}(\Om)$.
In fact, the sequence 
$$ u_\veps = \frac{|z_\ell|^2}{\veps+ \sum_{i =1}^n |z_i|^2}
$$
are smooth functions in $W^*(\Om)$ whose $W^*$-norms are uniformly bounded. Also $u_\veps \uparrow u$ in $W^{1,2}(\Om)$. Hence, $u\in W^*(\Om)$ by Lemma~\ref{lem:weak-compactness}. Using the homogeneity of $u$ we can easily see that $u\in {\rm BMO}(\Om)$ but not in ${\rm VMO}(\Om)$.
\end{expl}

We have the following  BMO estimate for $W^*(\Om)$ similar to psh functions.

\begin{lem} Every function in $W^*(\Om)$ belongs to ${\rm BMO}(\Om')$, where $\Om'\subset \subset \Om$.
\end{lem}

\begin{proof} Let $h\in W^*(\Om)$ and assume $dh\wed d^c h \leq T$, where $T$ is a closed positive current with $\|T\|_\Om <+\infty$. Let $B(x,r) \subset \Om$ be a ball. Recall that 
$$ (h)_{x,r} := \intavg_{B(x,r)} h(y) dy. $$
It follows from the Poincar\'e inequality \cite[Eq (7.45), page 164]{GT} in balls for $h \in W^{1,2}(\Om)$ that 
$$\begin{aligned}
	\int_{B(x,r)} |h - (h)_{x,r}|^2 dy 
&\leq 	A  r^2 \int_{B(x,r)} |\na h|^2 dy \\
&\leq		A r^2 \int_{B(x,r)} dh \wed d^c h \wed \om^{n-1},
\end{aligned}$$
where $A = A (n)$.
 Then,   by H\"older inequality
\[\label{eq:lelong-number}\begin{aligned}
	\intavg_{B(x,r)} |h - (h)_{x,r}| dy 
&\leq 	\left( \intavg_{B(x,r)} |h - (h)_{x,r}|^2 dy \right)^\frac{1}{2} \\ 
&\leq  	A' \left( \frac{1}{r^{2n-2}} \int_{B(x,r)} T\wed \om^{n-1} \right)^\frac{1}{2}\\
&\leq  A' \sqrt{\nu (T, x,r)},
\end{aligned} 
\]
where $\nu(T,x,r)$ is an increasing function of $r \in [0,r_0]$ and $\lim_{r\to 0} \nu(T,x,r) = \nu(T,x)$ the (finite) Lelong number of $T$ at $x$. Hence, $h\in {\rm BMO}(\Om')$.
\end{proof}

The following result shows that there are differences and difficultly to deal with non-smooth functions in the complex Sobolev spaces in higher dimension. This can be anticipated by Example~\ref{expl:basic} as the space contains singular psh functions.

The next result is contained in \cite[Proposition~7]{Vigny}.
\begin{prop} \label{prop:density} Let $\Om\subset \bC^n$ be a bounded domain. 
\begin{itemize}
\item[(a)] If $n=1$, then $W^*(\Om) = W^{1,2}(\Om)$ together with its Sobolev norm. 
\item[(b)] If $n\geq 2$, then the continuous functions in $C^0(\Om) \cap W^*(\Om)$ are not dense in $W^*(\Om)$ with respect to $W^*$-norm. 
\end{itemize}
\end{prop}

\begin{proof}  (a) is obvious. Next, we prove (b).  The proof used an idea from  \cite[Proposition~7]{Vigny}. Namely, 
in Example~\ref{expl:VMO} there are plenty of functions in $W^*(\Om)$ which does not belong to ${\rm VMO}(\Om)$.  Let $f$  be  such a function. We show that it cannot be approximated by functions in $W^*(\Om)\cap C^0(\Om)$. 

In fact, if $g \in C^0(\Om)$, then 
$
	\lim_{r\to 0} \intavg_{B(x,r)} |g - (g)_{x,r}| dy =0
$.  Hence,  for $h = f -g$ with $g\in W^*(\Om)\cap C^0(\Om)$, 
$$\begin{aligned} 
& \lim_{r\to 0} \intavg_{B(x,r)} |f - (f)_{x,r}|dy  \\
&= \lim_{r\to 0}	\left( \intavg_{B(x,r)} |f - (f)_{x,r}| dy - \intavg_{B(x,r)} |g- (g)_{x,r}| dy \right) \\&\leq  \lim_{r\to 0}\intavg_{B(x,r)} |h-(h)_{x,r}| dy.
\end{aligned}$$
Since $h \in W^*(\Om)$, there exists a positive closed $(1,1)$-current such that  $d h \wed d^c h \leq T$. The inequality \eqref{eq:lelong-number} applied for $h$ combined with the inequality above gives
$$
 \lim_{r\to 0} \intavg_{B(x,r)} |f - (f)_{x,r}|dy  \leq  A' \sqrt{\nu (T, x)},
$$
where $\nu(T, x)$ is the Lelong number of $T$ at the point $x$. This means that if $C^0(\Om) \cap W^*(\Om)$ is dense in $W^*(\Om)$ with respect to $W^*$-norm, then we can find such a closed positive current with arbitrary small mass $\|T\|_{B(x,r)}$. In particular, 
$$
	\limsup_{r\to 0} \intavg_{B(x,r)} |f - (f)_{x,r}| dy =0, \quad x\in \Om.
$$
This means $f\in {\rm VMO}(\Om)$ which is impossible.
Thus, $W^*(\Om)\cap C^0(\Om)$ cannot be dense in $W^*(\Om)$ with respect to $W^*$-norm.
\end{proof}

This leads to a consequence \cite[Corollary~8]{Vigny} which is similar to one of the space of $\de$-psh functions studied by  Cegrell and Wiklund \cite{CW05}. We give the detail proof for the reader's convenience. 

\begin{cor} \label{cor:reflexive} Let $\Om$ be an open set in $\bC^n$,  $n\geq 2$. Then, $W^*(\Om)$ is not reflexive.
\end{cor}

\begin{proof} We follow closely the one in \cite[Corollary~8]{Vigny}. Suppose that the space is reflexive. Then, the closed unit ball in $W^*$-norm is weakly compact with respect to the dual topology. Let $0\leq u \leq 1$ be a function in Example~\ref{expl:VMO} and $0 \leq \chi \leq 1$ be a cut-off function in $\Om$. Clearly, $\chi u \in W^*(\Om)$. Define $u_k = u*\eta_{1/k}$.  It follows from Proposition~\ref{prop:smoothing} that $f_k=\chi u_k \in W^*(\Om) \cap C^\infty(\Om)$ and $\|f_k\|_* \leq A$ for a constant $A$ independent of $k$.
By the weak compactness (Corollary~\ref{cor:w-compactness})  there exists a subsequence $f_j \to g$ weakly in $W^{1,2}(\Om) $ and $g\in W^*(\Om)$. So, $g =f $ by uniqueness of a weak limit in $W^{1,2}(\Om)$. Mazur's lemma (see, e.g., \cite[Theorem~2, page 120]{Yo80}) implies that $g$ is a limit of  convex sums  of $\{f_j\}$ in $W^*$-norm. This is impossible due to the previous proposition. 
\end{proof}

\subsection{Bedford-Taylor capacity for plurisubharmonic functions}
\label{ss:BT-cap}
Recall that for a Borel set $E\subset \Om$,
\[\label{eq:BT-cap}
 cap(E,\Om) = \sup\left\{ \int_E (dd^c u)^n :  u\in PSH(\Om), \;-1 \leq u \leq 0\right\}.
\]
We know   $cap (\cdot,\Om)$ is both inner and outer regular by a classical result in \cite{BT82}. The capacity of a given Borel set is realizable by its relative extremal function. Namely,  assume $\Om$ is a strictly pseudoconvex and  $E\subset \Om$ a Borel set. Define
\[\label{eq:rel-ext-fct}
	u_E(z) = \sup\left\{ u(z): u\in PSH(\Om), \; u \leq -1 \text{ on } E,\; u \leq 0 \right\}
\]
and $u_E^*$ is its upper semicontinuous regularization. Then,
\[\label{eq:BT-cap-id}
	cap(E, \Om) = \int_\Om (dd^c u_E^*)^n.
\] 
We refer the reader to \cite[Proposition~6.5]{BT82} for a general formula where the left hand side is the outer capacity applied for the sets  which are  not necessarily Borel.

We also need the comparison between total mass integrals of 
a closed positive $(1,1)$-current. This is a consequence of the weak convergence of measures.  We state it in the simplest form.

\begin{lem} \label{lem:CP-mass} Let $0\leq k\leq n-1$.  Let $T$ be a closed positive $(1,1)$-current in an open subset $\Om \subset \subset \bC^n$. 
Assume $\vphi,\psi \in PSH \cap L^\infty(\Om)$ satisfy  $\vphi \leq \psi \leq 0$ in $\Om$ and $\lim_{z\to \d\Om} \vphi(z) =0$. Then, 
\[\label{eq:CP-mass}
	\int_\Om T  \wed (dd^c \psi)^k \wed \om^{n-k-1} \leq \int_\Om T  \wed (dd^c \vphi)^k \wed \om^{n-k-1}. 
\]
\end{lem}

\begin{proof}
In fact, for $\veps>0$,  $\max\{\vphi+\veps, \psi\} = \vphi+\veps$ near $\d\Om$. Hence, by integration by parts,
$$
	\int_\Om (dd^c \vphi)^k \wed T \wed \om^{n-k-1} = \int_\Om (dd^c\max\{\vphi+\veps, \psi\})^k \wed T \wed \om^{n-k-1}.
$$
Let $\veps \to 0$ and observe that the integrand on the right hand side converges weakly to $(dd^c \psi)^k \wed T\wed \om^{n-k-1}$ by  Demailly \cite[Theorem~1.7]{De89}. Hence, the desired inequality follows.
\end{proof}

\section{Basic integral inequalities}
\label{sec:int}
Throughout this section we assume $\Om \subset \subset \bC^n$  is a  strictly pseudoconvex domain. Let $\rho$ be a strictly psh  defining function of $\Om$ which is smooth in a neighborhood of $\bar\Om$ and satisfies
$$ 
	dd^c \rho \geq \om \quad \text{on }\bar \Om. 
$$
Let $-1 \leq \phi \leq 0$ be a psh function in $\Om$. We define
\[\notag
	\ga:= (dd^c\phi)^{n-1}.
\]
Let us also consider a non-negative smooth cut-off function
\[\notag
	\chi \in C^\infty_c(\Om), \quad \chi \geq 0.
\]
We have a list of results for the estimates of the $L^2$-norm of a smooth function with respect to the Monge-Amp\`ere measure $(dd^c \phi)^n$ and $d\chi \wed d^c \chi \wed \ga$ as follows.

\begin{lem}\label{lem:basic-grad} Let $ g \in C_c^\infty(\Om,\bR)$. Then,
$$
	\int g^2 (dd^c \phi)^n \leq 10 \int dg \wed d^c g\wed (dd^c\phi)^{n-1}.
$$
\end{lem}

\begin{proof} Since $(dd^c\phi *\eta_\veps)^n$ converges weakly to $(dd^c\phi)^n$ as $\veps\to 0$, we may assume that $\phi$ is smooth in $\Om$. By integration by parts,
\[\label{eq:IBP1}
	\frac{1}{2}\int g^2 (dd^c\phi)^n = -\int g dg \wed d^c \phi \wed \ga. \\
\]
The Cauchy-Schwarz inequality gives
$$\begin{aligned}	
 -\int g dg \wed d^c \phi \wed \ga   &\leq  \left( \int dg \wed d^c g \wed \ga \right)^\frac{1}{2} \left( \int g^2 d\phi\wed d^c\phi \wed\ga \right)^\frac{1}{2} \\
&	\leq 2\int dg \wed d^c g \wed \ga + \frac{1}{8} \int g^2 d\phi \wed d^c\phi \wed\ga.
\end{aligned}$$
Another integration by parts gives
$$\begin{aligned}
	\frac{1}{8}\int g^2 d\phi \wed d^c\phi \wed\ga &= -\frac{1}{4}\int g \phi dg\wed d^c\phi \wed\ga - \frac{1}{8}\int g^2 \phi (dd^c\phi)^n \\
&\leq		\frac{1}{4}\left| \int g \phi dg\wed d^c\phi \wed\ga \right| + \frac{1}{8}\int g^2 (dd^c\phi)^n.
\end{aligned}$$
Using the Cauchy-Schwarz inequality once more for the first integral, we obtain
$$\begin{aligned}
\frac{1}{4}\left| \int g \phi dg\wed d^c\phi \wed\ga \right|  &\leq \frac{1}{4}\int \phi^2 dg \wed d^c g \wed \ga + \frac{1}{16}\int g^2 d\phi \wed d^c \phi \wed\ga \\
&\leq		\frac{1}{4} \int dg \wed d^c g \wed \ga +  \frac{1}{16}\int g^2 d\phi \wed d^c \phi \wed\ga.
\end{aligned}$$
Combining the last two inequalities we arrive at
$$
	\frac{1}{16}\int g^2 d\phi \wed d^c \phi \wed \ga   \leq   \frac{1}{4}   \int dg\wed d^c g\wed \ga+ \frac{1}{8}  \int g^2 (dd^c\phi)^n .
$$
Therefore,
$$
	 -\int g dg \wed d^c \phi \wed \ga   \leq  \frac{5}{2} \int dg \wed d^c g\wed \ga + \frac{1}{4} \int g^2 (dd^c\phi)^n. 
$$
This and the  identity \eqref{eq:IBP1} imply the desired inequality.
\end{proof}

\begin{lem}\label{lem:extra-grad} Let $f\in C^\infty(\Om)$ and $0\leq \chi \in C^\infty_c(\Om)$. Then
$$
	\int f^2 d\chi \wed d^c \chi \wed \ga \leq 4 \int \chi^2 df \wed d^c f \wed \ga +  2A_0\int \chi f^2 \om\wed\ga, 
$$
where $A_0$ is a positive constant such that $-A_0 \om \leq dd^c\chi \leq A_0\om$.
\end{lem}

\begin{proof} An integration by parts gives
$$\begin{aligned}
	\int f^2 d\chi \wed d^c\chi \wed\ga &= -2 \int f \chi df\wed d^c\chi \wed\ga - \int f^2 \chi dd^c\chi \wed \ga\\
&\leq		2 \left| \int f \chi df\wed d^c\chi \wed\ga \right| + A_0 \int f^2  \chi \om \wed\ga,
\end{aligned}$$
where we used the fact $dd^c \chi \leq A_0\om$.
Next, the  Cauchy-Schwarz inequality implies 
$$\begin{aligned}
2\left| \int f \chi df\wed d^c\chi \wed\ga \right|  &\leq 2 \int \chi^2 df \wed d^c f \wed \ga + \frac{1}{2}\int f^2 d\chi \wed d^c \chi \wed\ga. \\
\end{aligned}$$
Combining the last two inequalities we arrive at
$$
	\frac{1}{2}\int f^2 d\chi \wed d^c \chi \wed \ga   \leq  2 \int \chi^2 d f\wed d^c f\wed \ga + A_0 \int f^2  \chi \om \wed\ga.
$$
This is equivalent to our desired estimate.
\end{proof}

\begin{lem} \label{lem:L2-cap-grad} Let $f\in C^\infty(\Om)$ and let $0\leq \chi \in C^\infty_c(\Om)$. 
$$
	\int (\chi f)^2 (dd^c \phi)^n \leq  100 \int \chi^2 df \wed d^c f\wed (dd^c\phi)^{n-1} + 40A_0 \int \chi f^2 \om \wed (dd^c\phi)^{n-1}
$$
\end{lem}

\begin{proof} Applying Lemma~\ref{lem:basic-grad} for $g = \chi f$ we get
$$\begin{aligned}
	\int (\chi f)^2 (dd^c\phi)^n 
&\leq 10 \int d(\chi f) \wed d^c(\chi f) \wed \ga \\
&\leq 20 \int [ \chi^2 df \wed d^c f + f^2 d\chi \wed d^c \chi] \wed \ga.
\end{aligned}$$
Now the desired conclusion follows by applying Lemma~\ref{lem:extra-grad} for the second term in the last integral.
\end{proof}

In other words, in order to get the inequality in Lemma~\ref{lem:L2-cap-grad} we  replace the integrand $(\chi f)^2 dd^c\phi$ by $100\chi^2 df \wed d^c f + 40A_0 \chi f^2 \om$.
We can continue this process  for the second integrals on the right hand side with $
\chi := (\chi_1)^2$ where $\chi_1$ is again a cut-off function. After repeating this process $n$-times we obtain easily 

\begin{cor}\label{cor:L2-cap-grad}  Let $f\in C^\infty(\Om)$ and $0\leq \chi \in C^\infty_c(\Om)$. Then,
$$\begin{aligned}
	\int \chi^{2^{n}} f^2 (dd^c\phi)^n &\leq 100 \int df \wed d^c f \wed \left( \sum_{k=1}^{n-1} A_k \;\chi^{2^{k}}  (dd^c \phi)^{k} \wed \om^{n-1-k} \right) \\ &\quad + (40A_0)^n \int \chi f^2 \om^n.
\end{aligned}$$
where $A_k =[40(1+ A_0)]^{n-1-k}$.
\end{cor}

Note that the inequality of the corollary holds for $-1 \leq \phi \leq 0$ being a general bounded psh function and $f$ is smooth. However, its right hand side involves only gradient of $f$ and its $L^2$-norm, we can  use a simple approximation argument to  exchange the regularity assumptions on $\phi$ and $f$ as follows.

\begin{cor}\label{cor:L2-cap-general} Let $f\in W^{1,2}(\Om)$. Assume further $\phi$ is smooth in a neighborhood of $\supp\chi$.  Then,
$$\begin{aligned}
	\int \chi^{2^{n}} f^2 (dd^c\phi)^n &\leq 100 \int df \wed d^c f \wed \left( \sum_{k=1}^{n-1} A_k \;\chi^{2^{k}}  (dd^c \phi)^{k} \wed \om^{n-1-k} \right) \\ &\quad + (40A_0)^n \int \chi f^2 \om^n.
\end{aligned}$$
where $A_k =[40(1+ A_0)]^{n-1-k}$.
\end{cor}

We are ready to state the main application of the above estimates. This strengthens \cite[Lemma~2.5]{DMV} which says roughly that continuous functions in $W^*(\Om)$ is locally uniformly $L^2$-integrable with respect to  Monge-Amp\`ere measures of bounded psh functions. More precisely, 

\begin{prop} \label{prop:L2-cap} Let $f\in W^*(\Om) \cap C^0(\Om)$. Let $D\subset \subset \Om$ be a subdomain. Then, for every  Borel set $K\subset D$,
$$
	\int_K  f^2 (dd^c \phi)^n \leq A \|f\|_*^2,
$$
where $A$ depends only on $D, \Om$.
\end{prop}

\begin{proof} We take $0\leq \chi \leq 1$ a smooth cut-off function such that $\chi \equiv 1$ on $D$ and $\supp \chi \subset\subset \Om$.  It is enough to show that
$$	J:=\int \chi^{2^n}  f^2 (dd^c \phi)^n \leq A\|f\|_*^2.
$$
By weak convergence of $[dd^c (\phi*\eta_\veps)]^n$ to $(dd^c\phi)^n$ and the continuity of $f$, we may assume that $\phi$ is smooth. Since $\phi$ is bounded we can modify it outside the support of $\chi$ by considering $\max\{\phi, A_0 \rho\}$, where $A_0$ depends only on $\rho$. So, $\phi \geq A_0 \rho$ and $\phi= A_0\rho$ near $\d\Om$. Moreover, we assume further that $f$ is smooth by the smoothing property in Proposition~\ref{prop:smoothing}.

Now, let $T$ be a closed positive current such that $df \wed d^c f\leq T$. 
Corollary~\ref{cor:L2-cap-grad} implies that 
$$
	J \leq A_1 \int T \wed \sum_{k=1}^{n-1} (dd^c\phi)^k \wed \om^{n-1-k} + A_1 \|f\|_{L^2(\Om)}^2.
$$
Since $A_0 \rho \leq \phi \leq 0$ in $\Om$, the mass comparison  (Lemma~\ref{lem:CP-mass}) implies that
$$\begin{aligned}
	\int T \wed (dd^c\phi)^k \wed \om^{n-1-k} 
&\leq \int T \wed [dd^c (A_0\rho)]^k \wed \om^{n-1-k} \\
&\leq A_2 \int T\wed \om^{n-1}.
\end{aligned}$$
Therefore, $J \leq A_1A_2 (\|f\|^2_{L^2(\Om)} + \|T\|_\Om)$. Taking infimum over all $T$ in $\Ga_f$ we complete the proof.
\end{proof}

\begin{remark}
At the moment the continuity of $f$ is needed in Proposition~\ref{prop:L2-cap} so that the integral on the left hand side is well-defined. However, we will show in Theorem~\ref{thm:quasi-mod} below that $f$ admits a quasi-continuous representative $\wt f \in W^*(\Om)$ with respect to the Bedford-Taylor capacity. So, the integral can be defined  for a general $f$ in Definition~\ref{defn:L2-ma}. 
Furthermore, in Section~\ref{sec:dirichlet}, we will see that all estimates in this section hold for a general $f\in W^*(\Om)$ and a bounded psh function $\phi$.
\end{remark}

Next result is the quantitative version of second part in \cite[Lemma~2.5]{DMV}.

\begin{lem} \label{lem:wed-prod} Let $f\in W^*(\Om) \cap C^0(\Om)$ be such that $|f| \leq 1$. Let $\chi\in C^\infty_c(\Om)$ be such that  $0\leq \chi \leq 1$. Assume $df \wed d^c f \leq dd^c u$ for some $u \in PSH\cap L^\infty(\Om)$. Then,
$$
	\int \chi^{2^n} f^2 (dd^c\phi)^n \leq  A \left[ \int \chi f^2 (\om + dd^c u)^n \right]^\frac{1}{2^n},
$$
where $A$ is a uniform constant depending on $\supp \chi$, $\Om$ and $\|u\|_{L^\infty}$. More precisely, 
$$ 
	A = A_0 \|u\|_{L^\infty}^{a}, \quad \text{where }  
	a = \sum_{i=1}^{n-1} \frac{i-1}{2^{i-1}}
$$
and $A_0$ is a constant depending only on $\chi, \om$.
\end{lem}

\begin{proof} Recall that $\ga= (dd^c\phi)^{n-1}$ and by replacing $\phi :=\phi+1$ we have $0\leq \phi \leq 1$. By weak convergence of approximation sequences we may assume first $\phi$ is smooth. Define $f_\veps = f*\eta_\veps$ and $u_\veps= u*\eta_\veps$.  Next, invoking Proposition~\ref{prop:ineq-weak}, we get
$$
	d f_\veps \wed d^c f_\veps \leq dd^c u_\veps.
$$
Since $f_\veps \to f$ uniformly on $\supp \chi$, we have $f_\veps^2 (\om+dd^c u_\veps)^n \to f^2 (\om+dd^cu)^n$ weakly in $\Om$.  Therefore, it is enough to prove the inequality for $\phi, u$ and $f$ which are smooth. Its proof is done via using repeatedly the integration by parts and the Cauchy-Schwarz inequality. The first step  is as follows:
\[\label{eq:L2-1st-a}\begin{aligned}
	\int (\chi f)^2 dd^c\phi \wed\ga
&=-2 \int \chi f d(\chi f) \wed d^c \phi \wed\ga \\
&\leq  2\left(\int f^2 d(\chi f) \wed d^c(\chi f) \wed \ga \right)^\frac{1}{2}\left( \int \chi^2 d\phi \wed d^c\phi \wed \ga\right)^\frac{1}{2}.
\end{aligned}\]
The second integral on the right hand side  is bounded by
$$	
	 \int \chi^2 d\phi \wed d^c\phi \wed \ga \leq \int \chi^2 dd^c\phi^2 \wed \ga \leq A_1',
$$
 where we used the Chern-Levine-Nirenberg (CLN) inequality in the last  inequality for    a constant $A_1'>0$ depending only on support of $\supp \chi$. Next, for the other integrand of \eqref{eq:L2-1st-a},
$$\begin{aligned}
	f^2 d(\chi f) \wed d^c (\chi f) 
&\leq 2 f^2 (\chi^2 df \wed d^c f + f^2 d\chi \wed d^c\chi) \\
&\leq 2 f^2[\chi^2 dd^c u +2 d\chi \wed d^c\chi],
\end{aligned}$$
where we used the fact $|f|\leq 1$ for the second term in the bracket. Integrating both sides and 
using Lemma~\ref{lem:extra-grad} for the second term in the bracket  we derive
\[\label{eq:L2-1st-b}\begin{aligned}
	\int f^2 [\chi^2 dd^c\phi ]\wed\ga  \leq A_1 \left( \int  f^2 [ \chi^2  dd^c u + A_0 \chi \om]  \wed \ga\right)^\frac{1}{2},
\end{aligned}\]
where $A_1 = 2\sqrt{18} A_1'$.
Thus, we have been replacing one factor $\chi^2 dd^c \phi$ on the left hand side of \eqref{eq:L2-1st-b} by the factor $\chi^2 dd^c u + A_0 \chi \om $ on the right hand side where a square root is taken on the integral and multiplied with the uniform constant $A_1$.

We can repeat this replacing $\chi^2 dd^c \phi$ process for each  integral obtained on the right hand side. In order to keep the exponent of $\chi$ at least one in the final integrals without the factor $dd^c \phi$, we need to start with  $\chi^{2^n}$ at the beginning. Clearly, the constants possibly increase at each step.  

Namely, together with an elementary inequality $(x+y)^k \leq 2^k (x^k + y^k)$ for $x,y\geq 0$ and a positive integer $k$, we can rewrite the sequence of inequalities as follows:
\[\notag\begin{aligned}
	\left(\int f^2 [\chi^{2^n} dd^c\phi] \wed \ga\right)^{2^n} 
&\leq A_1 \left( \int f^2 \chi^{2^{n}} dd^c u  \wed \ga \right)^{2^{n-1}}  \\
&\quad + A_1 \left( \int f^2 \chi^{2^{n-1}} \om \wed \ga \right)^{2^{n-1}} \\
&\leq \cdots  \\
&\leq A_n  \sum_{k=1}^n  \int f^2 \chi^{2^{k}} (dd^cu)^k \wed \om^{n-k}. 
\end{aligned}\]
This is the desired estimate once we take into account the fact  $0 \leq \chi^{2^k} \leq \chi$ for every integer $1\leq k\leq n$ as $0\leq \chi \leq 1$.
Here the constants $A_k$ comes from the use of the CLN inequality
$$
	\int \chi dd^c\phi^2 \wed (dd^c u)^k \wed \om^{\ell} \wed (dd^c\phi)^{n-k-\ell} \leq A_k \|u\|_{L^\infty}^k,
$$
and they depend only on $\chi$ and $\Om$.
\end{proof}

\begin{lem} \label{lem:L1-cap-comparison} Let $v, w \in PSH \cap L^\infty (\Om)$. Let $0\leq f\in W^*(\Om) \cap C^0(\Om)$. Then,
$$\begin{aligned}
	\int \chi^2 f (\om+ dd^c w)^n &\leq \int \chi^2 f  (\om+ dd^c v)^n \\ 
&\quad +\left( \int \chi d(v-w)\wed d^c (v-w) \wed \tau \right)^\frac{1}{2}  (I^\frac{1}{2} + J^\frac{1}{2}) ,
\end{aligned}$$
where 
$$
\begin{aligned}
\tau &= \sum_{k=1}^{n-1} (\om+ dd^cv)^k \wed (\om+ dd^cw)^{n-1-k},\\
 I &= \int \chi f^2 d\chi \wed d^c \chi \wed \tau,\quad 
J = \int \chi^2 df \wed d^cf \wed \tau.
\end{aligned}$$ 
\end{lem}

\begin{proof} We write $(\om+ dd^cw)^n - (\om+ dd^c v)^n =dd^c(w-v)\wed \tau$ and 
denote $h=w-v$. By integration by parts
$$
	\int \chi^2 f dd^c h \wed \tau = - \int [2\chi f d\chi + \chi^2  df] \wed d^c h \wed \tau.
$$
We will use the Cauchy-Schwarz inequality for each integral on the right hand side. For the first one we get
$$\begin{aligned}
	\left( \int \chi fd\chi \wed d^c h \wed \tau \right)^2 \leq \int \chi f^2 d\chi \wed d^c \chi \wed \tau  \int \chi  dh\wed d^c h \wed \tau.
\end{aligned}$$
The second integral on the right hand side is bounded by
$$\begin{aligned}
\left( \int \chi^2  df \wed d^c h \wed \tau \right)^2 
&\leq \int \chi^2 df\wed d^c f \wed \tau \int \chi^2  dh \wed d^c h\wed \tau\\
\end{aligned}$$
Combining these inequalities and the fact that $0\leq \chi \leq 1$ we get the proof of the lemma.
\end{proof}

Applying the localization principle we have 

\begin{cor}\label{cor:L1-cap-comparison-a} Let $v, w, f$ and $\tau$ be  as in Lemma~\ref{lem:L1-cap-comparison}. Assume that $w\geq v$ in a neighborhood $D \subset \subset \Om$ of $\supp\chi$. Then,
$$\begin{aligned}
	&\left|\int \chi^2 f (\om+ dd^c w)^n - \int \chi^2 f  (\om+ dd^c v)^n \right| \\ 
&\quad\quad\leq   A \|f\|_* \left( \int_D  (w-v)(\om+dd^c v)^n \right)^\frac{1}{2},
\end{aligned}$$
where $A$ is a uniform constant depending on $\|v\|_{L^\infty}$, $\supp\chi$ and $D$.
\end{cor}

\begin{proof} By subtracting the same constant  from $v,w$ we may assume that $v\leq w \leq 0$. The integrals on the left hand side  do not change if we modify $v,w$ outside  $\supp\chi$. Therefore, replacing $v, w$ by $\max\{v, A_0\rho\}$ and $\max\{w, A_0 \rho\}$ we may assume that they are equal $A_0\rho$ outside $\bar D$, in particular, they are equal on  a neighborhood of $\d \Om$. Now, applying Lemma~\ref{lem:L1-cap-comparison} we obtain 
$$\begin{aligned}
&\left| \int \chi^2 f (\om+ dd^c w)^n -\int \chi^2 f (\om+ dd^c v)^n \right| \\ 
&\quad\quad \leq \left( \int \chi d(v-w)\wed d^c (v-w) \wed \tau \right)^\frac{1}{2}  (I^\frac{1}{2} + J^\frac{1}{2}).
\end{aligned}$$
By integration by parts and $w \geq v$ in $\Om$  we have 
$$\begin{aligned}
 \int  d(v-w)\wed d^c (v-w) \wed \tau &= \int (w -v) 	(\om_v^n - \om_w^n)\\ &\leq \int (w - v) \om_v^n.
\end{aligned}$$
It remains to estimate $I$ and $J$ which is very similar to the proof of Proposition~\ref{prop:L2-cap}. In fact, let $M = 2(2\|v\|_{L^\infty} + \|\rho\|_{L^\infty})$ and define $$\phi := \frac{v+w + \rho}{M}.$$
We have $\tau \leq M^{n-1} (dd^c\phi)^{n-1}$.
If $T\in \Ga_f$, then 

$$J \leq \int T \wed \tau \leq  M^{n-1}\int T \wed (dd^c\phi)^{n-1}.$$ 
As $A_0\rho \leq \phi \leq 0$, the mass comparison (Lemma~\ref{lem:CP-mass}) implies
$$	
 \int T \wed (dd^c\phi)^{n-1} \leq \int T\wed (dd^c A_0\rho)^{n-1} \leq A_1 \|T\|_\Om.
$$  
Taking infimum over $T\in \Ga_f$ we get $$J \leq M^{n-1} A_1 \|f\|_*^2.$$ 
To estimate $I$ we observe that $d\chi \wed d^c\chi \leq A_0 dd^c \rho$ and thus
$$
	d\chi \wed d^c \chi \wed \tau \leq  A_0 dd^c \rho \wed \tau \leq  A_0M^{n} (dd^c \phi)^{n}.
$$
Thus, by Proposition~\ref{prop:L2-cap},
$$
	I \leq A_2 \int \chi f^2 (dd^c\phi)^n \leq A_3 \|f\|_*^2.
$$
The proof is completed.
\end{proof}

\begin{cor}\label{cor:L1-cap-comparison-b} Let $v\in PSH(\Om)\cap L^\infty(\Om)$ be such that $v\leq 0$. Assume $\supp\chi \subset \subset D \subset \subset \Om$ and denote $v_j = v*\eta_j$. Then, for $0\leq f \in W^*(\Om)\cap C^0(\Om)$,
$$
	\left| \int \chi^2 f (\om+ dd^c v_j)^n -\int \chi^2 f (\om+ dd^c v)^n \right| \leq  A \|f\|_* \left[\int_{D} (v_j -v) (\om+ dd^c v)^n\right]^\frac{1}{2}
$$
where $A$ is a uniform constant depending on $\|v\|_{L^\infty}, \supp\chi$ and $D$. 
\end{cor}

\begin{proof} It is a consequence of Corollary~\ref{cor:L1-cap-comparison-a} for $w= v_j$.
\end{proof}

\section{A functional capacity}

We start with basic definitions and properties of capacity. Let $\Om \subset \subset \bC^n$ be an open subset.   Let $E \subset \Om$ be a Borel set. Define
 \[\label{eq:ke-cap}
 	\Kc(E) = \left\{ v \in W^*(\Om) \rv \{v \leq  -1\}^o \supset E \text{ and } v\leq 0 \right\}.
\]
If $E=G$ is an open set, then its interior $E^o = E$ and 
\[\notag
		\Kc(G) = \left\{ v \in W^*(\Om) \;\rvert\; v \leq  -1 \text{ in $G$ and } v\leq 0 \right\}.
\]
We define the capacity by
\[\label{eq:cap-defn}
	{\tt c}(E) = \inf 
	\left\{ \| v \|_*^2 \;\rvert\;  v\in \Kc(E)
	\right\}.
\] 
In the definition of capacity both $\Kc(E)$ and $\tc(\cdot)$ depends on $\Om$. However, in most cases the domain $\Om$ is fixed, then for simplicity we write
$$
		\tc(E) := \tc(E,\Om).
$$
Note that if $v \in \Kc(E)$, then $\max\{v,-1\}\in \Kc(E)$ and$ \|\max\{v, -1\} \|_* \leq \|v\|_*$ by  \eqref{eq:norm-max-min}. Therefore, we may consider candidates  in $\Kc(E)$ such that 
\[ \label{eq:reduction} 
	-1\leq v \leq 0, \quad \text{and } v=-1  \text{ on an open set } G\supset E.
\] Moreover, we have for a compact set $K \subset \Om$, 
\[ \label{eq:c-compact} \tc(K) <+\infty.
\] In fact, let $\chi$ be a cut-off function which is 1 on $K$ and $0$ outside a neighborhood of $K$. Then, $-\chi \in W^*(\Om)$ because it is a smooth function. Thus,
$$
	\tc (K) \leq \| \chi \|_*.
$$

We list basic properties whose proofs are rather standard (see e.g. \cite{AH96}). However, we include the details for the reader convenience. 

\begin{prop}\label{prop:basic} Let $E, F$ and  $E_j$, $j\geq 1$, be Bore sets in $\Om$.
\begin{itemize}
\item
[(a)] If $E \subset F \subset \Om' \subset \Om$, then $\tc(E) \leq 
\tc(F)$ and
$\tc(E, \Om') \leq \tc (E,\Om)$ where $\tc(E,\Om')$ is  defined with respect the subdomain $\Om'$.
\item
[(b)] Let $\{E_j\}_{j\geq1} $ is an increasing sequence and $E= \cup_{j\geq 1} E_j$. Then $\lim_j \tc(E_j) \leq \tc(E)$.
\item
[(c)] For every sequence $\{E_j\}_{j\geq 1}$, $E = \cup_{j=1}^\infty E_j$, we have $\tc(E) \leq \sum_{j=1}^\infty \tc(E_j)$.
\item
[(d)] Let $K_j$ be a decreasing sequence of compact sets and $K = \cap_j K_j$. Then, $\tc(K) = \lim_j \tc(K_j)$.
\end{itemize}
\end{prop}

\begin{proof} (a) follows easily from the definition of capacity \eqref{eq:cap-defn} and (b) is a consequence of (a). To prove the subadditivity property (c) we may assume that $A:= \sum_{j} \tc (E_j)<+\infty$, otherwise it trivially holds.  Let $\veps>0$ be fixed. For each $j$, there exists $-1\leq v_j \leq 0$ in $\Kc(E_j)$ such that  $\tc (E_j) \geq \|v_j\|_*^2 -2^{-j}\veps$ and $v_j = -1$ a.e on an open set $G_j \supset E_j$. Put $u_k = \min_{j\leq k} v_j$. The sequence $u_k$ is decreasing. It follows from Lemma~\ref{lem:norm-min-max} that
$$ \|u_k\|_*^2 \leq \sum_{j=1}^k \|v_j\|_*^2 \leq \sum_{j=1}^\infty \tc (E_j) -\veps = A+\veps.$$
By Corollary~\ref{cor:w-compactness} there exists a weak limit $u\in W^*(\Om)$ of $\{u_k\}_{k\geq 1}$ satisfying $\|u\|_*^2 \leq A+\veps$. For a fixed $j$, we have  ${\bf 1}_{G_j} \cdot u_k \to {\bf 1}_{G_j} \cdot u$ weakly in $L^2(\Om)$ as $k\to \infty$.  It follows from Proposition~\ref{prop:weak-conv-cp} that $u \leq -1$ in $G_j$ for $j=1,2...$. Then, $u \leq -1$ on $G:= \cup_j G_j \supset E$ and $u \in \Kc(E)$.
Hence, $u \in \Kc(E)$ and $\tc (E)\leq \|u\|_*^2 \leq A+\veps$. Since $\veps>0$ is arbitrary the proof of (c) follows. 

Finally we prove (d). By the monotonicity in (a), $\tc(K) \leq \lim_{j} \tc(K_j)$. Let $\veps>0$. There exists $v\in \Kc(K)$ such that $v\leq -1$ on an open set $G\supset K$ and $\tc (K) \geq \|v\|_*^2 -\veps$. Since $K_j$ be compact sets, $K_j \subset G$ for $j \geq 1$ large enough. Thus, $v \in \Kc(K_j)$ and $\tc (K_j) \leq \|v\|_* \leq \tc(K) +\veps$. In other words, $\lim_j \tc (K_j) \leq \tc (K) +\veps$. Let $\veps \to 0$, the proof of (d) follows.
\end{proof}

\begin{remark} \label{rmk:c-equiv} \mbox{}
\begin{enumerate}
\item[(a)] Let $E \subset E'$ be Borel sets in $\Om$. If $V_{2n}(E'\setminus E)=0$, then $\tc(E) = \tc(E')$. It is  non-decreasing under holomorphic maps, i.e., if  $h: G \to \Om$ is a holomorphic map, then
$$
	\tc (E,G) \leq \tc (h(E), \Om). 
$$
\item[(b)] Let $D\subset\subset \Om' \subset \Om$. 
We do not know if  there exists a constant $A>0$ depending only on $\Om$ and $\Om'$ such that  for every Borel set $ E\subset D$, 
\[\label{eq:c-comparablity}
	 \tc(E,\Om) \leq A\; \tc(E, \Om').
\]
However, we will have a weaker version with some exponent on the right hand side later 
in Corollary~\ref{cor:dom-equiv-c} for the new capacity $\tc (\cdot, \Om)$.
\end{enumerate}
\end{remark}

One of our main goals is to show that $\tc (\cdot, \Om)$ is a capacity in the Choquet sense. The inner regularity is harder which will be proved in Theorem~\ref{thm:choquet-c}, however the outer regularity is readily obtained as follows. 

\begin{lem}[outer regularity] \label{lem:outer-reg} For a Borel subset $E\subset \Om$,
\[\notag	\tc(E) = \inf\{ \tc(G) : E \subset G,\; G \text{ is open}\}.
\]
\end{lem}

\begin{proof} By the monotonicity in Proposition~\ref{prop:basic}-(a), 
$$
	\tc(E) \leq \inf\{ \tc(G) : E \subset G,\; G \text{ is open}\}.
$$
To prove the reverse inequality we may assume $\Kc(E)$ is non empty. Let $\veps>0$. There exist $v\in \Kc(E)$ with 
$
	\|v\|_*^2 \leq \tc(E) +\veps
$
and  an open set $G_\veps$ such that $v(x)\leq - 1$ a.e $x\in G_\veps$. Therefore, 
$$
	\inf\{ \tc(G) : E \subset G,\; G \text{ is open}\} \leq \tc(G_\veps) \leq \|v\|_*^2 \leq \tc(E) +\veps.
$$
The lemma follows by letting $\veps\to 0$.
\end{proof}

A similar argument shows the Lebesgue volume is dominated by the capacity.

\begin{cor}\label{cor:vol-c} For a Borel set $E\subset \Om$, 
$
	|E| \leq \tc(E).
$ In particular, every set of capacity zero has Lebesgue measure zero.
\end{cor}

\begin{proof} We may assume $\tc(E)<\infty$, otherwise, the inequality is true. Let $\veps>0$ and let $v\in \Kc(E)$ with $\|v\|_*^2 \leq \tc(E) +\veps$.
There is an open set $G\supset E$ such that $v(x)\leq -1$ for a.e $x\in G$. This implies
$$
	|E| \leq |G| \leq \int_G |v|^2 dx \leq \|v\|_{L^2(\Om)}^2 \leq \|v\|_*^2 \leq \tc(E)+\veps.
$$
The desired inequality follows by letting $\veps\to 0$.
\end{proof}

We have seen in above examples that bounded psh functions belong to $W^*(\Om)$. For such functions the Bedford-Taylor capacity $cap(\cdot,\Om)$ in Section~\ref{ss:BT-cap} is very  suitable for studying their fine properties. Dinh, Marinescu and Vu \cite{DMV} also used $cap(\cdot,\Om)$ to study some fine properties of functions in $W^*(\Om)$. In fact, 
we are going to show that the two capacities $\tc (\cdot,\Om)$ and $cap(\cdot,\Om)$ are  comparable to each other.   

The first part of Theorem~\ref{thm:intro-cap-comp}-(a) reads as follows.

\begin{lem}[dominated by Bedford-Taylor capacity]  \label{lem:c-cap} Assume $\Om$ is strictly pseudoconvex. There exists a constant $A>0$ such that for every Borel set $E\subset\subset \Om$,
$$
	\tc(E) \leq A \left[cap(E,\Om)\right]^\frac{1}{n}.
$$
\end{lem}

\begin{remark}\label{rmk:inner-reg} The relative compactness condition of $E$ will be removed later after proving that $\tc(\cdot,\Om)$ is a capacity in the sense of Choquet (Theorem~\ref{thm:choquet-c}).
\end{remark}

\begin{proof} Let  $u_E^*$ denote  the relative extremal for $E$ in \eqref{eq:rel-ext-fct}.
Since both capacities are outer regular (Lemma~\ref{lem:outer-reg} and \cite{BT82}), we may assume that $E=G$ is an open set and hence $u_G = u_G^*$. Then, $-1\leq u_G \leq 0$ and $u_G =-1$ on $G$.  It follows that
\[\label{eq:BT-dom-a}
	\tc(G) \leq \|u_G\|_*^2.
\]
Let $\rho$ be a strictly psh defining function for $\Om$. We may assume that $dd^c \rho \geq \om$. Since $G\subset \subset \Om$, it implies that $u_G =0$ on $\d\Om$ (see, e.g., \cite[Properties (13.11)]{De89}). Using  $-1 \leq u_G \leq 0$ and integration by parts,
$$
	\|u_G\|_{L^2}^2 = \int  (u_G)^2 (dd^c \rho)^n \leq \int - u_G (dd^c\rho)^n = \int -\rho dd^c u_G \wed (dd^c\rho)^{n-1}.
$$
Here and in the proof below we skip writing $\Om$ in the integral symbols 
for simplicity. The last integral is less than 
$$
	\|\rho\|_{L^\infty(\Om)} \int dd^c u_G \wed (dd^c\rho)^{n-1}.
$$
Since $0\leq u_G+1 \in PSH(\Om)$, we have
$d u_G \wed d^c u_G \leq \frac{1}{2} dd^c (1+u_G)^2.$
Hence, 
\[\label{eq:BT-dom-b}\begin{aligned}
	\|u_G\|_*^2 
&\leq \int (u_G)^2 (dd^c \rho)^n + \frac{1}{2}\int dd^c (1+u_G)^2 \wed (dd^c \rho)^{n-1} \\
&\leq A_1 \int dd^c u_G \wed (dd^c\rho)^{n-1} + \int dd^c u_G \wed (dd^c \rho)^{n-1},
\end{aligned}\]
where $A_1 = \|\rho\|_{L^\infty(\Om)}$ and we used the fact that 
$$\begin{aligned}
&	\frac{1}{2}\int dd^c (1+u_G)^2 \wed (dd^c \rho)^{n-1} \\
&=\int du_G \wed d^c u_G \wed (dd^c\rho)^{n-1} + \int (1+u_G) dd^c u_G \wed (dd^c \rho)^{n-1} \\
	&=	\int dd^c u_G \wed (dd^c \rho)^{n-1}.
\end{aligned}
$$
Invoking an inequality of Cegrell  \cite{Ce04} we have
\[\label{eq:BT-dom-c}\begin{aligned}
	\int dd^c u_G \wed (dd^c \rho)^{n-1}& \leq \left[\int (dd^c\rho)^n\right]^\frac{n-1}{n} \left[\int (dd^c u_G)^n\right]^\frac{1}{n} \\ &=  A_0 [cap(G,\Om)]^\frac{1}{n},
\end{aligned}\]
where $A_0$ is a uniform constant as $\rho$ is smooth on $\bar\Om$ and we used the formula \eqref{eq:BT-cap-id} for the last identity.

Combining \eqref{eq:BT-dom-a}, \eqref{eq:BT-dom-b} and \eqref{eq:BT-dom-c} we get $\tc (G) \leq A_0(1 + A_1) [cap(G,\Om)]^\frac{1}{n}.
$
\end{proof}

We are ready to state Theorem~\ref{thm:intro-cap-comp}-(b) which is the reverse inequality for the two capacities.

\begin{lem} \label{lem:cap-c} Let $D \subset \subset \Om$. There exists  a constant $A'$ depending  only on $D$ and $\Om$ such that for every Borel set $E \subset D$,
$$	cap(E,\Om) \leq A'\, \tc(E)
$$
\end{lem}

\begin{proof} Since both capacities are outer regular, 
again it is enough to prove the inequality for $E=G$ an open subset in $D$. Let $u= u_G=u_G^*$ be the relative extremal function of $G$ in $\Om$ and consider $u:= \max\{u, A_0 \rho\}$ we may assume $u= u_G$ on $\Om$ and it is defined in a neighborhood of $\bar\Om$. We can furthermore increase $A_0$ (depending on $D,\Om$) such that 
\[\label{eq:capc-rho} A_0\rho \leq u_G \quad \text{in }\Om.
\]
Denote $u_j:= u*\eta_j$ the convolution with the smooth kernel family in Section~\ref{ss:approx}. These functions  are defined on $\Om$ for $j$  large enough.

Let $f \in \Kc(G)$ be such that $f \leq -1$ a.e. on $G$ and assume also that $-1\leq f\leq 0$ in $\Om$.  Let $0\leq \chi \leq 1$ belong to $C^\infty_c(\Om)$  such that $\chi =1$ on $D$ and $\supp \chi \subset\subset \Om$. Since $(dd^c u_j)^n$ converges  to $(dd^c u)^n$ weakly, the formula \eqref{eq:BT-cap-id} implies that
\[\label{eq:capc-dominate}
	cap(G,\Om) \leq \liminf_{j\to \infty} \int_G (dd^c u_j)^n \leq \liminf_{j\to \infty} \int_D  \chi^{2^n} f^2 (dd^c u_j)^n.
\]

Fix an index $j$, thus $u_j$ smooth. Applying Corollary~\ref{cor:L2-cap-general} for $\phi = u_j$ to get
$$\begin{aligned}
	\int \chi^{2^{n}} f^2 (dd^cu_j)^n &\leq \int df \wed d^c f \wed \left( \sum_{k=1}^{n-1} 100A_k \;\chi^{2^{k}}  (dd^c u_j)^{k} \wed \om^{n-1-k} \right) \\ &\quad + (40A_0)^n \int \chi f^2 \om^n,
\end{aligned}$$
where $A_k =[40(1+ A_0)]^{n-1-k}$. 
If $T \in \Ga_f$, i.e., $d f \wed d^c f \leq T$ in $\Om$, then 
$$\begin{aligned}
	\int \chi^{2^{n}} f^2 (dd^cu_j)^n &\leq \int T \wed \left( \sum_{k=1}^{n-1} 100A_k \;\chi^{2^{k}}  (dd^c u_j)^{k} \wed \om^{n-1-k} \right) \\ &\quad + (40A_0)^n \int \chi f^2 \om^n \\
	&=: J_1 + J_2.
\end{aligned}$$
We first consider $J_2$. Clearly, 
\[\label{eq:capc-j2}
	J_2
 \leq A_1' \|f\|_{L^2(\Om)}^2.
\]
Next we estimate $J_1 = J_1 (j)$. 
  By letting $j\to +\infty$, the weak convergence  implies
$$\begin{aligned}
	\limsup_{j\to\infty} J_1 &\leq A_0'\int\chi T \wed \sum_{k=1}^{n-1} (dd^cu)^{k}\wed\om^{n-1-k} \\&\leq A_0'\int_\Om T\wed \sum_{k=1}^{n-1} (dd^cu)^{k}\wed \om^{n-1-k}.
\end{aligned}$$
Thanks to \eqref{eq:capc-rho} we invoke the comparison of mass in Lemma~\ref{lem:CP-mass} to get that
 \[\label{eq:capc-j1} \begin{aligned}
\limsup_{j\to \infty} J_1
&\leq A_0' \int_{\Om}  T \wed \sum_{k=1}^{n-1} [dd^c(A_0\rho)]^{k} \wed \om^{n-1-k}\\ 
&\leq A_1'  \int_\Om T\wed \om^{n-1},
  \end{aligned}\]
where the last inequality used the fact $dd^c \rho \leq A \om$.

Combining \eqref{eq:capc-dominate}, \eqref{eq:capc-j2} and \eqref{eq:capc-j1}  together  we arrive at
$$
	 cap(G,\Om) \leq \limsup_{j\to \infty} (J_1+J_2) \leq A_1' ( \|f\|^2_{L^2(\Om)} +  \|T\|_\Om).
$$
Taking infimum over all currents $T\in \Ga_f$ we get
$
	cap(G,\Om) \leq A'_1 \| f\|_*^2.
$
This holds for every $f\in \Kc(G)$ and therefore, $cap(G,\Om) \leq A_1' \tc (G)$.
\end{proof}

As a consequence of the above comparison of two capacities 
we are able to state now a weak equivalence property of $\tc (\cdot, \Om)$ mentioned in Remark~\ref{rmk:c-equiv}-(b) when we consider it on a smaller domain.

\begin{cor}\label{cor:dom-equiv-c} Let $D \subset \subset \Om' \subset \Om$. Then, there exists a constant $A = A(D,\Om)$ such that for every Borel set $E \subset D$,
$$
	\tc (E, \Om') \leq \tc (E, \Om) \leq A \left[\tc (E,\Om')\right]^\frac{1}{n}.
$$
\end{cor}

\begin{proof} The first inequality follows immediately from Proposition~\ref{prop:basic}-(a). It remains to verify the second one. For the Bedford-Taylor capacity we know that
$cap(E, \Om) \leq cap(E,\Om')$. Therefore, using Lemmas~\ref{lem:c-cap} and \ref{lem:cap-c}, we get
$$
	\tc(E, \Om) \leq A [cap(E, \Om)]^\frac{1}{n} \leq A [cap(E,\Om')]^\frac{1}{n} \leq A' [\tc (E,\Om')]^\frac{1}{n},
$$
where $A'$ depends only on $\Om$ and $D$.
\end{proof}

\section{Quasi-continuity}

In this section we will show that a function in $W^*(\Om)$ admits a quasi-continuous representation with respect to the capacity $\tc(\cdot,\Om)$. Let us make it precise by defining the terminologies. 

\begin{defn} \label{defn:c-quasi-c} A function $u$ is called quasi-continuous (q.c) in $\Om$ if there exists for every $\veps> 0$ an open set $G \subset \Om$ such that $\tc (G)<\veps$ and the restriction $u_{|_{\Om\setminus G}}$  is continuous  (with respect to the induced topology on $\Om\setminus G$).
\end{defn}

A statement is said to hold {\em quasi-everywhere (q.e)} if there exists a set $P$ of capacity zero, i.e., $\tc (P) =0$, such that the statement is true for every $x\notin P$.

\begin{remark}\label{rmk:no-conf} The Bedford-Taylor capacity $cap(\cdot, \Om)$ and $\tc (\cdot, \Om)$  are equivalent to each other by Lemmas~\ref{lem:c-cap}, \ref{lem:cap-c} (also Theorem~\ref{thm:choquet-c} below). Hence, these quasi-continuity and quasi-everywhere notions coincide with the classical ones in pluripotential theory.  Thus, there is no confusion if we use these properties for psh functions which belong to $W^*(\Om)$.
\end{remark}

The following statement is classical in potential theory \cite[Lemma~2.14]{FOT94}. 
\begin{prop} \label{prop:ae-qe}  Let $D\subset \Om$ be an open set and $f$ is quasi-continuous on $D$. If $f\geq 0$ a.e. on $D$, then $f\geq 0$ quasi-everywhere on $D$.
\end{prop}

\begin{proof} For each $k\geq 1$, there exists $G_k$ open such that $\tc (G_k) <1/k$ and $f_{|_{F_k}}$ is continuous with $F_k = D\setminus G_k$. Define $F_k' = \supp ({\bf 1}_{F_k} \cdot dV_{2n}$) and $G_k' = D\setminus F_k'$. Then, $G_k\subset G_k'$ and $V_{2n} (G_k' \setminus G_k) =0$. By definition  $\tc (G_k') = \tc (G_k)$. It is enough to prove $f(x)\geq 0$ for every $x\in \cup_k F_k'$.

Assume $f(x)<0$ for some $x\in F_k'$. Then, by the continuity, there exists a neighborhood $U_x$ such that $f(y)<0$ for every $y\in F_k' \cap U_x$. Since $V_{2n} (F_k' \cap U_x)>0$,  it would contradict to the assumption that $f\geq 0$ a.e on $D$. We conclude that $f(x) \geq 0$ for $x\in \cup_k F_k'$.
\end{proof}

Unlike psh functions, the functions in $W^*(\Om)$ are defined up to a set of Lebesgue measure zero. To get a good presentative in each class of a given function we need the following notion.

\begin{defn}
Let $f,g$ be functions in $\Om$. Then,  $g$ is said to be a quasi-continuous modification of $f$ if $g$ is quasi-continuous and $g= f$ a.e. We denote $g$ by $\wt f$ in this case.
\end{defn}

Thanks to the above proposition two quasi-continuous modifications are equal quasi-everywhere. 
Our goal is to show that every element in $W^*(\Om)$ admits a quasi-continuous modification.

\begin{lem}\label{lem:cap-sublevel-set} Let $f\in W^*(\Om) \cap C^0(\Om)$.  Assume either $f\geq 0$ or $f\leq 0$ on $\Om$. Denote $E_s = \{z\in \Om: |f(z)| > s\}$ for $s > 0$. Then,
\[\label{eq:c-sublevel-set-est}
	\tc (E_s) \leq \frac{1}{s^2} \|f\|_*^2.
\]
\end{lem}

\begin{proof} Since $f$ is continuous, $E_s = \{z\in \Om : |f|/s>1\}$ is open and $-|f|/s \in \Kc(E_s)$. The conclusion follows easily from the definition of $\tc (E_s)$. The only reason we add the assumption that $f$ does not change sign is  to avoid the constant 4 would appear on the right hand side as in general $\| |f| \|_* \leq 2 \|f\|_*$.
\end{proof}

\begin{defn} A sequence of functions $\{f_k\}_{k\geq 1}$ in $\Om$ is called a Cauchy sequence with respect to the capacity $\tc(\cdot,\Om)$ if for every compact set $K\subset \Om$ and $\de>0$, 
$$
	\limsup_{j,k\to \infty} \tc( \{|f_j - f_k| >\de\} \cap K) =0.
$$
\end{defn}

We are ready to state a key ingredient that helps to select a good representative of functions in $W^*(\Om)$. This result is inspired by the results of Dinh, Marinescu and Vu \cite[Lemma~2.9, Theorem~2.10]{DMV} where they used the Bedford-Taylor capacity. We have an advantage of  the comparison between two capacities. Therefore, we use the suitable capacity to handle  separate parts of super-level sets.

\begin{thm}\label{thm:cap-cauchy} Let $f\in W^*(\Om)$.  Denote $f_k= f* \eta_{
1/k}$, where $k=1,2...$. Then, $\{f_k \}$ is a Cauchy sequence in capacity.
\end{thm}

\begin{proof} Let us fix a compact set $K\subset \Om$ and $\de>0$. Let $\veps>0$, we wish to show that  there exists $j_0>0$ such that for $j, k>j_0$,
\[\label{eq:cc-goal}
	\tc (\{ |f_j -f_k| >\de\} \cap K,\Om) <\veps.
\]
Strictly speaking since functions $f_j$ are defined on a bit smaller domain $\Om'\subset\subset \Om$ where $K \subset \subset \Om'$, we need to consider $\tc (\cdot, \Om')$ instead. On the other hand $K$ is fixed,  by  Corollary~\ref{cor:dom-equiv-c} we may assume $f\in W^*(\wt\Om)$ for some neighborhood of $\bar\Om$  without loss of generality. Thus, we can still write $\tc(E) = \tc(E,\Om)$ in what follows.

Now, by writing $f = f^+ - f^-$, we may assume that $f\geq 0$. 
Then, we reduce the estimate to the case $\{f_j\}$ is uniformly bounded as follows. 
Fix a large $N>0$ and consider $\wh f_j = \min\{f_j, N\}$. Then,
$$
	\{| f_j - f_k |>\de\} \subset \{|\wh f_j -\wh f_k|>\de\} \cup E_{jk}(N),
$$
where
$E_{jk}(N) = \{ f_j > N\} \cup \{f_k > N\}.$
Using Lemma~\ref{lem:cap-sublevel-set} we get that for $N > j_1$,
\[\label{eq:cc-j1}
	\tc (E_{jk}(N)) \leq  \frac{1}{N^2} (\|f_j\|_*^2 + \|f_k\|_*^2) \leq \veps/4.
\]
Thus, 
by considering $\wh f_{j}$ we may assume from now on that the sequence is uniformly bounded by $N$. For simplicity we use again the same notation $f_j$. These functions satisfy 
\[\label{eq:cc-smooth-ineq}
	d f_k \wed d^c f_k \leq T*\eta_{1/k}.
\] 

Next, by covering $K$ by finitely many small balls and using the subadditivity of capacity, we reduce the problem to the case  $df \wed d^c f \leq T = dd^c u$ in the ball $\Om = B(0, 2r)$ and $K \subset B(0,r)$ for some $r>0$, where $u\in PSH(B(0,3r))$ is negative. Hence, if we define  $u_k:= u*\eta_{1/k}$, then it follows from \eqref{eq:cc-smooth-ineq} that
$$
	d f_k \wed d^c f_k \leq dd^c u_k  \quad \text{in } B(0, 3r/2).
$$
We use now the Bedford-Taylor capacity which satisfies 
\[\notag
	cap (\{u<-N\} \cap B(0,r) , \Om) \leq \frac{1}{N} \|u\|_{L^1(\Om)}
\]
by the CLN inequality (see e.g. \cite[Proposition~1.10]{ko05}). 
Hence, for $N>j_2$,
\[\label{eq:cc-BT-j2}
	cap (\{u<-N\} \cap B(0,r) , \Om) \leq (\veps/4)^n/A,
\]
 where $A$ is the constant appeared in Lemma~\ref{lem:c-cap}.
We divide 
$$ K \subset  (K \cap \{u\geq - N\}) \cup \{u<- N\}.$$
By \eqref{eq:cc-BT-j2} it remains to estimate $cap( \{| f_j - f_k| > \de\} \cap K \cap \{u \geq -N\}, \Om)$. It follows from the definition of Bedford-Taylor capacity and Markov's inequality that this part is dominated by
\[\label{eq:cc-BT-markova}
 \sup \left\{ \frac{1}{\de^2}\int_{K \cap \{u\geq -N\}}  (f_j -f_k)^2 (dd^c \phi)^n \right\},
\]
 where the supremum is taken over all $0\leq \phi \leq 1$ is psh function in $\Om$. 
 By an elementary inequality
$$
	{\bf 1}_{\{u\geq -N+1\}} \leq \max \{u, -N\} +N:= u_N
$$
we get
 $$
 	\int_{K\cap \{u\geq -N+1\}} (f_j-f_k)^2(dd^c\phi)^n \leq \int_K  u_N^2 (f_j-f_k)^2 (dd^c\phi)^n.
$$

\begin{claim} \label{cl:quasi-c} We have $g_j := u_N f_j \in W^*(\Om')$ for every $\Om'\subset\subset \Om$ and 
$$
	d g_j \wed d^c g_j \leq N^2 dd^c (u_N^2 + u_{j, N+1}),
$$
where $u_{j,N+1} = \max\{u_j, -N-1\}$. Moreover, 
$$
	v_j:=N^2(u_N^2 + u_{j,N+1})  \searrow  N^2(u_N^2 + u_{N+1})=:v.
$$
\end{claim}

\begin{proof}[Proof of Claim~\ref{cl:quasi-c}] The first statement follows immediately from Example~\ref{expl:construction} and Remark~\ref{rmk:general-construction}, while the second part follows from the its formula.
\end{proof}

Next, let $0\leq \chi \leq 1$ be a smooth cut-off function such that $\chi =1$ on a neighborhood of $K$ and $\supp\chi \subset \subset \Om$.
We are ready to apply Lemma~\ref{lem:wed-prod} for $g_j-g_k = u_N (f_j-f_k)$ which satisfies $$d(g_j-g_k) \wed d^c (g_j-g_k) \leq 2 dd^c (v_j+v_k).$$ Here, $0\leq v_j = N^2(u_N^2 + u_{j,N+1}) \leq N^4$ is a psh function for $j\geq 1$. Then,
\[\label{eq:cc-BT-markovb} \begin{aligned}
&	\int_K (g_j-g_k)^2 (dd^c\phi)^n \\
&\leq \quad A_0 N^{4a} \left(  \int \chi (g_j-g_k)^2 \left[\om + 2dd^c (v_j+v_k) \right]^n\right)^\frac{1}{2^n}. 
\end{aligned}\]
The last integral on the right hand side is taken care by Corollary~\ref{cor:L1-cap-comparison-a} as follows:
$$\begin{aligned}
	\int \chi (g_j-g_k)^2 [\om &+ 2dd^c(v_j+v_k)]^n \\
&\leq \int \chi (g_j-g_k)^2 (\om +4dd^c v)^n \\ 
&\quad + A_1 \|g_j-g_k\|_* \int_\Om \left[2(v_j+v_k) - 4v\right] (\om+ 4dd^cv)^n \\
&=: I +J.
\end{aligned}$$
By Lebesgue monotone convergence, for $j,k \geq j_3$, 
\[\label{eq:cc-j3}
	J \leq \frac{1}{A_0N^{4a}}\left(\de^{2}/A^n \right)^{2^n} (\veps/4)^n =: \veps'.
\] To estimate $I$,  we
apply again once  more Corollary~\ref{cor:L1-cap-comparison-b} with  another independent index $\ell$,
$$\notag\begin{aligned}
I 
&\leq  \int \chi (g_j-g_k)^2 [\om +4dd^c (v*\eta_{1/\ell})]^n  \\
&\quad + A_2 \|g_j-g_k\|_* \int_\Om 4(v*\eta_{1/\ell} -v ) (\om+ 4dd^cv)^n.
\end{aligned}$$
Therefore, we can choose first $\ell = \ell_0$ so large that the second integral is less than $\veps'/2$ 
and next choose $j,k > j_5> \max\{j_1,j_2, j_3, j_4\}$ so that the first integral is less than $\veps'/2$ 
Therefore, for $j,k>j_5$,
\[\label{eq:cc-j4}
I  \leq  \veps'.
\]

Combining \eqref{eq:cc-j1}, \eqref{eq:cc-BT-j2}, \eqref{eq:cc-BT-markova}, \eqref{eq:cc-BT-markovb}, \eqref{eq:cc-j3},  \eqref{eq:cc-j4} and Lemma~\ref{lem:c-cap} we get that for $j, k>j_5$, $$\tc (\{|f_j-f_k|>\de\} \cap K) <\veps.$$ This proved  \eqref{eq:cc-goal}. 
\end{proof}

\begin{remark} \label{rmk:cauchy} The above argument is applicable for  every sequence  $f_k = f*\eta_{\veps_k}$, where  $\veps_k \searrow 0$ and every sequence $f_{r_j} = f * \chi_{r_j}$ with $r_j \to 0$ defined as in \eqref{eq:mean}. Therefore, the statement of the theorem still holds for such sequences.
\end{remark}

We are ready to prove the existence of a quasi-continuous modification in Corollary~\ref{cor:intro-quasi-mod} which  is a reformulation of Dinh, Marinescu and Vu \cite[Theorem~2.10]{DMV} in our setting. The proof given below is standard, we follow closely the one in  \cite[Theorem~2.13]{FOT94}.

\begin{thm} \label{thm:quasi-mod} Let  $f\in W^*(\Om)$. There exists a quasi-continuous modification $\wt f$.
\end{thm}

\begin{proof} Let $D \subset \subset \Om$ be a subdomain. Let $f_k = f*\eta_{1/k}$ with $k=1,2,...$. Define 
$$
	E_k = \{x\in D: |f_k - f_{k+1}| >1/2^{k}\}, \quad
	G_j = \cup_{k\geq j} E_k.
$$ 
By choosing a subsequence we can assume that $\tc (E_k) \leq 2^{-k}$. Hence, $\tc(G_j) \leq 2^{-j}$ and 
$$
	\tc \left( \bigcap_{j=1}^\infty G_j \right) =0.
$$
If $x\in D\setminus G_j$, then 
$$
	|f_{N+s}(x) - f_N(x)| \leq \sum_{k\geq N} |f_{k+1} -f_k| \leq 1/2^N.
$$
for all $s \geq 1$ and $N\geq j$. 
This means $\{f_{k}\}$ restricted to $D\setminus G_j$ is uniformly convergent as $k\to \infty$. Let
$\wt f(x) :=\lim_{k\to \infty} f_k(x)$  for $x\in \cup_j (D\setminus G_j)$. Then, it is  a quasi-continuous function and $\wt f = f$ almost everywhere on $D$. 

Now we prove the quasi-continuity of $\wt f$ in $\Om$. Let $D_\ell$ be a sequence of increasing relative compact domains that $\Om = \cup_\ell D_\ell$.
Define 
$$
	E_k(\ell) = \{x\in D_\ell: |f_k - f_{k+1}| >1/2^{k}\}, \quad
	G_j (\ell)= \cup_{k\geq j} E_k(\ell).
$$ 
By diagonal argument we can pass to a subsequence satisfying  $\tc (E_k(\ell)) \leq 2^{-\ell - k}$ for every $\ell\geq 1$. Hence, $\tc(G_j(\ell)) \leq 2^{-\ell- j}$ and 
$$
	\tc \left( \bigcap_{j=1}^\infty G_j(\ell) \right) =0.
$$
If $x\in D_\ell\setminus G_j(\ell)$, then 
$$
	|f_{N+s}(x) - f_N(x)| \leq \sum_{k\geq N} |f_{k+1} -f_k| \leq 1/2^N.
$$
for all $s \geq 1$ and $N\geq j$. 
This means $\{f_{k}\}$ restricted on $D_\ell\setminus G_j(\ell)$ is uniformly convergent as $k\to \infty$. Let
\[\label{eq:quasi-continuity-def} \wt f(x) :=\lim_{k\to \infty} f_k(x) \quad   \text{for }
x\in \bigcup_{\ell\geq 1} \cup_{j\geq 1}[ D_\ell\setminus G_j (\ell)]. \]
Then, $\wt f = f$ almost everywhere on $\Om$. We verify the quasi-continuity.
The function restricted to 
$F_j(\ell) := D_\ell\setminus G_j(\ell)$ is continuous. Define $F_j = \cap_{\ell =1}^\infty F_j(\ell)$. We have 
$$
	\tc (\Om \setminus F_j) \leq \sum_{\ell=1}^\infty \tc (G_j(\ell)) \leq 1/2^{j}.
$$
Thus, the proof of theorem is completed.
\end{proof}

Finally we conclude the quasi-continuous modification statement in the introduction.

\begin{proof}[Proof of Corollary~\ref{cor:intro-quasi-mod}] The existence follows from Theorem~\ref{thm:quasi-mod} while the uniqueness up to a pluripolar set follows from Proposition~\ref{prop:ae-qe}.
\end{proof}

Next we define
\[ \label{eq:space-quasi-c} \wt W^*(\Om) = \{ g \in W^*(\Om): g \text{ is quasi-continuous}\}.
\]
It follows from Theorem~\ref{thm:quasi-mod} and Proposition~\ref{prop:ae-qe} that the equivalence class of $\wt W^*(\Om)$ in the sense of a.e. with respect to the Lebesgue measure is identical with the equivalence class in the sense q.e.  We have a generalization  of  $\eqref{eq:c-sublevel-set-est}$  as follows.

\begin{cor}\label{cor:cap-sublevel-gen}  Let $f \in \wt W^*(\Om)$ be such that either $f\geq 0$ or $f\leq 0$. Assume $\Om' \subset\subset \Om$. Define $E_s' = \{z\in \Om' : |f(z)|>s\}$ for $s>0$. Then,
$$
	\tc (E_s' , \Om') \leq \frac{1}{s^2} \|f\|_{W^*(\Om')}^2.
$$
\end{cor}

\begin{proof} We may assume $f\geq 0$. Let $f_k = f*\eta_{1/k}$ and $\wt f$ as in the proof of Theorem~\ref{thm:quasi-mod}. We can choose $k\geq 1$ large enough so that all $f_k$ belong to $W^*(\Om') \cap C^0(\Om')$. For a Borel set we write now $\tc (E) = \tc(E,\Om')$ in this proof.

Let $\veps>0$. Since $f = \wt f$ q.e., there is an open set $G \subset \Om'$ such that $\tc(G) <\veps$ and $f_k \to f$ uniformly on $\Om'\setminus G$. Hence, we have for $s>\de>0$, there exists $k_0 = k_0(\de)$ such that  for $k\geq k_0$,
$$E_s'= \{z\in\Om': f > s \} \subset \{z\in \Om': f_k > s -\de \} \cup G.$$ Applying \eqref{eq:c-sublevel-set-est} for $f_k$ we get
$$ \tc (E_s') \leq 
\tc (\{f_k >s-\de\}) + \veps  \leq \|f_k\|_{W^*(\Om')}^2 / (s - \de)^2 +\veps.
$$  Letting $k\to \infty$, $\de\to$ and $\veps\to 0$ in this order and using Proposition~\ref{prop:smoothing} we get the desired inequality.
\end{proof}

Another nice corollary is local $L^2$-integrability of functions in $W^*(\Om)$ with respect to Monge-Amp\`ere measures of bounded psh functions. 
Let $\phi \in PSH(\Om)$ be such that $0\leq \phi \leq 1$. 
Define $\mu: = (dd^c\phi)^n$.
Let $f \in W^*(\Om)$ and let $\wt f$ be a quasi-continuous modification. Let $D \subset \subset \Om$ be a subdomain. By Theorem~\ref{thm:quasi-mod},  there exists a sequence $f_k:=f* \eta_{1/k}$ converges q.e. to $\wt f$ in $D$. If $g$ is another quasi-continuous modification, then $g = \wt f$ q.e. by Propostion~\ref{prop:ae-qe}. If $|f| \leq N$, then $|f_k| \leq N$ and $|\wt f| \leq N$ quasi-everywhere in $D$. By Lebesgue's dominated convergence theorem and Proposition~\ref{prop:L2-cap} we obtain for $K\subset D$
\[\label{eq:uni-bound-L2-defn}	
	\int_K (\wt f)^2 d\mu = \lim_{k\to \infty} \int_K f_k^2 d\mu \leq A \|f\|_{W^*(\Om)}^2,
\]
where $A>0$ is independent of $f$. 
Therefore, 
\[
	\int_K f^2 d\mu := \int_K (\wt f)^2 d\mu 
\]
is well-defined for a bounded function $f \in W^*(\Om)$. 
In general, we consider $f_N := \min\{|f|, N\}$ for $N>0$ and define
$$
	\int_K f^2 d\mu = \lim_{N\to +\infty} \int_K \wt f^2_N d\mu.
$$
The limit on the right hand side exists because it is an increasing sequence in $N$ and uniformly bounded by \eqref{eq:uni-bound-L2-defn}. Thus, the following integrals are well-defined.

\begin{defn}
\label{defn:L2-ma} Let $\phi \in PSH\cap L^\infty(\Om)$.
Define $\mu: = (dd^c\phi)^n$. Let $g \in W^*(\Om)$ and $K \subset \subset \Om$ be a compact subset.  If $g\geq 0$, then
$$
	\int_K g d\mu := \int_K \wt g d\mu,
$$
where $\wt g$ is a quasi-continuous modification. 
Generally, 
$$
	\int_K g d\mu := \int_K g^+ d\mu - \int_K g^- d\mu.
$$
\end{defn}

\section{Capacitability}

Let $\Om$ be a bounded open set in $\bC^n$. Our goal is to show that $\tc(\cdot, \Om)$ is a capacity in the sense of Choquet.  The first one is inspired by  an analogue result of \cite[Corollary~31]{Vigny} on compact K\"ahler manifolds. Additionally, we obtain the uniqueness of the extremal function in the local setting. Following \cite[Lemma~2.1.1]{FO87} we assume first the openness of the set and remove this condition later. 

\begin{lem} \label{lem:extr-funciton-o} Assume $E\subset\Om$ is an open subset. There is a unique $\Phi_E \in W^*(\Om)$ such that 
$\tc (E) = \|\Phi_E\|_*^2$. Moreover, $-1\leq \Phi_E \leq 0$ a.e. in $\Om$ and $\Phi_E=-1$ q.e. on $E$.
\end{lem}

\begin{proof} We first prove the existence. Let $\{u_j\}_{j\geq 1}$ and $-1\leq u_j \leq 0$ be a sequence in $\Kc(E)$ such that $\tc(E) = \lim_{j\to \infty} \|u_j\|_*^2$. It follows from Corollary~\ref{cor:w-compactness} that by passing to a subsequence $u_j$ converges weakly in $W^{1,2}(X)$ to  $u \in W^*(\Om)$ and $\|u\|_* \leq \liminf_j \|u_j\|_*$.  By Proposition~\ref{prop:weak-conv-cp} we have $u =-1$ a.e. on $E$ and $-1\leq u\leq 0$ in $\Om$.  The proof of this proposition used only the properties of Lebesgue points of functions in $L^p(\Om)$. Since $E$ is open,  $u\in \Kc(E)$ (this is the place where the openness is used) and $\tc(E) = \|u\|_*^2$. Of course we can take $\wt u$ in the place of $u$.  By Proposition~\ref{prop:ae-qe} and the openness of $E$, we have $\wt u = -1$ q.e. 

Next,  we show the uniqueness. Suppose $u_1$ and $u_2$ are such two functions. Let $T_1$ and $T_2$ are the minimum closed positive currents of $u_1$ and $u_2$, respectively. Hence,
$$
	\tc(E) = \|u_1\|_{L^2(\Om)}^2 + \|T_1\|_\Om = \|u_2\|_{L^2(\Om)}^2 + \|T_2\|_\Om.
$$  
Define $v = (u_1+u_2)/2$. Clearly, 
$
	dv \wed d^c v \leq (T_1 +T_2)/2
$
and therefore, 
$$\begin{aligned}
	\tc(E) &\leq \|v\|_{L^2(\Om)}^2 + \|T\|_\Om \\&=\frac{1}{4} (\|u_1\|_{L^2(\Om)}^2 + \| u_2\|_{L^2(\Om)}^2) + \frac{1}{2} \|u_1u_2\|_{L^1(\Om)} + \frac{1}{2} (\|T_1\|_\Om + \|T_2\|_\Om).
\end{aligned}$$
This implies that 
$$
	2\int_\Om u_1 u_2  dx \geq \int_\Om u_1^2 \;dx + \int_\Om u_2^2 \; dx.
$$
Thus, $u_1 = u_2$ a.e. 
\end{proof}

It is easy to see that if $v\in W^*(\Om)$ and $v\leq -1$ a.e in a neighborhood of $E$, then $\wt v\in \wt W^*(\Om)$ and $\wt v\leq -1$ q.e on $E$. We show that the infimum in the definition of the capacity can be considered in this bigger space.  

\begin{lem}\label{lem:defn-cap-mod} For a Borel set $E$, 
 $$
 	\Kc'(E) := \left\{ v \in W^*(\Om) \;\rvert\; \wt v \leq  -1 \text{ q.e. on $E$ and } \wt v\leq 0 \text{ q.e.}\right\}.
$$
Then,
$$
	\tc(E) = \inf\{\|v\|_*^2 : v \in \Kc'(E)\}.
$$
\end{lem}

\begin{proof} Since $\Kc(E) \subset \Kc'(E)$, we have $\tc (E)$ is greater than the infimum. It remains 
to show the reverse  inequality. 
Let us consider $h\in  \Kc'(E)$ and fix a quasi-continuous modification $\wt h$ of $h$ (Theorem~\ref{thm:quasi-mod}). Given $0< \veps <<1$ we choose an open set $G:= G_\veps$ such that $\tc (G)<\veps$ and $\wt h$ restricted to $\Om\setminus G$ is continuous and $\wt h \leq -1$ on $E \cap (\Om\setminus G)$. Denote  by $u$ the function $\Phi_{G}$ of  the open set $G$ in Lemma~\ref{lem:extr-funciton-o}. Then
$$
	D_\veps = \{x\in \Om \setminus G: \wt h < -1+\veps \} \cup G
$$
is an open subset in $\Om$ and $E\subset D_\veps$. Moreover, $h + u \leq -1 + \veps$ a.e. on $D_\veps$ and it is negative in $\Om$.
Therefore,
$$\begin{aligned}
	\tc (E) &\leq \tc (D_\veps) \leq \frac{\|h+u\|_*^2}{(1-\veps)^2}  \\& \leq \frac{ (\|h\|_* +\|u\|_*)^2}{(1-\veps)^2} \\ &\leq \frac{ (\|h\|_* + \sqrt{\veps})^2}{(1-\veps)^2}.
\end{aligned}$$
By letting $\veps \searrow 0$, $\tc(E) \leq \|h\|_*^2$. Since $h$ is arbitrary,  we conclude  the reverse  inequality.
\end{proof}

We need to improve the a.e. convergence  to the q.e. one in order  to remove the openness condition in Lemma~\ref{lem:extr-funciton-o}. This is the most difficult step where we have to use the functional capacity in the Dirichlet spaces in Section~\ref{ss:dirichlet-cap}.

\begin{lem}\label{lem:qe-limit} Let $\{u_j\}_{j\geq 1} \subset \Kc'(E)$ be such that  $ \|u_j\|_*^2 \leq A$ for $j\geq 1$. Assume $u_j $ converges weakly in $W^{1,2}(\Om,\bR)$ to $u\in W^*(\Om)$. If $u_j = -1$ q.e. on $E$, then $u = -1$ q.e. on $E$.
\end{lem}

\begin{proof} Without loss of generality we may assume $u_j$'s and $u$ are quasi-continuous and $-1\leq u_j, u \leq 0$. Define $F= \{u>-1\} \cap E$.  We wish to show that $\tc(F)=0$. By subadditivity it is enough to show that $\tc(F \cap K_j)=0$ for an exhaustive sequence of compact set $K_j \uparrow \Om$. In other words, we may assume $F\subset \subset \Om$ and $\Om$ is a strictly pseudoconvex domain.

Arguing by contradiction,  we suppose  that $\tc (F)>0$. Lemma~\ref{lem:c-cap} implies $cap (F)>0$. According to the characterization in Corollary~\ref{cor:FO-polar} there exists $\phi \in PSH(\Om)$, $-1\leq \phi \leq 0$ and  $0<\de \leq 1$ such that the capacity $\tc_\te$ associated with $\te = (dd^c \phi + \de\om)^{n-1}$ satisfying  \[\label{eq:cap-te-pos}\tc_\te (F)>0.\] This capacity is defined in \eqref{eq:cap-te} on the Dirichlet space $\cF$ (Definition~\ref{defn:O}), where $\cF$ is identified with the real Hilbert space $H_0^1(\Om, \te)$ (Lemma~\ref{lem:identification}). The latter space is equipped with the inner product $\lc \cdot, \cdot\rc_\te$  in \eqref{eq:inner-prod-te} and  the $\|\cdot\|_\te$-norms in \eqref{eq:norm-te}.

Let $\chi \in C^\infty_c(\Om)$ be such that $0\leq \chi \leq 1$ and $\chi =1$ on a neighborhood of $\bar E$. By Theorem~\ref{thm:embedding}, $v_j:=\chi u_j \in \cF$ and $v:= \chi u\in \cF$ whose $\|\cdot\|_\te$-norms  are uniformly bounded. Clearly, $v_j \to v$  weakly in $W^{1,2}(\Om, \bR)$. So, $\lc v_j, g\rc_\te \to \lc v, g\rc_\te$ for every $g\in C^\infty_c(\Om,\bR)$. 
Since the latter space is dense in $\cF$, it follows  that $v_j \to v$ weakly in $(\cF, \|\cdot\|_\te)$.
Mazur's lemma  implies there is a convex combination $\Phi_\ell= \sum_{j=1}^\ell a_j v_j$, where $a_j \geq 0$ and $\sum_{j=1}^\ell a_j=1$, converges  to $v$ in $\|\cdot\|_\te$-norm. 

We are ready to get a contradiction. By the assumption for $j\geq 1$, $v_j =-1$ on $E\setminus P_j$ where $P_j$ is pluripolar. Hence, $\Phi_\ell =-1$ on $E$ outside a pluripolar set $P'=
\cup_{j\geq 1} P_j$.
Thanks to Lemma~\ref{lem:FO-polar}-(b) each $\Phi_\ell$ is quasi-continuous with respect to  with respect to $\tc_\te(\cdot)$. By \cite[Theorem~2.1.4]{FOT94} for $\{\Phi_\ell\}_{\ell\geq 1}$, we infer that $\Phi_\ell \to v$ point-wise on $\Om\setminus P''$ where $\tc_\te(P'') =0$. 
Hence, $u=v=-1$ on $E\setminus ( P' \cup P'')$. By definition  $F \subset P' \cup P''$. Invoking Corollary~\ref{cor:FO-polar} once more we get  $\tc_\te(F) =0$. This is impossible by \eqref{eq:cap-te-pos}. 
\end{proof}

Thanks to the above lemma the function $\Phi_E$ obtained  in the proof of Lemma~\ref{lem:extr-funciton-o} belongs to $\Kc'(E)$ for every Borel set. Thus, Lemma~\ref{lem:defn-cap-mod} is now valid in general.

\begin{cor} \label{cor:extr-function} Let $E\subset\Om$ be a Borel  subset. There is a unique $\Phi_E \in W^*(\Om)$ such that 
$\tc (E) = \|\Phi_E\|_*^2$. It satisfies that $-1\leq \Phi_E \leq 0$ a.e. and $\Phi_E=-1$ q.e. on $E$.
\end{cor}

Now we prove  the main result of this section which is the content of Theorem~\ref{thm:intro-choquet-c}. The proof is more detail than the one in \cite[Theorem~30]{Vigny}.

\begin{thm} \label{thm:choquet-c} Let $\Om$ be a bounded open set in $\bC^n$. The set function on Borel sets $E\mapsto \tc (E)= \tc (E,\Om)$  satisfies:
\begin{itemize}
\item[(a)] $\tc (E_1) \leq \tc (E_2)$ for $E_1\subset E_2\subset \Om$.
\item[(b)] if $K_1 \subset K_2 \subset \cdots$ are compact sets in $\Om$ and $K = \cap K_j$, then $\tc (K) =\lim_{j\to \infty} \tc (K_j)$.
\item[(c)]  if $E_1\subset E_2 \subset \cdots \subset \Om$ are Borel sets and $E= \cup E_j$, then $\tc (E) = \lim_{j\to \infty} \tc (E_j)$.
\end{itemize}
\end{thm}

\begin{proof} The items (a) and (b) are proved in Proposition~\ref{prop:basic}, therefore  it remains only to prove (c). It is enough to show that $\tc (E) \leq \lim_{j\to \infty} \tc (E_j)$ as the other direction is obvious by (a).  By Lemma~\ref{lem:defn-cap-mod} there is  $u_j \in \Kc'(E_j)$  satisfying 
$$
	\tc (E_j) = \|u_j\|_*^2.
$$
We may assume $\|u_j\|_*$ is bounded uniformly in $j\geq 1$ (otherwise the conclusion is true). Thus, there exists a subsequence converging weakly to $u\in W^*(\Om)$. In particular, $u_j \to u$ weakly in $L^2(\Om)$. Proposition~\ref{prop:weak-conv-cp} implies that $u = -1$ a.e. on $E= \cup_j E_j$ and $u\leq 0$ in $\Om$. Lemma~\ref{lem:qe-limit} and Proposition~\ref{prop:ae-qe} yield $\wt u = -1$ q.e. on $E$ and $\wt u \leq 0$ in $\Om$, respectivly. So, $\tc (E) \leq \|u\|_*^2 \leq \liminf_{j\to \infty} \|u_j\|_*^2.$ 
\end{proof}

Since $\tc(\cdot,\Om)$ is a capacity in the sense of Choquet,  it is inner regular, i.e., for every Borel sets
$$
	\tc (E) = \sup \{ \tc (K) : K \subset E \text{ is compact}\}.
$$ 
This combined with 
 Lemma~\ref{lem:cap-c} and Lemma~\ref{lem:c-cap} give a useful observation.

\begin{cor}\label{cor:pluripolar-sets-c} Let $E\subset \Om$ be a Borel set. $E$ is pluripolar if and only if $\tc(E,\Om) =0$.
\end{cor}

\section{Embedding into Dirichlet spaces}
\label{sec:dirichlet}

\subsection{Dirichlet spaces associated with a closed positive current}
\label{ss:currents}

In this section we show that $W^*(\Om)$ can be naturally embedded into a Dirichlet space associated with a positive closed $(n-1,n-1)$-current. These positive currents comes from the wedge product of $(1,1)$-currents associated with  bounded  psh functions.

Let $-1\leq \phi_1,...,\phi_{n-1} \leq 0$ be  psh functions in $\Om$. 
Then, \[\label{eq:the-bdd}\te = dd^c \phi_1 \wed \cdots \wed dd^c \phi_{n-1}\] is a closed positive current of bidegree $(n-1,n-1)$. Without loss of generality we consider the case all functions $\phi_k=\phi$.  We introduce a quadratic form  associated with  $\te$ on  $C^\infty_c(\Om, \bR)$ as follows. For $f,g \in C^\infty_c(\Om,\bR)$, 
$$
	\Ec_\te(f,  g) = \int df \wed d^c g \wed \te 
$$
Here and all integrals considered in the remaining part of this section are on $\Om$, so we omit writing $\Om$ in the integrations. Define
$$
	\Ec_\te(f) :=  \Ec_\te(f,f) = \int df \wed d^c f \wed \te.
$$
The trace measure of $\te$ is given by
\[\notag	\mu := \te \wed \om = \sum \tau_{j\bar j},
\]
which  is a positive Radon measure.
\begin{remark} Let us write $\te = \sum \te_{j\bar k} \; \wh{dz^j} \wed \wh { d \bar z^k}$, where $\te_{j\bar k}$ are complex Radon measures. They are absolutely continuous with respect to 
$\mu = \te \wed \om$. Hence, $\te_{j\bar k} = h_{j\bar k} \cdot \mu$ for some bounded (complex valued) Borel function $h_{j\bar k} \in L^\infty(\Om, d\mu)$. This gives
$$
 \Ec_\te(v) = \int_\Om \sum_{j,k=1}^n \d_j v \cdot \overline{\d_k v} \; \te_{j\bar k} = \int_\Om \sum_{j,k=1}^n \d_j v \cdot \overline{\d_k v} \; h_{j\bar k} \; d\mu,
$$
\end{remark}

The following result of Okada \cite{Okada82} is very useful. 
 
\begin{lem} \label{lem:O} Let $\{f_i\}_{i\geq 1} \subset C^\infty_c(\Om,\bR)$ be such that $f_i \to 0$ in $L^2(\Om, d\mu)$. Assume $\{f_i\}_{i\geq 1}$ is a Cauchy sequence in $\lc \cdot, \cdot \rc_\te$-norm, i.e.,
\[\label{eq:Cauchy-seq}
 \Ec_\te( f_j - f_k) =\int d(f_j -f_k) \wed d^c (f_j-f_k) \wed \te \to 0 \quad\text{as } j,k \to \infty.
\]
Then, $\lim_{j\to \infty} \Ec_\te(f_j) =0$.
\end{lem}

\begin{proof} We first show the following characterization.  For $g\in C^\infty_c(\Om)$,
$$
	[\Ec_\te(g)]^\frac{1}{2} = \sup \left\{  \int g dd^c \chi \wed \te : \chi \in C^\infty_c(\Om), \; \| \chi\|_\te \leq 1 \right\}.
$$
By integration by parts,
$$
	\int g dd^c \chi \wed \te = - \int dg \wed d^c \chi \wed \te.
$$
Here the integration by part holds by a simple approximation argument via $\te_\veps := \te *\eta_\veps$. Using Cauchy-Schwarz inequality for this we get that the supremum is  less than $[\Ec_\te(g)]^\frac{1}{2}$.  Moreover, for $\chi = -f/ [\Ec_\te(f)]^\frac{1}{2}$ the supremum is attained. The characterization is proved. 

Next, we  apply it for $h= f_j - f_k$.  Fix $\veps>0$. 
Since $\Ec_\te (f_j-f_k) \to 0$, there exists $N$ such that we have  for $j,k>N$ and for $\chi \in C^\infty_c(\Om)$,
$$
	\int (f_j -f_k) dd^c\chi \wed \te < \veps.
$$
Clearly the total variation $\| dd^c \chi \wed \te \| \leq A  \om \wed \te$ for some constant $A$ depending only on $\chi$. By letting $k\to \infty$ we get 
$$
	\int f_j dd^c\chi \wed\te <\veps.
$$
Taking supremum over $\chi$ and the characterization gives the conclusion.
\end{proof}

Since the space $W^*\subset W^{1,2}(\Om)$, we need an additional assumption on $\te$ so that $L^2(\Om, d\mu) \subset L^2(\Om)$. The simplest way is by considering $\phi := \phi + \de |z|^2$ for some $0< \de \leq 1$. Since $\om = dd^c|z|^2$, we have
\[\label{eq:add}
	\te  = \sum_{k=0}^{n-1} \binom{n-1}{k} \de^{n-k} \om^{n-1-k} \wed (dd^c\phi)^{k}  =: \sum_{k=0}^{n-1} \de^k \te_k,
\]
where $\te_k =  \binom{n-1}{k} \om^{n-k}\wed (dd^c\phi)^{k}$.
Hence, the trace measure  $\mu=\te\wed \om$ satisfies
\[\label{eq:trace-m}
	 \mu \geq c_0 \sum_{k=0}^{n-1} \om^{n-1-k}\wed (dd^c\phi)^k.
\]
In particular, $\mu = \te \wed \om \geq c_0 \om^n$ and 
$
	\|\na f\|_{L^2(\Om)}^2 \leq  \|f\|_\te^2 
$
for  $f\in C^\infty_c(\Om).$

Let us consider  the following inner product 
\[ \label{eq:inner-prod-te}
	\lc f, g \rc_\te =  \int f g \; \om^n + \Ec_\te( f, g),
\]
and the associated norm
\[ \label{eq:norm-te}
	\|f\|^2_\te  =  \|f\|_{L^2(\Om)}^2 +  \Ec_{\te}(f,f).
\]
The assumption \eqref{eq:add} makes this inner product  positive definite on $C^\infty_c(\Om, \bR)$. Then, we define  the real Hilbert space 
$$H^{1}_0(\Om, \te) := \overline{C^\infty_c(\Om,\bR)}$$  to be the completion of $C^\infty_c(\Om,\bR)$  with respect to  $\| \cdot\|_\te$. A priori this is an abstract functional space. However,  we are going to identify precisely this space.

First,  thanks to  the result of Fukushima and Okada \cite[Theorem~2.5]{FO87} that $\lc \cdot, \cdot\rc_\te$ is closable in $L^2(\Om)$. More precisely, we can replace the assumption on $L^2$-convergence  with respect to $\mu$ in Lemma~\ref{lem:O} by the one for the Lebesgue measure. 

\begin{thm} \label{lem:FO}
 Assume $\{f_i\}_{i\geq 1} \subset C^\infty_c(\Om,\bR)$ such that $f_i \to 0$ in $L^2(\Om)$ and $\{f_i\}_{i\geq 1}$ is a Cauchy sequence in $\lc \cdot, \cdot \rc_\te$-norm, i.e.,
\[\label{eq:Cauchy-seq}
 \Ec_\te(f_j- f_k) =\int d(f_j -f_k) \wed d^c (f_j-f_k) \wed \te \to 0 \quad\text{as } j,k \to \infty.
\]
Then, $\lim_{i\to \infty} \Ec_\te(f_i)  =0$.
\end{thm}

This leads to a natural functional space studied in \cite{Okada82}. 

\begin{defn} \label{defn:O} A function $f\in L^2(\Om)$ belongs to $\cF$ if there exists a sequence $w_j \in C^\infty_c(\Om)$ such that $w_j \to f $ in $L^2(\Om)$ and $\{w_j\}_{j\geq 1}$ is a Cauchy sequence with respect to the semi-norm $\Ec_\te (\cdot)$.
\end{defn}
This space $\cF$ together with the symmetric form $\Ec_\te$ becomes a regular Dirichlet space on on $L^2(\Om)$ by Theorem~\ref{lem:FO} (see also  \cite[Th\'eor\`em~4]{Okada82}). 
We state the promised identification for the abstract Hilbert space.

\begin{lem} \label{lem:identification}  $H^{1}_0(\Om,\te) = \cF$.
\end{lem}

\begin{proof}   Consider identity map $\io : C^\infty_c(\Om,\bR) \to L^2(\Om)$.  Since $\|f\|_{L^2(\Om)} \leq \|f\|_\te$, this map is extended  continuously to $\wt\io: H^{1}_0(\Om, \te) \to L^2(\Om)$. Hence,  $\cF$ is contained in  the image of $H^1_0(\Om, \te)$ under $\wt\io$.
If  $\wt\io$ is injective, then the image can be identified with $\cF \subset L^2(\Om)$. In fact, let $f\in H_0^1(\Om,\te)$ such that $\wt\io (f) =0$. There exists $\{f_i\}_{i\geq 1} \subset C^\infty_c(\Om,\bR)$ which is a  Cauchy sequence representing $f$ with respect to the norm $\| \cdot\|_\te^2 = \| \cdot \|_{L^2}^2 + \Ec_\te(\cdot)$. In particular,  $\{f_i\}$ is $\Ec_\te$-Cauchy and 
$
	\|f_i\|_{L^2(\Om)} = \|\wt\io (f_i ) \|_{L^2(\Om)} \to 0.
$
Lemma~\ref{lem:FO} yields $\Ec_\te( f_j) \to 0$. Hence, 
$$
	\|f\|^2_{H^1_0(\Om,\te)}=	\lim_{i\to \infty} \| f_i \|^2_\te = \lim_{j\to \infty} \|f_i\|_{L^2}^2 + \lim_{j\to \infty} \Ec_\te( f_i, f_i) =0.
$$
 So, $f=0$ and the proof is complete.
\end{proof}

We are ready to state the embedding theorem.

\begin{thm} \label{thm:embedding} If $f \in W^*(\Om)$ and $\chi \in C^\infty_c(\Om,\bR)$, then $\chi  f \in \cF$.
\end{thm}

\begin{proof} Assume $df \wed d^c f\leq T$ for a closed positive $(1,1)$-current $T$ of finite mass in $\Om$.  Let $f_k = f*\eta_{1/k}$ and $T_k = T*\eta_{1/k}$ be the convolutions with the family of smooth kernels. Denote $g_k = \chi f_k$ which is smooth in $\Om$ for $k$ large. We estimate
$$\begin{aligned}
	\int d g_k \wed d^c g_k \wed \te 
&\leq 2 \int \chi^2 df_k \wed d^c f_k \wed \te + 2 \int f^2_k d\chi \wed d^c \chi \wed \te \\
&\leq 10 \int \chi^2 T_k \wed \te + 4A_0 \int \chi f^2_k \, \om\wed\te\\
&\leq A,
\end{aligned}$$
where we used Lemma~\ref{lem:extra-grad} and then $df_k \wed d^c f_k \leq T_k$ for the second inequality. 
The finiteness in the last inequality follows from the Chern-Levine-Nirenberg inequality and weak convergence of $T_k$ to $T$ for the first integral and Proposition~\ref{prop:L2-cap} for the second integral. 
This implies that there exists a subsequence $\{g_k\}_{k\geq 1}$ converging weakly to $g \in H^1_0(\Om,\te) = \cF$. By passing to a subsequence  and Mazur's lemma \cite[Theorem~2, page 120]{Yo80}  we find a sequence $\{w_j\}_{j\geq 1} \subset C^\infty_c(\Om,\bR)$ defined by
$$
	w_j = \frac{1}{j} \sum_{\ell =1}^j g_{k_\ell} = \chi  \left( \frac{1}{j}\sum_{j=1}^\ell f_{k_\ell} \right)
$$
which converges in $\| \cdot \|_\te$-norm to $g$. Since $w_j \to \chi f $ in $L^2(\Om)$, we have $g = \chi f \in \cF$ by the definition of this space.
\end{proof}

\begin{cor} Let $f, g \in W^*(\Om)$ and $\chi \in C^\infty_c(\Om, \bR)$. Let $\te$ be a closed positive $(n-1,n-1)$-current as in \eqref{eq:the-bdd}. Then, the integrals
$$
	\int df \wed d^c \chi \wed \te \quad
\text{ and } \quad 
	\int \chi^2 df \wed d^c g \wed \te
$$
are well-defined.
\end{cor}

\begin{proof} If $f$ is smooth, then by integration by parts 
\[\label{eq:defn-grad-genenal-1}
	\int df \wed d^c\chi \wed \te = - \int f dd^c \chi \wed \te. 
\]
Next, for a general $f\in W^*(\Om)$, we have $\{f_k\}_{k\geq 1} = \{f* \eta_{1/k}\}_{k\geq 1}$ is Cauchy with respect to the capacity. Passing to a subsequence we may assume  $\wt f_k \to \wt f$ in capacity and  $\wt f$ is quasi-continuous. Hence,  the right hand side above is defined by 
$$
	\int f dd^c \chi \wed \te := \int \wt f dd^c \chi \wed \te = \lim_{k\to \infty} \int \wt f_k dd^c \chi \wed \te.
$$
Thus, we take the limit, which is independent of subsequence, as the definition of the left hand side of \eqref{eq:defn-grad-genenal-1}.

Next, to define $\int \chi^2 df \wed d^c g \wed \te$, by polarization it is enough to define $\int \chi^2 df \wed d^c f\wed\te$ and $\int \chi^2 dg \wed d^c g\wed\te$. We first assume $\te = \te_\de$ satisfying \eqref{eq:add}. It follows from Theorem~\ref{thm:embedding} that 
$$
	\int d(\chi f) \wed d^c(\chi f) \wed \te_\de
$$
is well-defined. Therefore, one can define
$$
	\int \chi^2 df \wed d^c f \wed \te_\de := \int d(\chi f) \wed d^c(\chi f) \wed \te_\de - 2 \int df \wed d^c \chi \wed \te_\de - \int d\chi \wed d^c \chi \wed\te_\de,
$$
where each integral on the right hand side already makes sense. Of course,
this is an honest identity for $f$ being  smooth. By the weak convergence of $\chi f_j \to \chi f$ in the Hilbert space $\cF$,  it follows that 
$
	\int \chi^2 df \wed d^c f\wed \te_\de
$
is a non-negative  increasing sequence in $\de>0$. Finally, we put
$$\int \chi^2 df \wed d^c f\wed \te = \lim_{\de\to 0^+} \int \chi^2 df \wed d^c f\wed \te_\de.
$$
This completes the proof of the corollary.
\end{proof}

Thanks to this corollary and Definition~\ref{defn:L2-ma} we conclude that the integral estimates that obtained in Section~\ref{sec:int} now hold for a general $f\in W^*(\Om)$.

\subsection{Capacity in the Dirichlet spaces}
\label{ss:dirichlet-cap}

Let $\phi \in PSH(\Om)$ be such that $-1 \leq \phi \leq 0$. Let  $\de>0$.  By Theorem~\ref{lem:FO}, the pair $(\te, \mu)$ of the closed positive current 
\[\label{eq:current-te}
\te = [dd^c(\phi + \de |z|^2)]^{n-1}
\] and the measure $\mu = \te\wed \om$ is closable on $L^2(\Om, d\mu)$ and $\supp \mu = \Om$, i.e., it  is admissible in the sense of \cite[page 173]{FO87}. Hence, $\cF$ defined in Definition~\ref{defn:O} is a Dirichlet space on $L^2(\Om, d\mu)$ possessing $C^\infty_c(\Om)$ as its core, together with a Dirichlet form $\cE^\te$. In particular, it is regular \cite[page 6]{FOT94}.

By \cite[Th\'eor\`em~1]{Okada82} we have a Poincar\'e type inequality, for $f\in C^\infty_c(\Om)$,
$$
	\int f^2 d\mu = \int f^2 \te \wed dd^c |z|^2 \leq 8 \sup_\Om |z| \int df \wed d^c f \wed \te.
$$
It follows that $\cE_1^\te(f) = \| f  \|^2_{L^2(d\mu)} + \cE^\te(f)$ is equivalent to $\cE^\te(f)$. As noted in \cite[page 211]{FO84}  that potential theory of \cite[Chapter~2]{FOT94} can be formulated in terms of $\cE^\te$ instead of $\cE^\te_1$.  For an open  subset $E\subset \Om$ its capacity is given by
\[\label{eq:cap-te}
	\tc_\te (E) = \inf\{\cE(v) : v\in \cF, \; v\geq 1 \quad \mu-\text{\rm a.e.} \text{ on } E\}.
\]
For a general Borel set we let
$$
	\tc_\te(E) = \inf \{\tc_\te(G): E \subset G,  \quad G \text{ is open subset in }\Om\}.
$$
In particular, for a compact subset $K\subset \Om$,
$$
	\tc_\te(K) = \inf\{ \cE^\te(v): v\in \cF\cap C^0(\Om), \; v\geq 1 \text{ on }K\}.
$$
Thank to special feature of the functional capacity in a Hilbert space it is  a Choquet capcity \cite[Theorem~2.1.1]{FOT94}. 
Moreover,  we have a comparison between this capacity and the Bedford-Taylor capacity as follows.
\begin{lem}\label{lem:FO-polar} Let $\Om \subset \subset \bC^n$ be a strictly pseudoconvex domain and $0<\de \leq 1$.
\begin{itemize}
\item[(a)] For every compact subset $K\subset \Om$, 
$$
	cap (K) \leq 8 \tc_\te (K),
$$
where $\te= (dd^c u_K^* +\de \om)^{n-1}$.
\item[(b)] Let $D\subset \subset \Om$ be a subdomain. There exists a constant $A_0= A(D,\Om)$ such that  for every open subset $G\subset D$,
$$
	\tc_\te{(G)} \leq A_0 \left[cap (G)\right]^\frac{1}{n}.
$$ 
\end{itemize}
\end{lem}

\begin{proof} (a)
Denote $\phi:= u_K^*$ the relative extremal function of $K$ in $\Om$. Then, $\phi \in PSH(\Om), -1\leq \phi \leq 0$ which satisfies $cap (K) = \int_K (dd^c \phi)^n$. For $v \in C^\infty_c(\Om)$ and $v \geq 1$ on $K$,
$$
	\int_K (dd^c\phi)^n \leq \int_\Om v^2 dd^c\phi \wed \te =:I,
$$ 
where $\te = (dd^c\phi + \de \om)^{n-1}$.  The Poincar\'e type inequality \cite[Th\'eor\`em~1] {Okada82} implies
$$
	I\leq   8\|\phi\|_{L^\infty(\Om)} \int_\Om d v \wed d^c v \wed \te.
$$
Taking infimum over all such $v$'s implies the desired inequality.

(b) Let $\chi \in C^\infty_c(\Om)$ such that $0\leq \chi \leq 1$ and $\chi =1$ in a neighborhood $\bar D$. Let   $\veps>0$ be small. Let $u_G^*= u_G$ be the relative extremal function associated with $G$. By Theorem~\ref{thm:embedding}, $v=\chi u_G \in \cF$. Since   $v =1$ on $G$ we have
\[\label{eq:polar-cap}
	\tc_\te(G) \leq	\int dv \wed d^c v \wed \te^{n-1}.
\]
Let $\rho$ be the strictly psh defining function and $dd^c\rho \geq \om$. By replacing $\phi:= \max\{\phi, A\rho\}$ for $A>0$ large we can assume $\phi = A \rho$ on  $\Om\setminus \supp\chi$. By \eqref{eq:basic-ineq},
$$
	d(\chi u_G) \wed d^c (\chi u_G) \leq 2(\chi^2 du_G \wed d^c u_G + u_G^2 d\chi \wed d^c \chi).
$$
Hence, 
\[\label{eq:polar-grad}\begin{aligned}
	\int dv \wed d^c v \wed \te &\leq 2\int du_G \wed d^c u_G \wed \te + 2 \int u_G^2 d\chi \wed d^c\chi \wed \te \\ &:=2 (I_1 + I_2).
\end{aligned}
\]
By integration by parts,
$$
	I_1 = \int (-u_G) dd^c u_G \wed \te \leq \int dd^c u_G \wed [dd^c (\phi+\de\rho)]^{n-1}
$$
Using the Cegrell's inequality \cite{Ce04} we get
\[\label{eq:polar-i1}
	I_1 \leq \left[\int (dd^c u_G)^n\right]^\frac{1}{n} \left[ \int [dd^c(\phi + \de \rho)]^n  \right]^\frac{1}{n}.
\]
Since $\phi + \de \rho \geq (A+1)\rho$ on $\Om$ and has zero value on $\d\Om$, it follows from the comparison principle 
$$
	\int [dd^c(\phi + \de \rho)]^n  \leq (A+1)^n \int (dd^c\rho)^n:= A_1
$$
which is independent of $G$.
Thus, $$I_1 \leq A_1 [cap (G)]^\frac{1}{n}.$$
The estimate for $I_2$ is very similar. Namely, using $d\chi \wed d^c\chi \leq A' dd^c \rho$ and integration by parts \cite[Proposition~2.1.3]{Blocki-lect-ma},
$$
	I_2 \leq A' \int (-u_G) dd^c \rho \wed \te = A' \int (-\rho) dd^cu_G \wed \te.
$$
Therefore, using again \cite{Ce04} we get for $A_2 = A' A_1 \|\rho\|_{L^\infty}$,
\[\label{eq:polar-i2}I_2 \leq A_2 [cap (G)]^\frac{1}{n}.
\] 
Combining \eqref{eq:polar-cap}, \eqref{eq:polar-grad}, \eqref{eq:polar-i1} and \eqref{eq:polar-i2} we get $\tc_\te (G) \leq A_3  [cap (G)]^\frac{1}{n}$ where $A_3$ is a uniform constant.
\end{proof}

We state now the characterization of pluripolar sets in terms of the family of capacity that was used in the proof of Lemma~\ref{lem:qe-limit}.  This result is due to Fukushima and Okada \cite{FO84} whose proof based on conformal martingale diffusion. Our proof uses only pluripotential theory.

\begin{cor} \label{cor:FO-polar}
Let $E\subset \Om$ be a Borel subset. Then, $cap (E) =0$ if  and only if $\tc_\te(E) =0$ for all $\te $ of the form \eqref{eq:current-te} with $0<\de \leq 1$.
\end{cor}

\begin{proof}  Since both capacities are Choquet's capacities  we may assume $E=K$ is compact. To prove the sufficient  condition  assume that $cap (K)>0$. Lemma~\ref{lem:FO-polar} implies immediately that $\tc_\te(K)>0$ where 
where $\te = (dd^cu_K^* + \de \om)^{n-1}$. 
To prove the necessary condition we assume $cap (K) =0$. 
Let   $\veps>0$ be small. By outer regularity there exists an open subset $G\subset D \subset\subset \Om$ 
such that $K\subset G$ and $cap (G) < \veps$. Using Lemma~\ref{lem:FO-polar}-(b) we infer  $\tc_\te (G) \leq A_0 \veps^\frac{1}{n}$ where $A_0$ is  independent of $\veps$.
 Let $\veps\to 0^+$, then $\tc_\te(G) \to \tc_\te(K)$ and therefore,  $\tc_\te(K) =0$.
\end{proof}

\section{Exponential integrability}

\subsection{Comparison between volume and capacity}

In this section we give a simplification of a quantitative version of  the exponential integrability in \cite{DMV}. Then, using this we obtain several equivalence statements on the comparison between the Lebesgue volume of a Borel set and its capacity.

Let $B_R:=B(0,R)$ be the ball of radius $R>0$ in $\bC^n$. Let $u \in PSH(B_R)$. Denote by $W^*(B_R, u)$ the set of  functions
$f \in W^*(B_R)$ satisfying
\[\label{eq:sobolev-potential} df \wed d^cf \leq dd^c u.\] 

We have the following basic results for smooth functions in this family.

\begin{lem} \label{lem:restriction} Suppose $f \in W^*(B_R, u) \cap C^\infty(B_R)$ and $u\in PSH\cap C^\infty(B_R)$. Then, for $\ze \in \bC^n$ with $|\ze| =R$,   the function $\ell (\la) = f(\la \ze)$ belongs to $W^*\cap C^\infty(\bD)$ where $\bD$ is the unit disc in $ \bC$. Moreover,
$$
	\left| \frac{\d \ell}{\d \la} \right|^2 \leq  \De_\la u(\la \ze)/4.
$$
\end{lem}

\begin{proof} This follows from a simple calculation
$$
	\left|\frac{\d\ell }{\d \la}\right|^2 = \sum_{j,k} \frac{\d f}{\d z_j} \frac{\d f}{\d \bar z_k} \ze^j \bar \ze^k \leq \sum_{j,k}\frac{\d^2 u}{\d z_j \d \bar z_k} \ze^j \bar \ze^k = \De_\la u(\la \ze) /4,
$$
where the second inequality used the fact that $df\wed d^c f\leq dd^c u$.
\end{proof}

\begin{cor} \label{cor:grad-1d} Let $f(z), u(z)$  and $\ell (\la)$ be as in Lemma~\ref{lem:restriction}. Assume further  $u(0)=-M$ with $M>0$. Denote by $\bD_r:= \bD(0,r)$ the disc in $\bC$ for $0<r<e^{-\frac{1}{2}}$. Then, 
$$
	\|\na \ell \|^2_{L^2(\bD_r)} \leq \int_{\bD_r}  \De_\la u (\la \ze)/4 \leq \pi M.
$$
\end{cor}

\begin{proof} The first inequality follows from Lemma~\ref{lem:restriction}. The second inequality follows from the Riesz representation formula for the subharmonic function $L (\la) = u(\la\ze)/M$ in the unit disc $\bD$, see e.g., the proof of  \cite[Proposition~4.2.9]{Ho07}. 
\end{proof}

The next result is  crucial for proving uniform exponential integrability.

\begin{lem} \label{lem:exp-est-key} Let $u\in PSH \cap C^\infty (B(0,1))$ and $f\in W^*(u, B(0,\frac{3}{4}))\cap C^\infty(B(0,\frac{3}{4}))$ with $f\leq 0$. Define
$$
	\ka^2 = \int_{B(0,\frac{3}{4})} |f|^2 dx + \int_{B(0,\frac{3}{4})} dd^c u \wed \om^{n-1}.
$$ Assume $u(0)=-M \ka^2$. There exist uniform constants $A>0$ and $\al>0$ independent of $f$ and $u$ such that 
$$\int_{B(0,\frac{1}{2})} e^{ \frac{\al |f|^2}{M \ka^2}} dx  \leq A.$$ 
\end{lem}

\begin{proof} By considering $f:= f/\ka$ and $u:= u/\ka^2$, we may assume that $$\ka=1 \quad \text{and}\quad u(0)=- M.$$
We start with the estimate of the function restricted on each complex line whose domain is the corresponding disc. Namely, for $\ze \in \bC^n$ with $|\ze|=3/4$ let us consider $\ell(\la) = f(\la \ze)$ as in Lemma~\ref{lem:restriction}. Its average value on the disc $\bD(0,r)$, $0<r<1$, is given by
$$
	(\ell)_r = \intavg_{\bD_r} \ell(\la) dV_2(\la).
$$
By Corollary~\ref{cor:grad-1d} we have for $0<r<e^{-\frac{1}{2}}$,
$$\begin{aligned}
\int_{\bD_r} |\na \ell|^2dV_2  \leq \pi M.
\end{aligned}$$
Applying  Moser-Trudinger's inequality in  $\bD_r \subset \bC=\bR^2$ for the function $\ell(\la) - (\ell)_r$ of vanishing mean (see e.g. \cite{Le05}), we get
\[\label{eq:MT-mean}
	\int_{ \bD_r} e^{\al_0  |\ell (\la) -(\ell)_r |^2} dV_2 (\la) \leq c_0.
\]
where $ \al_0 = 2/M$. 

Let $0 <\tau < \al_0$ to be determined later.  Next, we have  for $\de = (1+\sqrt{17})/4$,
$$	a^2 \leq  (1+\de)(a-b)^2  + \frac{b^2}{2} \quad \forall \;a, b\geq 0
$$
by an elementary inequality. 
It follows that 
$$\begin{aligned} 
	\int_{\bD_r} e^{\tau |\ell(\la)|^2} dV_2 
&\leq \int_{\bD_r} e^{\tau (1+ \de) |\ell - (\ell)_r |^2} e^{\frac{\tau}{2} |(\ell)_r|^2} dV_2 \\
&\leq  \frac{1}{2} \int_{\bD_r} e^{2\tau (1+\de) |\ell - (\ell)_r|^2}  dV_2+ \frac{1}{2} \int_{\bD_r} e^{\tau |(\ell)_r|^2} dV_2,
\end{aligned}$$
where we used the Cauchy-Schwarz inequality for the integrand in the second inequality.  By symmetry,
$$
\int_{\bD_r} e^{\tau |(\ell)_r|^2} dV_2  \leq  \frac{1}{2} \int_{\bD_r} e^{2\tau (1+\de) |\ell - (\ell)_r|^2} dV_2 + \frac{1}{2} \int_{\bD_r} e^{\tau |\ell|^2} dV_2,
$$
The  two  inequalities above imply that
\[\label{eq:slice-mass-a}
	\int_{\bD_r} e^{\tau |\ell(\la)|^2} dV_2  \leq  \int_{\bD_r} e^{2\tau (1+\de) |\ell - (\ell)_r|^2} dV_2. 
\]
Now we can choose 
\[ \label{eq:tau-choice}
\tau = \frac{\al_0}{2(1+\de)} = \frac{4}{(5+ \sqrt{17})M} =: \frac{\al}{M}.
\]
Combining \eqref{eq:MT-mean},  \eqref{eq:slice-mass-a} and  \eqref{eq:tau-choice} we obtain the bound
\[\label{eq:slice-mass-b}
	\int_{\bD_r} e^{\tau |\ell(\la)|^2} dV_2  \leq c_0.
\]

Finally, we estimate the integral on $B(0,R)$ by using the polar coordinates.
$$\begin{aligned}
	\int_{B_R} e^{\tau |f|^2} dV_{2n} 
&= \int_{|\ze| =1} \int_{D_R} e^{\tau |f (\la \ze)|^2} |\la|^{2n-2} dV_2(\la) dS(\ze)\\
&\leq R^{2n-2}  \int_{|\ze| =1}  c_0 R^2 dS(\ze) \\
&\leq	 c_0R^{2n} \sig_{2n-1}.
\end{aligned}$$
This completes the proof as we can take $R=1/2 < e^{-1/2}$.
\end{proof}

We are ready to state the key component in the proof of  \cite[Theorem~1.2]{DMV}. By using Lemma~\ref{lem:exp-est-key} the proof  does not use neither the induction argument  nor  slicing theory for currents.
\begin{prop} \label{prop:DMV} Let $B(0,1) \subset \bC^n$ be the unit ball. There exist positive constants $\al$ and $A$  such that for all $f \in W^*(B(0,1))$ with $\|f\|_*\leq 1$, 
$$
	\int_{\bar B(0,\frac{1}{8})} e^{\al |f|^2} dx \leq A.
$$
\end{prop}

\begin{proof}
Let $T$ be a positive closed $(1,1)$-current in $B:=B(0,1)$ realizing $W^*$-norm of $f$, in particular, 
$$ 
df \wed d^c f\leq T
$$
and 
$$
 	\int_B |f|^2 dx+	\int_{B} T\wed \om^{n-1} \leq 1.
$$
 By a result of Skoda \cite{Sk72} in Lemma~\ref{lem:skoda-lelong}  of the appendix (see also \cite[Lemma~2.1]{Vu20}) we may choose a psh potential $u$ solving
 \[\notag\label{eq:T-psh-potential}
 	T = dd^c u
 \] such that $u \leq 0$ in $B(0, 3/4)$ and
\[\label{eq:skoda-lelong}
	\int_{\bar B(0,\frac{1}{8})} |u| \; dx \leq A_0.
\]
for a uniform constant $A_0$. Then,  $ f \in W^*(B(0,3/4), u)$.

 Assume first that $f,u\in C^\infty(B(0,3/4))$.
Set $-M := \sup_{\bar B(0,\frac{1}{8})} u < 0$. Since $u$ is subharmonic in $B$, $u(a)=-M$ at some  $a\in \d B(0,1/8)$. 
Clearly, 
\[
	0< M \leq \intavg_{B(0,\frac{1}{8})} |u| \;dx \leq \frac{A_0}{|B(0,1/8)|} =: M_0.
\]
Thus, $M$ is uniformly bounded from above by $M_0$.

Consider $v (z) = u(z+a)$ and $g(z) = f(z+a)$. 
Then, $dg \wed d^c g \leq dd^c v$ and $v(0) = -M$. By 
Lemma~\ref{lem:exp-est-key} there exist uniform constants $\al_1>0$ and $A_1>0$ such that
$$
 \int_{ \bar B(0,\frac{1}{4})}e^{\al_1 |g|^2/M} dx\leq A_1.
$$
Since $B(0,1/8) \subset B(a,1/4)$, 
$$	
	\int_{\bar B(0,\frac{1}{8})} e^{\al_1 |f|^2/M} dx \leq \int_{\bar B(a,\frac{1}{4})} e^{\al_1 |f|^2/M} dx =  \int_{ \bar B(0,\frac{1}{4})}e^{\al_1 |g|^2/M} dx\leq A_1.
$$
Set $\al = \al_1/M_0$, we complete the proof in the smooth case.

We now remove the smoothness assumptions on $f$ and $u$. Denote $f_\veps = f*\eta_\veps$ and $u_\veps = u*\eta_\veps$ for small $\veps>0$. Then, on $B_\veps = B(0,3/4-\veps)$, they satisfy
$$
	df_\veps \wed d^cf_\veps \leq dd^c u_\veps \quad \text{on } B_\veps.
$$
Moreover, for $B'= B(0, 5/8)$,
$$
	\int_{B'} |f_\veps|^2 dx + \int_{B'} dd^cu_\veps \wed \om^{n-1} \leq 1.
$$
The argument in the first part shows 
$
	\int_{\bar B(0,\frac{1}{8})} e^{\al |f_\veps|^2} dx \leq A_2.
$
As $f_\veps \to f$ in $L^1(B')$, we may assume the convergence is almost everywhere after passing to a subsequence. Hence,  the conclusion follows from Fatou's lemma.
\end{proof}

\begin{cor}\label{cor:exp-int}  
Let $K \subset \subset \Om$ be compact subset.   There exist positive constants $\al, A>0$ depending only on $K$ and $\Om$ such that 
$$
	V_{2n} (\{|f| \geq s\} \cap K) \leq A e^{-\al s^2}, \quad s>0
$$
for every $f\in W^*(\Om)$ and $\|f\|_* \leq 1$.
\end{cor}

\begin{proof} Since $K$ is compact, there is a finite covering by  balls $B(a_i,r_i)$ such that $B(a_i, 8 r_i) \subset \Om$. By translation and dilation it is enough to estimate the volume of $\{|f|>s\} \cap B(0, 1/8)$. It follows easily from Proposition~\ref{prop:DMV} that
$$\begin{aligned}
	V_{2n} \left(\{|f|\geq s\} \cap \bar B(0, 1/8)\right) 
\leq \int_{\bar B(0,1/8)} e^{\al (|f|^2 - s^2)} dx \leq A e^{-\al s^2}. 
\end{aligned}$$
This is the desired estimate.
\end{proof}

\begin{remark} The uniform constants $\al$ and $A$ can be computed explicitly in the proof of Proposition~\ref{prop:DMV}. So are the ones in the corollary. However, they are very far from optimal ones and thus we will not produce them here. It would be very interesting to get the optimal ones for the unit ball.
\end{remark}

We are in position to characterize this decay of volume.

\begin{proof}[Proof of Theorem~\ref{prop:equiv-est-intro}] By  Corollary~\ref{cor:exp-int} the proof will follows immediately from the list of equivalent statements below. 
\end{proof}

\begin{prop} \label{prop:equiv-est}  Recall  $\tc(E)=\tc(E,\Om)$ for a Borel set $E$.
Let $K\subset\subset \Om$ be a compact subset.  The following statements  are equivalent to each other.
\begin{itemize}
\item[(a)] 
There exist constants $A_1>0$ and $\al>0$ depending only on $K, \Om$ such that for every  Borel  set $E\subset K$ $$V_{2n}(E) \leq A_1 e^{-\al/\tc (E)}.$$

\item[(b)] There exist uniform constants $A_1>0$ and $\al>0$ depending only on $K, \Om$ such that for every  $f\in W^*(\Om)$ and $f\leq 0$,
 $$V_{2n}(\{f< -1\} \cap K) \leq A_1 e^{\frac{-\al}{\|f\|_*^{2}}}.$$
\item
[(c)]   There exist uniform constants $A_1>0$ and $\al>0$ depending only on $K, \Om$ such that for every $f\in W^*(\Om)$ and $f\leq 0$ whose norm $\|f\|_* \leq 
\ka$,
$$V_{2n} (\{f < -s\} \cap K) \leq A_1 e^{-\al s^2/\ka^2}, \quad \forall s>0.$$

\item 
[(d)]  There exist uniform constants $A_1>0$ and $\al>0$ depending only on $K, \Om$ such that for every $f\in W^*(\Om)$ and $f\leq 0$ with $\|f\|_* \leq 1$,
$$V_{2n} (\{f < -s\} \cap K) \leq A_1 e^{-\al s^2}, \quad \forall s>0.$$
\end{itemize}
\end{prop}

\begin{proof} We shall prove 
$(a) \Rightarrow (b) \Rightarrow (c) \Rightarrow (d) \Rightarrow (a).$
We start with (a) $\Rightarrow$ (b). Let $\wt f$ be a quasi-continuous modification of $f$ from Theorem~\ref{thm:quasi-mod}.  Since $f = \wt f$ a.e, we have $V_{2n}(f<-1) = V_{2n}(\wt f<-1)$ and $\|f\|_* = \| \wt f\|_*$.  
Corollary~\ref{cor:cap-sublevel-gen} gives  $\tc (\{\wt f<-1\} \cap K) \leq \|f\|_*^2$. Applying (a) for $E = \{\wt f<-1\} \cap K $ we get
$$
	V_{2n} (\{f<-1\} \cap K) \leq A_3 e^{-\al/\|f\|_*^2}.
$$
To prove (b) $\Rightarrow$ (c) we can consider $g= f/s$. Then, $\|g\|_*^2 = \|f\|_*^2/s^2$ and the proof follows easily. Next, (c) $\Rightarrow$ (d) is obvious by considering $f:= f/\ka$. Finally, we  prove (d) $\Rightarrow$ (a). By outer regularity of both sides, it is enough to prove for open subsets $E$.  Let $f=\Phi_E$ be the extremal function for $E$ in Lemma~\ref{lem:extr-funciton-o}. In particular, $f\in W^*(\Om)$ and $f= -1$ a.e. on $E$. Notice that $\tc (E) = \|f\|_*^2$. Applying (d) for $g = f / \sqrt{\tc (E)}$ we have for $\veps>0$ small,
$$
	V_{2n}(E) \leq V_{2n} (g< (-1+\veps)/\sqrt{\tc (E)})  \leq A_2 e^{-\al (1-\veps)^2/\tc(E)}.
$$
Let $\veps\to 0$ we get the proof of (a). 
\end{proof}

\begin{remark}\label{rmk:sharp-DMV} The inequalities in Proposition~\ref{prop:equiv-est} are sharp as far as the exponents are concerned. In fact, it follows from the first inequality (a) and Lemma~\ref{lem:c-cap} that for a Borel set $E\subset \Om$, 
$$
	V_{2n} (E) \leq A_1 e^{-\al/ \tc (E)} \leq A_1 e^{-\al/ [cap(E, \Om)]^\frac{1}{n}}.
$$
We know from \cite[Theorem~A]{ACKPZ09} that the exponent $1/n$ in the inequality between the  volume and Bedford-Taylor capacity is sharp. So are the exponents in the inequalities of the proposition. Consequently, the exponent 2 in Proposition~\ref{prop:DMV} and \cite[Theorem~1.2]{DMV} is the optimal one for all $n\geq 1$.
\end{remark}

\begin{remark}\label{rmk:DMV-holder} Let $\Om$ be a bounded strictly psedocovex domain and $K\subset \subset \Om$ be a compact subset. Let $u_1,...,u_n$ be H\"older continuous psh functions in $\Om$. It follows from \cite[Proposition~2.4]{DKN} that the Radon measure 
$$\mu = {\bf 1}_K \, dd^c u_1 \wed \cdots \wed dd^c u_n$$
is $W^*(\Om)$-H\"older continuous. That means there exist positive constants $c>0$ and $\al>0$  such that for $f\in W^*(\Om)$ and $\|f\|_* \leq 1$,
$$
	\left| \int_\Om f d\mu \right| \leq  c \|f\|_{*}^\al.
$$
Combining this inequality and the arguments in the proof of \cite[Proposition~2.9]{Ng18} we get that there exist $\al_1>0$ and $c_1>0$ such that   for $f \in W^*(\Om)$, $f\leq 0$ and $\|f\|_* \leq 1$, 
$$
	\mu (f < -s) \leq c_1 e^{-\al_1 s^2},  \quad  \forall s>0.
$$
In other words, the equivalent inequalities in Proposition~\ref{prop:equiv-est} holds for a very large family of Monge-Amp\`ere measures associated with  H\"older continuous psh functions. This result is also equivalent to the statement of \cite[Theorem~1.2]{DMV}.
\end{remark}

\subsection{The Alexander-Taylor type inequality}
\label{ss:AT}
 In this section  we compare the new capacity $\tc (\cdot)$ and another classical capacity defined by means of extremal functions. 

Let $B_R = B(0,R)$ and $B_r:=B(0,r)$ be the balls in $\bC^n$ centered at the origin with radius $0<r< R$.  Let $K\subset B_r$ be a compact set, then  the Siciak-Zaharjuta extremal function $U_K^*$ is given by
$$	U_K (z) = \sup\left\{v(z) \rv v\in \Lc(\bC^n), \; v\leq 0 \text{ on } K\right \}.
$$ 
Here we recall that 
\[\label{eq:L-class}
	\Lc(\bC^n) = \{ v\in PSH(\bC^n): \limsup_{z\to \infty} [v(z) -\log|z|] \leq O(1)\}
\]
Denote $M_K := \sup_{B_R} U_K^*$. We define another capacity by
\[\label{eq:AT-cap}
	T_R(K) = e^{- M_K} = e^{-\sup_{B_R} U_K^*}.
\]

\begin{lem} \label{lem:AT-type} There exist positive constants $A_r$ and $A_R$ depending on $r,R$ respectively such that for every compact  set $K\subset B_r$,
$$ \exp(- A_R \; \tc (K,B_R)^{-1}) \leq  T_R(K) \leq 	\exp(- A_r\;\tc (K, B_R)^{-1/n}).
$$ 
Equivalently,
\[\label{eq:AT}
	\frac{1}{A_r} \frac{1}{M_K^n} \leq \tc (K,B_R) \leq \frac{A_R }{M_K}.
\]
\end{lem}

\begin{proof}  
The Bedford-Taylor capacity satisfies $cap(K, B_R) \geq 1/M_K^n$ by \cite[Theorem~2.1]{AT84}. The first inequality in \eqref{eq:AT} follows readily from the fact in Lemma~\ref{lem:cap-c}  that
$cap(K,B_R) \leq A(r) \tc (K,B_R)$ for some constant $A(r)$. To prove the second inequality we follows closely \cite[Theorem~11.1]{De89}. Consider $B':= B(0, eR)  \supset\supset B_R$. Put  $M= \sup_{B_R} U_K$. We have $U_K^* \leq M+ 1$ on $B'$ by log-convexity. Define
 $$u := \frac{U_K^*-M-1}{M+1}.$$
Clearly, $-1\leq u \leq 0$ on $B'$ and $u =-1$  quasi-everywhere on $K$. 
By definition of capacity and Lemma~\ref{lem:defn-cap-mod},
$$
	\tc(K,B_R) \leq  \|  u\|_{W^*(B_R)}^2.
$$
It remains to estimate the norm $\|u\|_*$ on $B_R$. 

Now, there exists $z_0\in \bar B_R$ such that $U_K^*(z_0) =M$, hence $u(z_0) = -1/(M+1)$. 
As $u \leq 0$ in $B'$ the mean value inequality and Harnack inequality show that
$$
	u(z_0) \leq A_1 \int_{B'} u \om^n,
$$
where $A_1 = A(R)$. This implies 
$$
	\|u\|_{L^1(B')} \leq -\frac{u(z_0)}{A_1} \leq \frac{A_2}{M}.
$$
Using the fact $|u|\leq 1$ we get
$$
	 \int_{B_R} u^2 \om^n \leq \|u\|_{L^1(B')}\leq \frac{A_2}{M}.
$$
Let $\chi$ be a cut-off function such that $\chi \equiv 1$ on $B_R$ and $\supp\chi \subset B'$. Then,
$$\begin{aligned}
	\int_{B_R} dd^c (1+u)^2\wed \om^{n-1} 
&\leq  \int_{B'} \chi dd^c (1+u)^2\wed \om^{m-1} \\
&= \int_{B'} (1+ u)^2 dd^c \chi\wed \om^{n-1} \\
&= \int_{B'} [2u+u^2] dd^c \chi \wed\om^{n-1}.
\end{aligned}$$
The first term in the bracket is estimated  as follows:
$$
	\int_{B'} 2u dd^c \chi \wed \om^{n-1} \leq A_0 \int_{B'} |u| \om^n \leq \frac{A_0 A_2}{M}. 
$$
The second one is  less than
$$
	A_0 \int_{B'} u^2 \om^{n} \leq A_0 \int_{B'} |u| \om^n \leq \frac{A_0 A_2}{M}.
$$
Therefore,
$$
	\int_{B_R} dd^c (1+u)^2 \wed \om^{n-1}	\leq \frac{2A_0A_2}{M}.
$$
Since 
$
du \wed d^c u \leq dd^c (1+u)^2
$, it follows that $\tc (K,B_R) \leq 2(1+A_0) A_2/M.$
\end{proof}

Now we can conclude the sharpness of exponents of inequalities between capacities in the introduction.

\begin{remark}[Proof of the last statement in Theorem~\ref{thm:intro-cap-comp}]
\label{rmk:sharp-exp} The Alexander-Taylor inequality \cite{AT84} reads
$$
	\exp (-A_r/cap(K,\Om)) \leq T_R(K) \leq \exp(-2\pi/[cap(K,\Om)]^\frac{1}{n}),
$$
where the exponents in both inequalities are sharp (\cite[Remark~2]{AT84}).
Together with $cap(K) \leq A_r \tc(K)$ in Lemma~\ref{lem:cap-c}
and the first inequality in Lemma~\ref{lem:AT-type} we derive 
$$
	\exp (-A/cap(K,\Om)) \leq \exp (-A/\tc(K,\Om))    \leq \exp(-2\pi/[cap(K,\Om)]^\frac{1}{n})
$$
with the sharp exponents. So are the ones in Lemma~\ref{lem:c-cap} and Lemma~\ref{lem:cap-c}.
\end{remark}

\section{Lebesgue points} 

In this section we study the set of Lebesgue points of functions in $W^*(\Om)$.  We simplify the proof of \cite[Theorem~1.1]{VV24}. Moreover, we show  that the precise representation of $f$ in $W^{1,2}(\Om)$ is a quasi-continuous modification.

A point $x$ is called a Lebesgue point of $f\in L^1_{\rm loc}(\Om)$ (with respect to the Lebesgue measure) if 
\[\label{eq:LP-def}
	\lim_{r\to 0} \intavg_{B(x,r)} |f-f(x)| dy =0
\]
It follows from \eqref{eq:LP} that the limit exists almost everywhere. Hence,
we can define for $f\in L^1_{\rm loc}(\Om)$ its {\em precise representative} as follows.
\[\label{eq:rep} f^\star(x) \equiv
\begin{cases}
\lim_{r\to 0} \intavg_{B(x,r)} f dy &\quad \text{if this limit exits,} \\
=0 &\quad \text{otherwise}.
\end{cases}	
\] 
Note that if $f=g$ in $L^1_{\rm loc}(\Om)$, then $f^\star = g^\star$ for all points $x\in \Om$. Moreover, if $f\in W^*(\Om)$ and $f = g$ a.e., then $g$ is also in $W^*(\Om)$. Thus, to study the fine properties of $f$ we need to turn our attention to its precise representative $f^\star$.

 Thanks to Alexander-Taylor's type inequality in Section~\ref{ss:AT}  we construct a psh subextension for a given element in $W^*(\Om)$. This result is similar to  El Mir's theorem \cite{EM79} in pluripotential theory. Such a subextension has been obtained recently by Vigny and Vu \cite{VV24}. 
 Compared to \cite{VV24} our psh subextension has a smaller exponent, however, it is a function in the Lelong class \eqref{eq:L-class} (minimal growth at infinity)   which satisfies the inequality everywhere.  The proof given below is also somewhat simpler.

\begin{thm} \label{thm:subextension} 
Let $f \in W^*(B(0,1))$ be such that $f\leq -1$. Assume $\|f\|_*\leq 1$. Let $\veps>0$  and $0<s<1$. Then, there exists $u \in \cL(\bC^n)$ such that $u\leq -|f^\star|^{\frac{2}{n}-\veps}$ on $B(0,s)$ everywhere.
\end{thm}
\begin{proof} Assume first $f$ is continuous. We follows closely  the proof in \cite[Theorem~11.4]{De89} by Demailly.  
As $f\leq -1$, it follows that $\|f\|_*>0$. For $t\geq -1$, set $G_t = \{z\in B(0,s): f(z) <-t\}$ and let  $U_t^* \in\Lc(\bC^n)$ be the Siciak-Zaharjuta extremal function of $G_t$. Since $G_t$ is open, $U_t^* =0$ on $G_t$. Denote $M(t) = \sup_{B(0,1)} U_t^*$ and 
$$
	u(z) = \frac{1}{\veps} \int_1^\infty t^{-1-\veps} \left[U_t^* - M(t)\right] dt.
$$
We have from Lemma~\ref{lem:cap-sublevel-set} that
$
	\tc (G_t) \leq \|f\|_*^2/t^2.
$
Therefore, by the Alexander-Taylor type inequality \eqref{eq:AT},
$$
	M(t) \geq \frac{a_1}{\left[\tc (G_t)\right]^{1/n}} \geq a_1t^{2/n},
$$
where we used the fact $a_1/\|f\|_*^\frac{2}{n} \geq a_1>0$ in the second inequality. Here $a_1>0$ is possibly a small number.  Since $U_t^* - M(t) \leq 0$ on $B(0,1)$,  it follows from the logarithmic convexity that 
$U_t^* - M(t) \leq \log_+|z|$. Hence, 
$$
	u(z) \leq \log_+|z|.
$$
Now, for $z\in B(0,s) \cap G_t$, we have $U_t^*=0$. Hence,
$$
	u(z) \leq - \frac{1}{\veps} \int_1^{|f(z)|} t^{-1-\veps} M(t) dt \leq - \frac{a_1}{\veps} \int_1^{|f(z)|} t^{-1-\veps+2/n} dt. 
$$
A direct computation of the last integral yields
$$
	u(z) \leq -\frac{a_1}{\veps} \left(\frac{2}{n} -\veps \right) |f(z)|^{\frac{2}{n}-\veps} + \frac{a_1}{\veps} \left(\frac{2}{n} -\veps\right).
$$
Therefore, we may start with a smaller value of $\veps$ if necessary and subtracting a constant from  $u$ we will get that
$$
	u \leq -|f|^{\frac{2}{n} -\veps} \quad\text{ on } B(0,s).
$$
It remains to verify that $u$ is not identically $-\infty$. By the logarithmic convexity again, we have
$$
	\sup_{\bar B(0,1/2)} U_t^* \geq M(t) - \log 2.
$$
Therefore, there exists $z_0\in S(0,1/2)$ such that $U_t^*(z_0) - M(t) \geq -\log 2$. The Harnack inequality shows that
$$
	\frac{1-1/4}{(1+1/2)^{2n}} \int_{\d B(0,1)} (U_t^* - M(t)) d\sig(z) \geq U_t^*(z_0) - M(t) \geq -\log 2.
$$
Integrating this with respect to $t$ we get
$$
	\int_{\d B(0,1)} u(z) d\sig(z \geq -\frac{4}{3} (3/2)^{2n} \log 2.
$$

Now, assume $f \in W^*(B(0,1))$ in general. We consider the continuous approximation $f_r(x) =(f)_{x,r}$ defined in \eqref{eq:mean}. Choose $r>0$ small so that $B(0,s) \subset \subset B(0,1-r)$. The previous step gives  $u_r \in \Lc(\bC^n)$ satisfying  $u_r \leq - |f_r|^{\frac{1}{n} -\veps}$ on $B(0,s)$.
Define on $B(0,1)$ the function
$$
	u = \left(\limsup_{r\to 0} u_r \right)^*.
$$
By the Lebesgue differentiation theorem there exists a set $E\subset B$ whose $|E| =0$ such that $\lim_{r\to 0} (f)_{r,x} = f^\star(x)$. Thus,  for $x\in B\setminus E$, we have $u (x)  \leq - |f^\star (x)|^{\frac{1}{n}-\veps}.$ On the other hand,  if $x\in E$, then $u(x) \leq 0 = f^\star(x)$ by definition.
\end{proof}

We need later a slight improvement of quasi-continuous modification in Theorem~\ref{thm:quasi-mod}.

\begin{lem}\label{lem:quasi-fine} Let $\wt f$ be the quasi-continuous modification of $f\in W^*(\Om)$. Then, $\wt f$ is pluri-fine continuous quasi-everywhere. This means there exists a pluripolar set $F$ such that $\wt f$ is plurifine continuous on $\Om\setminus F$.
\end{lem}

\begin{proof} It follows from Fuglede \cite[Theorem~3]{Fu65} (see also \cite[Theorem~IV, 3.]{Brelot71}) that there is a Borel subset $F\subset \Om$ with  $\tc (F,\Om) =0$ such that $\wt f$ is pluri-fine continuous on $\Om\setminus F$ with respect to the capacity $\tc(\cdot)$. Thanks to Corollary~\ref{cor:pluripolar-sets-c}, $F$ is pluripolar. 
\end{proof}

We state now the main result of this section about the Lebesgue points due to Vigny and Vu \cite{VV24}. Such properties of this set in the usual Sobolev spaces hold for its corresponding Sobolev functional capacity (see e.g., \cite[Section~4.8]{EG92}). Now they are stated for the functions  in complex Sobolev space with respect to $\tc(\cdot)$. The crucial tool in \cite{EG92} is the maximal function which is  no longer available in the new space. Hence, the proofs we give here are very different. We state a slight improvement of \cite[Theorem~1.1]{VV24} whose proofs follows the strategy in \cite{VV24} but it is simpler.

\begin{thm} \label{thm:VV} Let $\Om \subset \subset \bC^n$ be an open subset and let $f\in W^*(\Om)$. There exists a Borel subset $E\subset \Om$ such that $\tc (E) =0$ and it satisfies 
\begin{itemize}
\item
[(i)] for each $x\in \Om\setminus E$,
$$
	\lim_{r\to 0} \intavg_{B(x,r)} f dy = f^\star(x).
$$
\item
[(ii)] Moreover, for each $x\in \Om\setminus E$,
$$
	\lim_{r\to 0} \intavg_{B(x,r)} |f-f^\star(x)|^2 dy =0.
$$
\end{itemize}
Consequently, the precise representative $f^\star$ is quasi-continuous.
\end{thm}

\begin{proof} We observe that the statements (i) and (ii) are local ones. In fact, assume that  for each relatively compact open subset $\Om' \subset \subset \Om$, there exists a Borel set $E$ such that $\tc (E,\Om) =0$ and (i) and (ii) hold for $x\in \Om'\setminus E$. Take an exhaustive sequence $\Om_j \subset \subset \Om$ and $\Om = \cup_j \Om_j$ and sequence $E_j \subset \Om$ whose $\tc(E_j, \Om)=0$ such that (i) and (ii) hold on $\Om_j \setminus E_j$. By the subadditivity of capacity \eqref{prop:basic}-(c) we set
$$
	E = \cup_j E_j.
$$
Then the statements (i) and (ii) will hold on $\Om \setminus E$ where $\tc(E,\Om) =0$.

Next, since  $cap(\cdot,\Om)$ dominates $\tc(\cdot)$ (Lemma~\ref{lem:c-cap}), it is enough to find $E$ such that $cap(E,\Om) =0$. Using well-known properties of subadditivity and locally comparable of this capacity we may assume $\Om = B(0,1)$ and $f$ defined in a neibhorhood of $\bar B(0,1)$. According to Theorem~\ref{thm:subextension} there exists $u\in \Lc(\bC^n)$ such that 
\[ \label{eq:subext} u \leq - |f^\star|^\frac{1}{2n} \text{ on }\bar B(0,1).
\] Put 
\[ \label{eq:Leb-polar-pt}
P := \{z\in B(0,1): u = -\infty\}.
\]
This is a part of the set $E$ that we need to construct.  The following results are analogous to \cite[Lemmas~3.1, 3.2]{VV24}. 
\begin{lem}\label{lem:VV-key} Let $x\in B(0,1)\setminus P$ where $P$ is defined in \eqref{eq:Leb-polar-pt}. Let $u$ be the psh subextension in \eqref{eq:subext}. Assume  $v \in PSH(B(0,1))$  such that $v(x)>0$.  Then,
\begin{align}
\tag{a}
\label{eq:fine-set}	&\lim_{r\to 0} \frac{| B(x,r) \cap \{v\leq 0\} |}{|B(x,r)|} =0. \\
\tag{b} \label{eq:fine-psh}	&\lim_{r\to 0} \frac{1}{|B(x,r)|} \int_{B(x,r) \cap \{v \leq 0\}}  |u|^{2n} dy =0.
\end{align}
\end{lem}

\begin{proof} It is easy to see that the first identity is equivalent to
$$
	\lim_{r\to 0} \frac{1}{|B(x,r)|} \int_{B(x,r) \cap \{v>0\}} v(x) dy = v(x).
$$
To this end  let $\veps>0$. By upper semi-continuity of $v$ there exists $r_\veps>0$ such that $B(x,r_\veps) \subset \{v< v(x)+\veps\}$. Hence, for $0<r <r_\veps$,
$$
	  \int_{B(x,r) \cap \{v>0\}} (v-\veps) dy \leq  \int_{B(x,r) \cap \{v>0\}} v(x) dy.
$$
This implies 
$$
	\intavg_{B(x,r)} v dy -\veps \leq \frac{1}{|B(x,r)|}  \int_{B(x,r) \cap \{v>0\}} v (x) dy.
$$
Together with  the subharmonicity of $v$ we obtain
$$
	v(x) - \veps \leq \frac{1}{|B(x,r)|}  \int_{B(x,r) \cap \{v>0\}} v(x) dy \leq v(x).
$$
Letting $r\to 0$ and then $\veps \to 0$ we get the desired identity.	

Next, we prove the second identity. Let $0<\veps < - u(x)$. Since $u$ is upper semicontinuous and $u(x)>-\infty$, we have $B(x,r_\veps) \subset \{u < u(x)+\veps\}$ for some $r_\veps >0$ small. 
Therefore, for $0< r<r_\veps$,
$$\begin{aligned}
\frac{1}{|B(x,r)|}\int_{B(x,r) \cap \{v \leq 0\}} |u|^{2n} dy 
&\leq 
	\frac{(|u(x)|+\veps)^{2n}}{|B(x,r)|}\int_{B(x,r) \cap \{v \leq 0\}}  dy 
\end{aligned}$$
It follows from the first identity that the right hand side goes to zero as $r\to 0$.
\end{proof}

Now we follows closely \cite{VV24} for the remaining part of the proof.  Let $\wt f$ be a quasi-continuous modification of $f$. By Lemma~\ref{lem:quasi-fine} there is a pluripolar set $F$ such that $\wt f$ is pluri-fine continuous in $B(0,1)\setminus F$.
 Define
\[
	E : = P \cup F.
\]
We shall show that the pluripolar set $E$ satisfies the requirement of the theorem. In fact,  let $x\in B(0,1)\setminus E$. We claim that 
\[ \label{eq:fine-point-id} f^\star(x) = \wt f(x), \]
or equivalently 
$$
	\lim_{r\to 0} \intavg_{B(x,r)} f dy = \lim_{r\to 0} \intavg_{B(x,r)} \wt f dy = \wt f(x).
$$
Indeed, the first identity follows from the fact that $f = \wt f$ a.e in $B(0,1)$. To prove the second identity we use the Cartan's theorem relating the fine  and ordinary limits (\cite[Theorem~2.4]{BT87} and \cite[Theorem~III, 2.]{Brelot71}). To this end let $V$ be a pluri-fine open neighborhood of $x$. It follows from \cite[Theorem~2.3]{BT87} that  we can assume $$V = \{|z-x| < a\} \cap \{v>0\}$$ for some constant $a>0$ and $v$ is psh in $B(0,1)$ such that $v(x)>0$. Then, Cartan's theorem reads 
$$
	\lim_{ \\ \substack{z \to x \\ z\in V}} \wt f(z) = \wt f(x).
$$
By Lebesgue's theorem $f^\star = f$ a.e and $|f^\star| \leq |u|^{2n}$ in \eqref{eq:subext}. So, $|\wt f| \leq |u|^{2n}$ almost everywhere. Hence,  Lemma~\ref{lem:VV-key}-(b) implies
$$
	 \lim_{r\to 0} \frac{1}{|B(x,r)|} \int_{B(x,r) \cap \{v\leq 0\}} \wt f dy =0.
$$
These combined with Lemma~\ref{lem:VV-key}-\eqref{eq:fine-set} implies
$$\begin{aligned}
	 \lim_{r\to 0} \frac{1}{|B(x,r)|} \int_{B(x,r)} \wt f dy  &=  \lim_{r\to 0} \frac{1}{|B(x,r) \cap \{v>0\}|} \int_{B(x,r) \cap \{v>0\}} \wt f dy  \\
	&= \wt f(x),
\end{aligned}$$
Thus, \eqref{eq:fine-point-id} is proved.

Finally, we have  
$$\begin{aligned}
	\intavg_{B(x,r)} |f - f^\star(x)| \;dy 
&=  \intavg_{B(x,r)} |\wt f - f^\star(x)| \; dy \\
&= \frac{1}{|B(x,r)|} \int_{B(x,r) \cap V} |\wt f - f^\star(x)| \; dy \\ &\quad +  \frac{1}{|B(x,r)|} \int_{B(x,r) \setminus V} |\wt f - f^\star(x)| \;dy.
\end{aligned}$$
The first integral on the right hand side goes to zero as $r\to 0$ by the fine continuity of $\wt f$ at $x$ and $\wt f(x) = f^\star(x)$. The second integral is bound by
$$
\frac{1}{|B(x,r)|} \int_{B(x,r) \setminus V} (|u|^{2n} + |u(x)|^{2n}) dy
$$
which tends to zero as $r\to 0$ by  Lemma~\ref{lem:VV-key}.
This completes the proof of (i).

Next, to prove (ii) we note that for  $x\in B(0,1)\setminus E$,
$$\begin{aligned}
	\lim_{r\to 0} &\left( \intavg_{B(x,r)} |f- f^\star(x)|^2 dy \right)^\frac{1}{2} \\
&\leq \lim_{r\to 0} | (f)_{x,r} - f^\star(x)| + \lim_{r\to 0} \left( \intavg_{B(x,r)} |f- (f)_{x,r}|^2 dy\right)^\frac{1}{2} \\
&=0,
\end{aligned}$$
where we recall that $f_{r}(x) = (f)_{x,r} = \intavg_{B(x,r)} f dy$.

Lastly, we prove the quasi-continuity of $f^\star$. According to Remark~\ref{rmk:cauchy} the sequence $\{f_{r_j}\}_{r_j}$ with  $r_j =1/j$ is a Cauchy sequence in capacity. There exists a subsequence converges point-wise to a quasicontinuous function $g$. Moreover, given $\veps>0$ there is an open set $U\subset \Om$ with $\tc (U) < \veps$ (Theorem~\ref{thm:quasi-mod}) and the subsequence converges uniformly to $g$ on $\Om \setminus U$. Since $f_{r_j}(x) \to f^\star(x)$ on $\Om\setminus E$ by (i), it follows that $g(x)=f^\star(x)$ on $\Om\setminus E$. In other words, $f^\star$ is quasi-continuous. 
\end{proof}

\section{Appendix}

In this section we give details of the following result due to Skoda \cite{Sk72} which was needed in the proof of Proposition~\ref{prop:DMV}.  Let $T$ be a positive closed $(1,1)$-current in the unit ball $B(0,1)$. Define $\be=  \ddbar |z|^2$ and assume that 
$$
	\int_{B(0,1)} T \wed \be^{n-1} \leq 1.
$$

\begin{lem} \label{lem:skoda-lelong}  There exists a psh potential $U(z) \leq 0$ in $ B(0,3/4)$ such that
 $$
 	dd^c U = T
 $$
and 
moreover, there exist uniform constants $c_0, A>0$ independent of $U$ (which are explicit computable) satisfying
 $$
 	\int_{|z| \leq \frac{1}{8}} \exp \left[ - \frac{U(z)}{c_0} \right] dz \leq  A.
 $$
\end{lem}

\begin{proof}
The first statement in the lemma is classical which can be obtained by various means. However, the extra property of uniform exponential integrability need to be carefully examined in the construction given in Skoda \cite[Section~7]{Sk72}.  By adding to $U$ a strictly psh function $\rho (z) = |z|^2-1$  we may assume that 
 \[ \label{eq:red-post} T \geq \be \quad \text{in } B(0,1).
 \]
Assume $0\leq \eta \leq 1$ is a cut-off function such that $\eta =1$ on $\bar B(0, 3/4)$ and $\supp\eta \subset B(0,1)$. For simplicity we keep the notations in \cite{Sk72} by writing $p= n-1$ and define 
$$
	U (z) = - \int |z-x|^{-2p} \eta (x) \;  T(x) \wed \be^p.
$$ 
By computation in \cite[Section~3]{Sk72} or  \cite[Lemma~4.3.4]{Ho07} we have
$$
	T(z) = \ddbar U (z) \quad\text{on } B(0,3/4).
$$

Next, we  verify the uniformly exponential integrability on $\bar B(0,1/8)$ of $U(z)$. To this end we 
denote $\al = \ddbar \log |z|^2$ and $E (z) = \log|z|^2$.  The identity \cite[Eq. (7.13)]{Sk72} reads
$$\begin{aligned}
	U(z) 
&= p  \int E(x-z) \al^p \wed \eta (x) T(x) \\
&\quad - \int \ii\; \dbar_x E(x-z) \wed \al^{p-1} \wed \d_x \eta \wed T \\
&\quad - p  \int \ii E\, \dbar_x E \wed \al^{p-1} \wed \d_x \eta \wed T.
\end{aligned}$$
If $z \in \bar B(0,1/4)$, then the last two integrals are smooth functions as $\eta =1$ on $\bar B(0,3/4)$. More precisely, 
$$
	|\d_x E(x-z)| \leq \frac{A_n}{|x-z|}, \quad  |\d_x\bar\d_x E(x-z)| \leq \frac{A_n}{|x-z|^2}.
$$
Hence, for $z\in \bar B(0,1/4)$ these two integrals are uniformly bounded by an absolute constant $A_0$. It remains to estimate 
$$
	U_3(z) = \frac{p }{2^p \pi^p}\int_{\bC^n} \log|z-x|^2 \al^p(z-x) \wed \eta(x) T(x).
$$
We can yet replace $U_3(z)$ by 
$$
	U_4(z) =\frac{p }{2^p \pi^p} \int_{|x|\leq R} \log |z-x|^2 \al^p(z-x) \wed T(x)
$$
for $z\in \bar B(0,1/4)$ and $1/2 \leq R < 3/4$
because their difference is uniformly bounded on $\bar B(0,1/4)$. 
Put 
\[
	\mu(z) = \int_{|x| \leq R} \al^p(z-x) \wed T(x) \quad\text{and}\quad
	\ka(z) = \frac{\mu(z)}{2^p \pi^p} \frac{n}{c}.
\]
We rewrite
$$
	- \frac{n}{pc} U_4 (z) = \int_{|x|\leq R} \log [ |z-x|^{- \ka(z)}] \frac{\al^p(z-x) \wed T(x)}{ \mu(z)}.
$$
Using the Jensen inequality and concavity of the logarithmic function we get
\[
	- \frac{n}{pc} U_4(z) \leq \log \left[  \int_{|x| \leq R} |z-x|^{-2 \ka(z)} \frac{\al^p(z-x) \wed T(x)}{\mu(z)}  \right] , 
\]
Therefore, for $r = 1/8$ and  $R=1/2$ we have
\[ \label{eq:exp-U4} \begin{aligned}
 &\int_{|z| \leq \frac{1}{8}} \exp \left[ - \frac{n}{pc} U_4(z) \right] dz \\
&\qquad\qquad \leq \iint_{|z| \leq \frac{1}{8}, |x| \leq \frac{1}{2}} |z-x|^{-2\ka(z)} \frac{\al^p(z-x) \wed T(x)}{\mu(z)}.
\end{aligned}\]

We need to estimate $\mu(z)$ and $\ka(z)$ from both above and below. Recall that 
the trace measure and projective measures are given by
$$
	\sig = \frac{1}{2^p p! } T(x) \wed \be^{p}, \quad
	\nu = \frac{1}{2^p\pi^{p}}T(x) \wed \al^{p}.
$$
Define the function
$$
	\sig(r) = \int_{|x|<r} d\sig(x). \quad
$$
An important result of Lelong says that 
$r\mapsto \sig(B(z,r)) / r^{2p}$ is increasing in $[0, 1]$ (see, e.g., Demailly \cite[Chapter III, Consequence~5.8.]{D-book}). Hence, 
the limit
$$
	\nu(z) = \lim_{r \to 0} \frac{1}{r^{2p}} \int_{|x -z| < r}  T(x) \wed \be^p = \lim_{r\to 0} \frac{\sig( B(z,r) }{r^{2p} \pi^{p}/p!}
$$
exists and it is called the Lelong number of $T$ at $z$. By \cite[Eq. (1.17)]{Sk72} we have
\[\label{eq:sk-eq}
	\nu(z) + \int_{0< |x-z| <r} T(x) \wed \al^{n-1} (x-z) = \frac{1}{2^p r^p}\int_{|x-z|<r} T(x) \wed \be^{n-1}.
\]
This gives for $|z| \leq r$, 
$$\begin{aligned}
	(2\pi)^{-p} \mu(z) 
&\leq (2\pi)^{-p}  \int_{|z-x| \leq r+R} \al^p(z-x) \wed T (x) \\
&\leq \pi^{-p} p! (r+ R)^{-2p} \int_{|x-z| \leq r+R} d\sig(x) \\
&\leq \pi^{-p} p! (r+ R)^{-2p} \int_{|x| \leq 2r+R} d\sig(x).
\end{aligned}
$$
In other words, 
$$\begin{aligned}
	(2\pi)^{-p} \mu(z) 
&\leq  \left(\frac{2r+R}{r+R}\right)^{2p} \frac{\sig(2r +R)}{(2r+ R)^{2p}}.
\end{aligned}$$
Thus, we get for  $|z| \leq r=\frac{1}{8}$ and $R=\frac{1}{2}$,
\[\label{eq:mu-up-bound}
	(2\pi)^{-p} \mu(z) \leq 2^{2p} \sig (3/4) =: \frac{c (1-\veps)}{n},
\]
where $0<\veps<1/2$ is fixed.

Next, we estimate $\mu(z)$ from below. If $\nu (z) =0$,  \eqref{eq:sk-eq} implies that for $|z| \leq r$,
$$\begin{aligned}
	(2\pi)^{-p} \int_{|z-x| \leq R-r} \al^p(z-x) \wed T(x) 
&= \pi^{-p} p! (R-r)^{-2p} \int_{|x-z| \leq R-r} d\sig (x) \\
&\geq \pi^{-p} p! (R-r)^{-2p} \int_{|x| \leq R-2r} d\sig (x)  \\
& = (2\pi)^{-p}  (R-r)^{-2p} \int_{|x| \leq R - 2r}  T(x) \wed \be^{p}.
\end{aligned}$$
Plugging $r=\frac{1}{8}$ and $R= \frac{1}{2}$  we get that for  $|z|\leq \frac{1}{8}$ and $\nu(z) =0$,
\[\label{eq:mu-low-bound}
\begin{aligned}
	(2\pi)^{-p} \mu(z) 
&\geq (2\pi)^{-p} (3/8)^{-2p} \int_{|x| \leq \frac{1}{4}} \be^n \\
&=  (2\pi)^{-p} (3/8)^{-2p}  2^n\pi^n (1/4)^n \\
&	=:c_1.
\end{aligned}\]
By [Sk72, Lemme~7.1]  the set $\{z: \nu(z)>0\}$ has zero measure as $U(z)$ is locally integrable (or simply by Siu's theorem \cite{Siu}). Hence, \eqref{eq:mu-low-bound} holds a.e.

Combining \eqref{eq:exp-U4}, \eqref{eq:mu-up-bound} and \eqref{eq:mu-low-bound}  we obtain
$$\begin{aligned}
&	\int_{|z| \leq \frac{1}{8}} \exp \left[ - \frac{n}{pc} U_4(z) \right] dz \\
&\leq	 \frac{1}{(2\pi)^p c_1} \iint_{|z| \leq \frac{1}{8}, |x| \leq \frac{1}{2}} |z-x|^{- 2(n-p-\veps)} \al^p(z-x) \wed T(x).
\end{aligned}
$$
Since the coefficients of $\al^p(z-x)$ are dominated by
$
	A(n,p) |z-x|^{-2p},
$
we have 
$$\begin{aligned}
	\int_{|z| \leq \frac{1}{8}} \exp \left[ - \frac{n}{pc} U_4(z) \right] dz 
&\leq	 \frac{A(n,p)}{(2\pi)^p c_1} \iint_{|z| \leq \frac{1}{8}, |x| \leq \frac{1}{2}} |z-x|^{- 2n + 2\veps} d\sig(x) dz \\
&= A_1.
\end{aligned}
$$
The proof is completed with $c_0 =  pc/n$.
\end{proof}

\end{document}